\documentclass[a4paper, 11pt]{article}
\usepackage[sc]{mathpazo}
\usepackage{verbatim}
\usepackage{amsmath,amssymb,amsfonts,mathrsfs}
\usepackage[amsmath,thmmarks]{ntheorem}

\usepackage{geometry}
\usepackage{varioref}
\usepackage[linkcolor=black,colorlinks=true,citecolor=black,filecolor=black]{hyperref}
\usepackage[nodayofweek]{datetime}
\usdate
\usepackage[skip=5pt plus1pt, indent=0pt]{parskip}
\renewcommand{\epsilon}{\ensuremath\varepsilon}

\renewcommand{\phi}{\ensuremath{\varphi}}

\newtheorem{theorem}{Theorem}[section]

\newtheorem{remark}[theorem]{Remark}
\newtheorem{corollary}[theorem]{Corollary}
\newtheorem{definition}[theorem]{Definition}
\newtheorem{lemma}[theorem]{Lemma}
\newtheorem{proposition}[theorem]{Proposition}
\newtheorem{assumption}[theorem]{Assumption}

\theoremstyle{nonumberplain}
\theorembodyfont{\normalfont}
\theoremsymbol{\ensuremath{\square}}
\newtheorem{proof}{Proof}

\title {A Probabilistic Approach to Discounted Infinite Horizon and Invariant Mean Field Games}
\author{Ren{\'e} Carmona\footnote{Operations Research and Financial Engineering, Program in Applied and Computational Mathematics, Princeton University, Partially supported by AFOSR FA9550-23-1-0324} \and Ludovic Tangpi\footnote{Operations Research and Financial Engineering, Program in Applied and Computational Mathematics, Princeton University, Partially supported by NSF CAREER award DMS-2143861 and the AMS Claytor-Gilmer fellowship} \and Kaiwen Zhang\footnote{Operations Research and Financial Engineering, Princeton University, Partially supported by AFOSR FA9550-23-1-0324 and the AMS Claytor-Gilmer fellowship}}

\begin{document}
\maketitle

\begin{abstract}
    This paper considers discounted infinite horizon mean field games by extending the probabilistic weak formulation of the game as introduced in \cite{Carmona15}. Under similar assumptions as in the finite horizon game, we prove existence and uniqueness of solutions for the extended infinite horizon game. The key idea is to construct local versions of the previously considered stable topologies. Further, we analyze how sequences of finite horizon games approximate the infinite horizon one. Under a weakened Lasry-Lions monotonicity condition, we can quantify the convergence rate of solutions for the finite horizon games to the one for the infinite horizon game using a novel stability result for mean field games. Lastly, applying our results allows to solve the invariant mean field game as well.
\end{abstract}

\tableofcontents

\section{Introduction}
Mean field games have been introduced by Lasry and Lions \cite{LasryLions, LasryLionsI, LasryLionsII} and Caines, Huang and Malhamé \cite{HuangCainesMalhame04, CainesHuangMalhame06, HuangCainesMalhame07, HuangCainesMalhame072} as a way to study large symmetric games for which  interactions between players can be captured by \emph{mean field} quantities. Since then, they have been studied under various formulations, though in most of the literature, only in \emph{finite horizon}. For a broad depiction of the existing literature on probabilistic approaches to mean field games we refer to Carmona and Delarue \cite{CarmonaDelarue, CarmonaDelarue2}. 

\subsection{Generalities on infinite horizon mean field games}
In many applications, especially in economic, mean field game models are formulated in the infinite horizon. For an example in macroeconomic, Achdou, Han, Lasry, Lions and Moll \cite{Achdou22} consider the Aiyagari consumption model, or more generally heterogeneous agent models, in the mean field framework. A discrete time version can be found in Huang \cite{Huang}, and an explicitly solved version is provided in \cite{CarmonaDelarue}. These models are fundamentally built around a discounted infinite horizon control problem. In a different application area, synchronization models such as the Kuramoto model in Carmona, Cormier and Soner \cite{Cormier} and the Jetlag recovery model in Carmona and Graves \cite{CarmonaGraves} are rephrased as infinite horizon mean field games. In these models, one is interested in finding the invariant solutions and identifying the convergence behaviors to these invariant solutions. Moreover, for the sake of stability and numerical tractability of the solutions, there is also a focus on invariant solutions. Note that the infinite horizon model can also incorporate models with unbounded random time horizons and are thus also relevant for optimal execution problems, such as e.g.\ in Tangpi and Wang \cite{Wang22, Wang23}.
\par
Our goal is to provide a general framework for solving (extended) infinite horizon mean field games. In addition to deriving well-posedness results, we will put infinite horizon games in relation with both finite horizon games, as well as the invariant mean field game. We do so using a fully probabilistic approach. In contrast, the overwhelming majority of the literature on the problem uses Partial Differential Equation (PDE) techniques, i.e. they study the related master equation, or at least the PDE system consisting of the Hamilton-Jacobi-Bellman and the Fokker-Planck equation. This way, Cardaliaguet and Porretta \cite{CardaliaguetPoretta} first considered the discounted infinite horizon mean field game and related invariant game on a torus for small discount factors using the master equation. Comparable results, as well as turnpike estimates, were also derived without the use of the master equation by Cirant and Porretta \cite{CirantPorretta} under a weaker version of the typical monotonicity condition. The master equation for the finite state discounted infinite horizon game has been considered by Cohen and Zell \cite{CohenZell}. Also under an analytical approach, Masaero \cite{Masoero} and Bensoussan, Frehse and Yam \cite{Bensoussan} solve infinite horizon potential games, and in a linear quadratic framework, Priuli \cite{Priuli} shows convergence of the infinite horizon $N$-player game to the mean field game for small discount rates. We also refer to \cite{Leong, PengZhang, Ni, Li} for other papers on the problem.
\par
Conditions needed in the PDE approach are usually restrictive. In fact, solving the master equation or the mean field PDE system requires rather stringent regularity assumptions on the coefficients. Typically, well-posedness results need some kind of \emph{monotonicity condition}. Additionally, non Markovian models cannot be covered. So far, within the probabilistic literature, general models for infinite horizon mean field games are few and far between. To our knowledge, the only paper on that problem beyond linear quadratic models, is the one of Bayraktar and Zhang \cite{Bayraktar} that proves existence of solutions for the discounted mean field control and mean field game problem using infinite horizon Forward-Backward Stochastic Differential Equation (FBSDE) and the stochastic Pontryagin's principle. To solve such FBSDEs they require, on top of the usual regularity conditions, fairly restrictive monotonicity conditions which among other things, restrict the set of admissible discount factors. Further, their solution does not cover any non Markovian models, and as usual with the stochastic maximum principle, separability assumptions and restrictions on the volatility are required. We also refer to Jackson and Tangpi \cite{TangpiJackson} for results along the same lines.

\subsection{Framework of the paper}
In this paper, we consider a general discounted infinite horizon mean field game in the non Markovian setting and with modest regularity conditions. 
We do not deal with ergodic models for which finite time behavior is completely irrelevant. More precisely, we concentrate on the so-called probabilistic weak formulation which is standard in optimal control literature, see e.g.\ Davis \cite{Davis} and El Karoui, Peng and Quenez \cite{Karoui}. In essence, this means that the controlled state dynamics is not anymore required to be satisfied as a strong, but only as a weak Stochastic Differential Equation (SDE). This is achieved by controlling the probability measure via Girsanov's Theorem. One should be aware that in a non Markovian framework, strong and weak formulation are not necessarily equivalent (see e.g.\ \cite[3.3.1]{CarmonaDelarue}). As one can also see in \cite{CarmonaDelarue}, solving the game in the strong formulation relies on much stronger regularity on the state variable. Another benefit of working with the weak formulation is that it simplifies working in a fully non Markovian framework. Of course, working with the weak formulation comes with caveats as well, as integrability requirements of the Doléans-Dade exponential restrict the growth of the drift. In the finite horizon case, mean field games in the weak formulation have been introduced and analyzed by Carmona and Lacker \cite{Carmona15}. Other instances can be found in Lacker \cite{Lacker15} where the formulation is further weakened to a controlled martingale problem formulation so that volatility control can be considered as well. We also refer to Possamaï and Tangpi \cite{PossamaiTangpi} that relate the weak formulated games to the corresponding $N$ player game. Our extension to infinite horizon is close in spirit to \cite{Carmona15}.

\subsection{Summary of the results}
The main contributions of the paper are as follows
\begin{itemize}
    \item We begin by showing well-posedness of the infinite horizon problem: Existence can be shown under fairly weak continuity assumptions, uniqueness is a consequence of the (weak) Lasry-Lions monotonicity condition.
    \item We relate the infinite horizon games to finite horizon games and find both qualitative and quantitative long time asymptotic results.
    \item We formulate an invariant mean field game within our probabilistic setting and solve it using our previously found results.
\end{itemize}
Let us stress that in contrast to the existing literature, the majority of our results do not require any restrictions on the discounting factor. We also point out that whenever possible, we are working with an "extended" mean field game setting where the mean field interaction does not only occur through the law of the state variable, but of the control variable as well, such as introduced by Gomes and Voskanyan \cite{Gomes}. This is particularly relevant for applications in finance and economics.

\subsubsection{Existence and uniqueness results}
Since we allow the mean field interaction to enter through the law of the entire path of the state variable on the whole time horizon, and not only through the law of the past of the path, an approximation argument through finite horizon solutions is not possible. Although this level of generality is rarely required in applications, it is an important feature that we need when we discuss the invariant mean field game. Our proof of existence is based on a fixed point argument in a similar spirit as in \cite{Carmona15}, but the infinite horizon problem comes with numerous new technical difficulties. The first one is already apparent when formulating our mean field game. The weak formulation is based on Girsanov's Theorem, but the latter has to be used with caution in the infinite horizon case, as singular measures can appear. This already causes problems with the usual augmentation of the filtration, which is why we propose a new filtration in Appendix \ref{apdx: filtr}. Since our mean field interaction is not of Markovian nature but pathwise, the necessity to deal with singular measures is inevitable, forcing us to find local workarounds. The biggest problem occurs in the fixed point argument as we lose the crucial compactness in the previously considered set wise topology discussed in Doob \cite{Doob94}, or as the s-topology in Balder \cite{Balder}. In this topology, compactness comes from the existence of a common dominating measure which in finite horizon is an immediate consequence of the weak formulation, but is not in the infinite horizon case. A similar problem is caused by the use of "relaxation" for flows of measures. Relaxation is a commonly used tool in stochastic optimal control and mean field games, to provide compactness when needed. See for example \cite{Carmona15}, or earlier works of Haussmann and Lepeltier \cite{HaussmannLepeltier},  El Karoui, Nguyen and Jeanblanc-Picqué \cite{KarouiNguyen}, or Lacker \cite{Lacker15, Lacker17}. This \emph{compactification} is usually performed using the stable topology described in Jacod and Memin \cite{JacodMemin}. It can also be achieved through the use of Young measures as in Florescu and Godet-Thobie \cite{Florescu}, or of the ws-topology as in Balder \cite{Balder}. Again, these constructions rely on the time horizon being finite. Facing these two topological issues, to still be able to complete the fixed point argument, we find compactness in a suitable inverse limit type topology. Working with inverse limit topologies is particularly convenient to prove other required topological properties of the spaces we consider, such as Hausdorffness, behaviour of subspace topologies, as well as metrizability. In this way, we prove existence for the extended game in Theorem \ref{thm: exst E}. For non extended games, we provide in Theorem \ref{thm: exst NE} a more general existence result using set valued analysis in the same spirit as in \cite{Carmona15}. Our proof is made more complete by the use of a different notion of upper hemicontinuity of set valued maps detailed in Appendix \ref{apdx: setvm}. A main step in these proofs is to analyze the discounted infinite horizon control problem. Instead of analyzing the infinite horizon Hamilton-Jacobi-Bellman PDE as in Fleming and Soner \cite{FlemingSoner}, our approach here is based on the use of infinite horizon backwards stochastic equations (BSDEs) with monotonous drivers as introduced in Pardoux \cite{Pardoux98}, later  refined by Royer \cite{Royer04}, and applied to infinite horizon control problems in e.g.\ Confortola, Cosso and Fuhrman \cite{Confortola} or Hu and Tessitore \cite{Hu07}. We extend some results regarding these BSDEs in Appendix \ref{apdx: IBSDE}. Our uniqueness result in Theorem \ref{thm: unq} follows as an analogue of the finite horizon result under appropriate adaptation of the usual conditions. Let us also point out that the techniques we use in these proofs also apply to weakly formulated finite horizon mean field games that allow the change of measure to be singular at the end of the time horizon.

\subsubsection{Asymptotics}
We elucidate the asymptotic behaviour of large but finite horizon mean field games in three results. First, in Theorem \ref{thm: eps opt}, we show that for a fixed finite time horizon, the infinite horizon solution is an approximate solution for the finite horizon game with an error that declines exponentially fast in the time horizon. The notion of approximate equilibrium is well established in the literature, see e.g.\ \cite{CarmonaDelarue2} and Djete \cite{Djete21}, though usually in a context linking mean field games with their N-player game counterparts. In our case, this approximating effect is due to the discounting of the control problem after a finite horizon, but also, to the properties of the infinite horizon BSDEs that we consider. Secondly, in Theorem \ref{thm: cnv inft}, we exploit our insights from the proofs of existence to derive a compactness result for solutions of different time horizons. In a very general setting, solutions of games with diverging time horizons always admit accumulation points that will necessarily be a solution to the infinite horizon game, in a similar spirit as the $N$ player convergence results in Lacker \cite{Lacker17}. This result can be interpreted as a construction scheme for infinite horizon solutions out of finite horizon ones, as it has been done in practice e.g.\ in \cite{CardaliaguetPoretta, Cormier}. As a next step, in Theorem \ref{thm: cnv rate}, we quantify this convergence under a weak Lasry-Lions monotonicity condition that is similar to \cite{CirantPorretta}. Under additional conditions, weak monotonicity can not only provide a second uniqueness result, but in the finite horizon case, it allows us to control the stability of mean field game solutions with respect to small perturbations of the terminal reward. Applying this idea to our convergence analysis, we derive an exponential rate that is reminiscent of the turnpike phenomenon e.g.\ in \cite{CardaliaguetPoretta, CirantPorretta}. These approximation results have high potential for numerical algorithms, as they show that the infinite horizon game can be approximately solved by finite horizon games. And on the other hand, when the infinite horizon game is easily solvable, as it is e.g.\ in the invariant case, it can also speed-up and enhance numerical calculations for large finite horizon games, as demonstrated in Carmona and Zeng \cite{CarmonaZeng}.

\subsubsection{Invariant games}
In this context, the interest for solution of the invariant game naturally arises. We formulate the invariant version of the infinite horizon discounted mean field game as finding a solution in which the marginal distributions of the state variable stay invariant. In contrast to the previous framework, the initial distribution of the state variable is not held fixed but becomes part of the solution. To approach this problem we first introduce a novel, weaker form of invariant mean field game which we call stationary mean field game. In such, the state variable starts in a fixed initial distribution, and the mean field interaction term is assumed to be a function of the large time limit of the marginal distributions of the state variable. To ensure that this limit exists, we enforce a drift condition on the controlled state dynamics that forces convergence towards a stationary measure. Due to its general nature, we have decided to adapt the drift condition of Khasminskii and Milstein \cite{Khasminskii}. Similar drift conditions used in mean field game literature can be found e.g.\ in Lacker, Shkolnikov and Zhang \cite{LackerShkolnikov} and Porretta and Riciardi \cite{PorrettaRicciardi}. Under such conditions, in Theorem \ref{thm: ces cnv}, we quantify this convergence without the need of any regularity on the drift. After ensuring that this stationary mean field game is well defined, we show that it can be embedded into our previously analyzed mean field game framework so that its solvability follows immediately. Back to the invariant game, it turns out that the solvability of the invariant mean field game and the stationary mean field game are closely related. This allows us to solve the invariant mean field game in Theorem \ref{thm: inv mfg exst}. Apart from existence, although convergence of infinite horizon solutions to the invariant solution have been observed in practice, see e.g.\ \cite{CardaliaguetPoretta, Cormier} and Höfer and Soner \cite{Hofer}, in our general framework, such convergence behavior still remains to be discussed.
\par
We conclude this introduction by delineating the structure of this paper. In Section \ref{sec: prelim}, we state the framework and our results in detail. In Section \ref{sec: prf}, we present the proofs of the main existence and uniqueness results. Some topological constructions that we rely on have been relegated to Section \ref{sect: topo}. Section \ref{sec: asym} contains the proofs of our long time asymptotic results. 
Lastly, Section \ref{sect: StMFG} deals with the invariant mean field game. Some intermediate results are provided in the appendices to keep the paper self-contained.

\section{Preliminaries and main results}
\label{sec: prelim}

\subsection{Notation}

Let us gather some frequently used notation in this section:
For any Polish space $E$, we denote by $\mathcal{P}(E)$ the space of all Borel probability measures on its Borel $\sigma$ algebra that we write as $\mathcal{B}(E)$. For $x \in E$ and $r>0$, we denote with $B_x(r)$ the open ball around $x$ with radius $r$. Let $GL_d(\mathbb{R})$ be the space of real valued invertible $d \times d$ matrices. With $id_E$, we typically denote the identity map on any set $E$. For any two probability measures $\mu$ and $\nu$, we denote the relative entropy $\mathcal{H}$ as $\mathcal{H}(\mu, \nu) = \int \log\left(\frac{d\mu}{d\nu}\right) d\mu$ if $\mu \ll \nu$ and otherwise $\mathcal{H}(\mu, \nu) = \infty$. We write the total variation distance as $d_{TV}(\mu, \nu)$.\par
We write $\mathcal{C}$ and $\mathcal{C}^T$ for the space of continuous $\mathbb{R}^d$ valued functions on $[0, \infty)$ with $\mathcal{C}$ being equipped with the metric of uniform convergence on compacts
$$d(\omega^1, \omega^2) = \sum_{n \geq 1} 2^{-n} \sup_{t \in [0, T]} \max(\vert \omega^1_t - \omega^2_t\vert, 1)$$
and $\mathcal{C}^T$ with the supremum metric. With $\mathcal{C}_0$ and $\mathcal{C}^T_0$ we denote the respective subspaces of functions starting at zero.\par
Further, for any measure space $(\Omega, \mathcal{F})$, we call the topology of setwise convergence the coarsest topology such that for any $A \in \mathcal{F}$, the map $\mu \mapsto \mu(A)$ is continuous. Note that by \cite[Proposition 2.4]{JacodMemin} applied to the special case of the metric space being a singleton, the setwise topology can also be characterized as the coarsest topology such that the maps $\mu \mapsto \int_\Omega \phi d\mu$ are continuous for any bounded measurable real valued $\phi$. Note that this topology is weaker than the one induced by the total variation metric. If $\Omega$ is a Polish space with Borel $\sigma$ algebra $\mathcal{F}$, then the setwise topology is stronger than the weak topology.

\subsection{Probabilistic setup}
We work on the canonical probability space $(\Omega, \mathcal{F}^0, \mathbb{P})$ of a $d$ dimensional infinite time horizon Wiener process $(W_t)_{t \geq 0}$ together with an initial distribution $\upsilon \in \mathcal{P}(\mathbb{R}^d)$. That is, $\Omega := \mathbb{R}^d \times \mathcal{C}_0$. Under $\mathbb{P}$, we assume that the canonical element $\zeta = (\xi, \omega)$ has distribution $\upsilon \times \mathbb{P}_W$ for the Wiener measure $\mathbb{P}_W$. $\mathcal{F}^0 := \mathcal{B}(\Omega)$ is given by the Borel $\sigma$-algebra, and we consider the natural filtration $\mathbb{F}^0 = (\mathcal{F}^0_t)_{t \geq 0} = (\sigma(\xi, \omega_{\vert [0, t]}))_{t\geq 0}$.\par
We would like to work with a complete filtration, but in the current infinite horizon setup, the usual completed filtration will be too big, as it cannot support possibly singular measures.
Let us therefore introduce a \emph{locally completed filtration}. 
To define it, we first introduce $\mathfrak{E}$ to be the collection of probability measures $\mathbb{Q}$ such that for any $T >0$, the restriction of $\mathbb{Q}$ on $\mathcal{F}^0_T$ is equivalent to that of $\mathbb{P}$, i.e.\ $\mathbb{Q}_{\vert \mathcal{F}^0_T} \sim \mathbb{P}_{\vert \mathcal{F}^0_T}$. Such measures have been discussed in \cite{JS}, and we will also denote their defining property by $\mathbb{Q} \overset{loc}{\sim} \mathbb{P}$. For any $\mathbb{Q} \in \mathfrak{E}$, let $\mathcal{N}^\mathbb{Q}$ be the set of all $\mathbb{Q}$-negligible subsets in $\mathcal{F}^0$ of $\Omega$, that is all $A \subset \Omega$ such that there is $B \in \mathcal{F}^0$ with $A \subset B$ and $\mathbb{Q}[B] = 0$. For each $\mathbb{Q} \in \mathfrak{E}$, we can complete as usual via $\mathbb{F}^\mathbb{Q} = (\mathcal{F}^\mathbb{Q}_t)_{t \geq 0} := (\sigma(\mathcal{F}^0_t \cup \mathcal{N}^\mathbb{Q}))_{t \geq 0}$ and $\mathcal{F}^\mathbb{Q} := \sigma(\mathcal{F}^0 \cup \mathcal{N}^\mathbb{Q})$ so that $\mathbb{Q}$ can be naturally extended to a measure on $\mathcal{F}^\mathbb{Q}$. 
We then define the locally completed filtration as 
$$\mathbb{F} = (\mathcal{F}_t)_{t \geq 0} := (\cap_{\mathbb{Q} \in \mathfrak{E}} \mathcal{F}^\mathbb{Q}_t)_{t \geq 0}\quad \text{and}\quad \mathcal{F}:=\cap_{\mathbb{Q} \in \mathfrak{E}} \mathcal{F}^\mathbb{Q}.$$ 
Let us highlight some properties of the locally completed filtration just defined:
\begin{itemize}
    \item On $\mathcal{F}_T$ (rather than $\mathcal{F}^0_T$), elements of $\mathfrak{E}$ are locally equivalent, see Proposition \ref{Prop.A1}.
    \item $\mathbb{F}$ is right continuous and contains all negligible sets up to finite time horizons.
    In particular, on any finite time horizon, $\mathbb{F}$ satisfies the usual conditions, see Proposition \ref{Prop.FF.right.cont}.
    \item $\mathbb{F}$ satisfies the predictable representation property, see Proposition \ref{prop: prp}.
\end{itemize}
Throughout this paper, when working on some completion (or local completion) of the sigma algebra of a Polish space $E$, by abuse of notation we will identify elements of $\mathcal{P}(E)$ with their unique extension on the completed sigma algebra.

\subsection{Infinite horizon mean field games}
Let us describe the mean field game that we are interested in.
We start by fixing an uncontrolled state process 
\begin{equation}\label{FSDE}
    dX_s = \sigma(s, X) dW_s\;\; X_0 = \xi.
\end{equation}
where $\xi$ is $\mathcal{F}_0$ measurable with law $\upsilon$ and $\sigma: [0, \infty) \times \mathcal{C} \rightarrow GL_d(\mathbb{R})$ is a progressively measurable volatility function. We assume throughout that \eqref{FSDE} admits a unique strong solution.
We are working with an action space $(A, \Vert \cdot \Vert_A)$ that we assume to be a compact subset of a normed vector space.
Let us further define the set $\mathbb{A}$ of admissible controls consisting of all $\mathbb{F}$-progressively measurable $A$ valued processes. We also define $\mathbb{A}^T$ as the space of admissible control processes on the finite time horizon $[0, T]$. As $A$ is bounded, on these spaces, we can use the norm $\Vert \alpha \Vert_{\mathbb{A}^T} := \mathbb{E}\left[ \int_0^T \Vert \alpha_t \Vert_A dt \right]$ on $\mathbb{A}^T$. On $\mathbb{A}$, we use the norm
$$\Vert \alpha \Vert_\mathbb{A} := \mathbb{E}\left[\int_0^\infty e^{-t} \Vert \alpha_t \Vert_A dt\right].$$
Note that the factor $e^{-t}$ does not play an important role as we could also have used any other metric that metrizes the same topology.\par
Now, for any $(\alpha,\mu)\in \mathbb{A}\times \mathcal{P}(\mathcal{C})$, we define $\mathbb{P}^{\mu, \alpha}$ as the unique measure on $(\Omega, \mathcal{B}(\Omega))$ such that for any $T>0$ and any $B \in \mathcal{F}_T$
$$\mathbb{P}^{\mu, \alpha}\left[ B \right]:= \mathbb{E} \left[ \mathbb{1}_B \mathcal{E}\left( \int_0^\cdot \sigma(s, X)^{-1} b(s, X, \mu, \alpha_s ) dW_s \right)_T \right].$$
Under our standing assumptions \ref{asmp: stnd} blow, existence and uniqueness of such a measure is shown in \cite[Theorem VIII.1.13]{RY}. Here, $\mathcal{E}(M)_T:=\exp(M_T - \frac12\langle M\rangle_T)$ is the Doléans-Dade exponential of the martingale $M$, and $b: [0, \infty) \times \mathcal{C} \times \mathcal{P}(\mathcal{C}) \times A \rightarrow \mathbb{R}^d$ is a progressively measurable drift.  
That is, for any $\mu \in \mathcal{P}(\mathcal{C})$ and $a\in A$, the process $b(\cdot, X, \mu, a)$ is $\mathbb{F}$-progressively measurable. In particular, this implies that $b$ is not anticipative, i.e.\ for any $t \geq 0$ and $x, x'\in \mathcal{C}$ with $x_{\vert [0, t]} = x'_{\vert [0, t]}$, we have $b(t, x, \mu, a)=b(t, x', \mu, a)$.
The reason why we defined the change of measure this way is that on $\mathcal{F}$, the measures $\mathbb{P}^{\mu, \alpha}$ and $\mathbb{P}$ are generally singular.
However, under our assumptions on $b$ and $\sigma$ given below, the measures are locally equivalent with $\mathbb{P}^{\mu, \alpha} \in \mathfrak{E}$ and we have the usual characterization
$$\frac{d\mathbb{P}^{\mu, \alpha}_{\vert \mathcal{F}^0_T}}{d\mathbb{P}_{\vert \mathcal{F}^0_T}} = \mathcal{E}\left( \int_0^\cdot \sigma(s, X)^{-1} b(s, X, \mu, \alpha_s) dW_s \right)_T.$$
Moreover, by Girsanov's Theorem on infinite time horizon \cite[Corollary 3.5.2]{KS}, we know that the process $W^{\mu, \alpha} := W - \int_0^\cdot \sigma(s, X)^{-1}b(s, X, \mu, \alpha_s)ds$ is a $(\mathbb{P}^{\mu, \alpha}, \mathbb{F}^0)$ Wiener process and through completion, it is a $(\mathbb{P}^{\mu, \alpha}, \mathbb{F}^{\mathbb{P}^{\mu, \alpha}})$ Wiener process as well. Note that we always have $\mathbb{F}^0 \subset \mathbb{F} \subset \mathbb{F}^{\mathbb{P}^{\mu, \alpha}}$. Thus, $W^{\mu, \alpha}$ is a $(\mathbb{P}^{\mu, \alpha}, \mathbb{F})$ Wiener process and under $\mathbb{P}^{\mu,\alpha}$, $X$ satisfies the \emph{controlled SDE}
\begin{equation*}
    dX_s = b(s, X, \mu, \alpha_s)ds + \sigma(s,X)dW^{\mu,\alpha}_s, \quad \mathbb{P}^{\mu,\alpha}\text{--a.s.,}\;\; X_0 = \xi.
\end{equation*}
Further, under $\mathbb{P}^{\mu, \alpha}$, the process $W^{\mu, \alpha}$ stays independent of $\xi$ so that $X$ is a weak solution to the SDE with drift.\par
Let the running reward be given by a progressively measurable function $f: [0, \infty) \times \mathcal{C} \times \mathcal{P}(\mathcal{C})\times \mathcal{P}(A) \times A \rightarrow \mathbb{R}$ and fix a discounting factor $\lambda>0$.
Given any $\mu \in \mathcal{P}(\mathcal{C})$, a Borel measurable flow of measures $q: [0, \infty) \rightarrow \mathcal{P}(A)$, we define the discounted reward
$$J^{\mu, q}(\alpha) := \mathbb{E}^{\mu, \alpha}\left[ \int_0^\infty e^{-\lambda t} f(t, X, \mu, q_t, \alpha_t) dt \right]$$
where here and henceforth, we use the notation $\mathbb{E}^{\mu,\alpha}$ to denote the expectation under $\mathbb{P}^{\mu,\alpha}$ and we define
$$V^{\mu, q} := \sup_{\alpha \in \mathbb{A}} J^{\mu, q}(\alpha).$$\par
Intuitively, the probability measure $\mu\in \mathcal{P}(\mathcal{C})$ represents the distribution of the (whole path of the) state of the population, while $q_t$ represents the distribution of the control of the population at time $t$.
Considering the law of the control process puts us in the setting of extended mean field games.
A solution of the problem which we call a mean field equilibrium is defined as follows:
\begin{definition}
\label{def.MFE}
We call a tuple $(\mu, q,\alpha)$ a solution to the discounted infinite horizon (extended) mean-field game, if $\alpha \in \mathbb{A}$ satisfies $V^{\mu, q} = J^{\mu, q}(\alpha)$ and 
\begin{equation}
\label{eq:MFG}
    \mu = \mathbb{P}^{\mu, \alpha} \circ X^{-1}\quad \text{and}\quad q_t = \mathbb{P}^{\mu, \alpha} \circ \alpha_t^{-1}\quad \text{for almost every } t\ge0.
\end{equation}
\end{definition}
To ensure that our problem is well-defined, unless otherwise specified, we will always make the following standing assumptions throughout the paper:
\begin{assumption}[Standing Assumptions]\label{asmp: stnd}
\begin{itemize}
    \item[(i)]The process $\Vert \sigma^{-1}b \Vert$ is uniformly bounded by some $C>0$.
    \item[(ii)]For $\mu$, $q$ and $a$, the process $\vert f(\cdot, X, \mu, q, a)\vert$ is uniformly bounded by some $M < \infty$.
    \item[(iii)]$b$ and $f$ are continuous in $a$.
    \item[(iv)] There exist suitable functions $f_1$ and $f_2$ such that $f(t, x, \mu, q, a)= f_1(t, x, \mu, a) +f_2(t, x, \mu, q)$.    
\end{itemize}
\end{assumption}
The first assumption ensures that our previously considered Doléans-Dade exponential is always well-defined, and will also be important later to prove compactness properties needed for our fixed point arguments. The second condition ensures that our running reward is always finite. The continuity assumptions guarantees that the Hamiltonian we will define later always has a non empty set of maximizers. The last condition is fairly standard in notably the extended mean field game literature.
In some cases, we will focus on the standard case where the law of the control $q$ is not considered to derive stronger results.
In such cases, we will assume $f_2=0$, and mean field equilibria are defined exactly as in Definition \ref{def.MFE}, except that Equation \eqref{eq:MFG} simplifies to
\begin{equation*}
    \mu = \mathbb{P}^{\mu,\alpha}\circ X^{-1}.
\end{equation*}
\subsection{Main results}
We are now ready to state the main results of the paper.
The first set of results concern well-posedness of the problem.
We will give conditions guaranteeing existence and uniqueness of the discounted infinite horizon mean field game.
After that, the second set of results will be on large time asymptotics of mean field games, as well as stationary games.
\subsubsection{Existence and uniqueness}
\label{sec:exist.unique}
The existence result will require suitable regularity conditions on the coefficients.
This prompts the question of the choice of the right topologies for the measure-valued arguments. 
For the law of the control argument $q \in \mathcal{P}(A)$, we choose to work with the topology of weak convergence which, as $A$ is compact, makes $\mathcal{P}(A)$ a compact space as well. 
For the measure of the state variable $\mu \in \mathcal{P}(\mathcal{C})$, we define a topology as follows:
\begin{definition}\label{def: loc stw}
    Let $B(\mathcal{C})$ be the space of bounded measurable real valued functions $\phi$ on $\mathcal{C}$ for which there exists some $T>0$ so that $\phi(\omega)$ only depends on the finite horizon restriction $\omega_{\vert [0, T]}$. We define the local topology of setwise convergence which we also denote by $\tau$ to be the coarsest topology such that the maps $\mu \mapsto \int_{\Omega} \phi d\mu$ are continuous for each $ \phi \in B(\mathcal{C})$.
\end{definition}
This can be seen as a local convergence variation of the topology used in \cite{Carmona15}. For a thorough discussion of the topology $\tau$ and its properties that will be used later in the proofs, we refer to Section \ref{sect: topo} below.\par
Here, let us briefly point out how this topology behaves locally in time. As an illustrative example, we consider a function $G: [0, \infty) \times \mathcal{P}(\mathcal{C}) \rightarrow \mathbb{R}$ that can be written in the form $G(t, \mu) = g(t, \mu_t)$ for $g: [0, \infty) \times \mathcal{P}(\mathbb{R}^d) \rightarrow \mathbb{R}$ where $\mu_t$ is the marginal law of the canonical process $\omega$ at time $t$ under $\mu$. As $\mathcal{C} \rightarrow \mathbb{R}^d, \omega \mapsto \omega_t$ is continuous and thus measurable with respect to the respective $\sigma$ algebras, and since $\omega_t$ only depends on $\omega_{[0, t]}$, we can see that $\mathcal{P}(\mathcal{C}) \rightarrow \mathcal{P}(\mathbb{R}^d), \mu \mapsto \mu_t$ is continuous with respect to $\tau$ and the topology of setwise convergence on $\mathcal{P}(\mathbb{R}^d)$. This way, we can find that $G$ is continuous in $\mu$ iff $g$ is continuous in $\mu$. This example demonstrates that it matches the intuition to restrict $B(\Omega)$ to $\phi$ that only depend on a finite horizon restriction of the canonical process in our definition above.\par
\begin{assumption}[Continuity]\label{asmp: cont}
    For every $t\ge0$, 
    the maps $(\mu, a) \mapsto b(t, x, \mu, a)$ and $(\mu, q, a) \mapsto f(t, x, \mu, q, a)$ are jointly sequentially continuous.
\end{assumption}
Let us define the Hamiltonian
$$\tilde{h}(t, x, \mu, q, z, a):= f(t, x, \mu, q, a) + z^\top \sigma(t,x)^{-1}b(t, x, \mu, a)$$
and the set $\mathcal{A}(t, x, \mu, q, z) := \arg\max_{a \in A} \tilde{h}(t, x, \mu, q, z, a)$ of its maximizers. 
Note that as we have assumed continuity of $b,f$ and compactness and $A$, the set of maximizers is always non empty. Now, in the case of extended mean-field games, we additionally assume the following.
\begin{assumption}[Extended Mean-Field Game]\label{asmp: sngl}
    For any $t, x, \mu, z$, the set $\mathcal{A}(t, x, \mu, z)$ consists of a single element.    
\end{assumption}
The reason why this stronger condition is needed in the extended case is that showing compactness for the flows of measures in the control variable is a more delicate issue when set valued maps are used. The first main result of this work is the following:
\begin{theorem}\label{thm: exst E}
    Under the assumptions \ref{asmp: stnd}; \ref{asmp: cont}, and \ref{asmp: sngl}, there exists a solution to the discounted infinite horizon extended mean field game.
\end{theorem}
In the case of non-extended mean field game (i.e.\ $f_2=0$) we can relax the conditions imposed on the Hamiltonian.
In fact, in the non-extended case, we do not need to assume that the maximizer of the Hamiltonian is unique. Instead, we replace that assumption by a convexity condition.
\begin{assumption}[Non Extended Mean  Field Game]\label{asmp: NEMFG}
    \begin{itemize}
        \item[(i)]For any $t, x, \mu, z$, the set of optimal drifts defined by $B(t, x, \mu, z):= b(t, x, \mu, \mathcal{A}(t, x, \mu, z))$ is convex.
        \item[(ii)]The mean field game is independent of $q$, i.e.\ we assume $f_2 = 0$.
    \end{itemize}
\end{assumption}
The convexity assumption imposed in Assumption \ref{asmp: NEMFG}$(i)$ plays the role of \cite[Assumption (C)]{Carmona15} on the convexity of the set $\mathcal{A}(t, x, \mu, z)$.
In fact, here we use the fixed point Theorem in a different way than in \cite{Carmona15} which leads to an adjusted convexity assumption.
Observe that this condition is satisfied for instance when the running reward $f$ is pointwise quasiconcave in $a$ and the set of all possible drifts $b(t, x, \mu, A)$ is a convex subset of $\mathbb{R}^d$. In particular, in case $A$ is convex, $f$ is quasiconcave and $b$ is affine, both conditions are equivalent.
Moreover, our convexity assumption is closely related to the convexity assumption guaranteeing existence of optimal controls in \cite[(3.4)]{HaussmannLepeltier}. 
Therein, convexity is also demanded in terms of the drift, not the control variable itself. 
For comparison, one can check that our convexity assumption is strictly weaker than that of \cite{HaussmannLepeltier}.
In fact, in contrast to their assumption, we do not need any convexity on $f$. 
Under this set of assumptions we obtain our second existence result.
\begin{theorem}\label{thm: exst NE}
    Under the assumptions \ref{asmp: stnd}; \ref{asmp: cont} and  \ref{asmp: NEMFG}, there exists a solution to the discounted infinite horizon mean field game.
\end{theorem}


It is well-known that uniqueness of  mean field equilibria is obtained under stronger (structural) conditions on the coefficients of the game.
A common such condition is the Lasry-Lions monotonicity discussed in \cite{LasryLions}. 
We will extend this condition to our setup.
\begin{assumption}[Uniqueness]\label{asmp: unq}
    \begin{itemize}
         \item[(i)]For any $t, x, \mu, z$, the set $\mathcal{A}(t,x, \mu, z)$, consists of a single element.
        \item[(ii)]$b$ is independent of $\mu$.
        \item[(iii)]$f$ can be separated into the form $f(t, x, \mu, q, a) = f_1(t, x, \mu) + f_2(t, \mu, q) + f_3(t, x, a)$ for bounded functions $f_1$, $f_2$ and $f_3$.
        \item[(iv)]For all $\mu, \mu' \in \mathcal{P}(\mathcal{C})$, we have $$\int_{\mathcal{C}} \int_0^\infty e^{-\lambda s}(f_1(s, x, \mu) - f_1(s, x, \mu'))ds (\mu - \mu')(dx) \leq 0.$$
    \end{itemize}
\end{assumption}
\begin{theorem}\label{thm: unq}
    Under assumptions \ref{asmp: stnd} and \ref{asmp: unq}, there exists at most one solution to the discounted infinite horizon extended mean field game.
\end{theorem}


\subsubsection{Links between finite and infinite horizon mean field games}\label{sct: LinkMFG}

After discussing existence and uniqueness, let us relate the infinite horizon mean field games to mean field games with large, but finite horizons.
We will do so in two ways.
We will first show that infinite horizon mean field equilibria give rise to what we call ``approximate mean field equilibria'' for finite horizon games.
Reciprocally, every sequence of solutions for the finite horizon mean field equilibria (indexed by the time horizon) converges to an infinite horizon mean field equilibrium.
As mentioned in the introduction, such convergence results seem not to have appeared in the probability theory literature on mean field games.\par
We consider these finite horizon games on the following collection of spaces. For any finite time horizon $T>0$ let us define the space $\Omega^T := \mathbb{R}^d \times \mathcal{C}_0^T$. We denote the product Borel $\sigma$-algebra with $\mathcal{F}^T$. Further, we consider the measure $\mathbb{P}^T := \upsilon \times \mathbb{P}^T_W$ where $\mathbb{P}^T_W$ is the unique Wiener measure on $\mathcal{C}_0^T$. We denote the canonical element by $\zeta = (\xi, \omega)$. Let $\mathbb{F}^T = (\mathcal{F}^T_t)_{t \in [0, T]}$ be the $\mathbb{P}^T$ completion of the filtration $(\sigma(\xi, \omega\vert s \leq t))_{t\in [0, T]}$.\par
To put the infinite horizon game in relation to the finite horizon ones, we define the maps 
$$\pi^T: \Omega \rightarrow \Omega^T,\, (\xi, \omega) \mapsto (\xi, \omega_{\vert [0, T]})$$ 
which induce the pushforward maps 
$$\tilde{\pi}^T: \mathcal{P}(\Omega) \rightarrow \mathcal{P}(\Omega^T),\, \mathbb{Q} \mapsto \mathbb{Q} \circ (\pi^T)^{-1}$$ 
on the spaces of measures.
Analogously, for $t \leq T$, we define $\pi^{T, t}: \Omega^T \rightarrow \Omega^t$ and $\tilde{\pi}^{T, t}:\mathcal{P}(\Omega^T) \rightarrow \mathcal{P}(\Omega^t)$.\par
To enable us to construct finite horizon games out of an infinite horizon one, we assume the following consistency assumption.
\begin{assumption}\label{asmp: cons}
    For each $t, x, q, a$, the maps $\mu \mapsto b(t, x,\mu, a)$ and $\mu \mapsto f(t, x, \mu, q, a)$ depend only on $\tilde{\pi}^t(\mu)$, i.e.\ for any $\mu, \mu'$ with $\tilde{\pi}^t(\mu) = \tilde{\pi}^t(\mu')$, the evaluations of $f$ and $b$ must coincide.
\end{assumption}
This is a reasonable assumption as it is intuitive that the interaction occurs only with past distributions.
In particular, in the classical setting where interaction happens only through the current state $\mu_t$, our assumption is fulfilled. 
Under slight abuse of notation, we will let $f(t,x, \cdot,q, a)$ and $b(t, x, \cdot, a)$ take arguments from $\mathcal{P}(\mathcal{C}^T)$ for any $T \geq t$ which is still well defined by the consistency assumption.\par

Let us now formulate the finite horizon mean field game reduction on a fixed horizon $T>0$. For a given measure $(\mu, q)$, we can consider the optimization problem 
\begin{equation}\label{fnt MFG}
    V^{\mu, q, T} = \sup_{\alpha \in \mathbb{A}^T} J^{\mu, q, T}(\alpha)\quad \text{with}\quad J^{\mu, q, T}(\alpha) := \mathbb{E}^{\mu, \alpha, T}\left[\int_0^T e^{-\lambda s} f(s, X, \mu, q_s, \alpha_s)ds\right]
\end{equation}
where the expectation is taken with respect to the equivalent measure $\mathbb{P}^{\mu, \alpha} \in \mathcal{P}(\Omega^T)$, defined via $\frac{d\mathbb{P}^{\mu, \alpha, T}}{d\mathbb{P}^T} = \mathcal{E}(\int_0^\cdot \sigma^{-1}b(s, X, \mu, \alpha_s) dW_s)_T$.
Note that in this definition, the terminal reward is $g=0$. 
Mean field equilibria are defined just as in the infinite horizon case.

We now define approximate mean field equililbria or more precisely $\epsilon$-mean field equilibria.
\begin{definition}
    We call $(\mu, q,\alpha)$ an $\epsilon$-mean field equilibrium for the game with horizon $T>0$ if  $\mu = \mathbb{P}^{\mu, \alpha} \circ (X_{\vert[0, T]})^{-1}$ and $q_t = \mathbb{P}^{\mu, \alpha} \circ \alpha_t^{-1}$ for almost every $t \in [0, T]$, and it holds
    $$
    J^{\mu, q, T}(\alpha) \geq V^{\mu, q, T} - \epsilon.$$
\end{definition}
Of course, a solution to the mean field game is a $0$-mean field equilibrium.
Our first asymptotic result allows to construct $\epsilon$-mean field equilibria for any finite horizon game via infinite horizon equilibria.
This Theorem can be seen as an inverse convergence theorem.
\begin{theorem}\label{thm: eps opt}
    Let $(\mu, q, \alpha)$ be a mean field equilibrium for the infinite horizon discounted mean field game.
    If assumptions \ref{asmp: stnd} and \ref{asmp: cons} are satisfied, then for every $T>0$, the tuple
    $(\mu^T, q^T,\alpha^T) := (\tilde{\pi}^T(\mu), q_{\vert [0, T]}, \alpha_{\vert [0, T]})$ is a $\frac{2M}{\lambda}e^{-\lambda T}$-mean field equilibrium.
\end{theorem}
Recall that $M$ was a bound for the non discounted running reward as introduced in Assumption \ref{asmp: stnd}.
\begin{remark}
    Observe that one can 
    revert the above result to start from 
    finite horizon equilibria $(\mu^T, q^T,\alpha^T)$.
    If there exists an infinite horizon continuation $(\mu, q)$ and $\alpha$ such that $\tilde{\pi}^T(\mu) = \mu^T$, $q_{\vert [0, T]} = q^T$ and $\alpha_{\vert [0, T]} = \alpha^T$ with $\mu = \mathbb{P}^{\mu, \alpha} \circ X^{-1}$ and $q_t = \mathbb{P}^{\mu, \alpha} \circ \alpha_t^{-1}$ (which exists under the previous existence assumptions), then $(\mu, q)$ is an $\frac{2M}{\lambda}e^{-\lambda T}$-mean field equilibrium in the same sense as above.
\end{remark}

The second question we discuss is how solutions of infinite horizon mean field games can be approximated by solutions of finite horizon mean field games. 
Here, for a sequence of finite horizons $T_1 < T_2 < \cdots$ that diverge to $\infty$, we want to use the respective solutions of the finite horizon mean field game to approximate a solution of the infinite horizon game.
To relate solution $(\mu^n, q^n, \alpha^n)$ of the game on horizon $T_n$ to infinite horizon ones, we first extend them onto infinite horizon. To do so, we define $\overline{q}^n := q^n_{\cdot \wedge T_n}$ and $\overline{\alpha}^n := \alpha_{\cdot \wedge T_n}$. To extend $\mu^n$, note that because of the fixed point property, we have $\mu^n = \mathbb{P}^{\mu^n, \alpha^n, T}$. By considering $\frac{d\mathbb{P}^{\mu, \alpha^T, T}}{d\mathbb{P}^T}$ as a $\mathcal{F}_T$ measurable random variable on $\Omega$, we can define the extended measure $\mathbb{P}^{\mu^n, \alpha^n} \in \mathcal{P}(\Omega)$ via $\frac{d\mathbb{P}^{\mu^n, \alpha^n}}{d\mathbb{P}} = \frac{d\mathbb{P}^{\mu, \alpha^T, T}}{d\mathbb{P}^T}$ and define $\overline{\mu}^n = \mathbb{P}^{\mu^n, \alpha^n} \circ X^{-1}$. Then, $\tilde{\pi}^{T_n}(\overline{\mu}^n) = \mathbb{P}^{\mu^n, \alpha^n} \circ (X_{T_n})^{-1} = \mu^n$.
We will show that the sequence $(\overline{\mu}^n, \overline{q}^n, \overline{\alpha}^n)$ belongs to a compact set and that there must exist a converging subsequence. For the space of $\mathcal{P}(A)$ valued flows of measures, we consider the topology of $dt$ a.e.\ convergence in the weak topology on $\mathcal{P}(A)$.

It is thus reasonable to think of these subsequential limits as solutions of the infinite horizon problem.
In fact, 
when the infinite horizon mean field equilibrium is unique, there is even no need to pass to a subsequence.\par

\begin{theorem}\label{thm: cnv inft}
    Let the assumptions \ref{asmp: stnd}; \ref{asmp: cont}, \ref{asmp: sngl} and \ref{asmp: cons} be satisfied.
    For any diverging sequence of time horizons $0<T_1<T_2<\cdots$, for each $n$ the mean field game with finite horizon $T_n$ admits a mean field equilibrium $(\mu^n, q^n,\alpha^n)$.
    The sequence of measures $(\overline{\mu}^n, \overline{q}^n)$ is compact. Further, for any accumulation point $(\mu,q)$ of $(\overline \mu^n,\overline q^n)$ there is an optimal control $\alpha$ such that $(\mu, q,\alpha)$ is a mean field equilibrium for the infinite horizon discounted mean field game.
    Moreover, for a converging subsequence $(\overline{\mu}^n, \overline{q}^n) \rightarrow (\mu, q)$, we have
    $$\lim_{n \rightarrow \infty} V^{\mu^n, q^n,T_n} = V^{\mu, q}, \; \; \Vert \overline{\alpha}^n - \alpha \Vert_\mathbb{A} \rightarrow 0.$$
    If the infinite horizon game admits a unique mean field equilibrium, then the full sequence $(\bar\mu^n, \bar q^n)$ converges to the mean field equilibrium for the infinite horizon problem.
\end{theorem}
The previous qualitative convergence result is stated under fairly general assumptions. 
In particular, uniqueness is not required. 

In many applications, it is important to quantify the convergence rate to the infinite horizon problem.
It turns out that the Lasry-Lions monotonicity conditions, in combination with fairly mild additional concavity and Lipschitz assumptions, is the proper structural condition allowing to provide a convergence rate.
It is well known that, in addition to guaranteeing uniqueness, this monotonicity condition further admits a stabalizing effect. This observarion is notably exploited by Cardaliguet, Delarue, Lasry and Lions \cite{cardaliaguet19} to solve the convergence problem in mean field games (in finite horizon). 
Further, it is possible to weaken the Lasry-Lions monotonicity condition in a sense that we discuss below.\par
\begin{assumption}[Convergence rates]\label{asmp: cnv rate}
    \begin{itemize}
        \item[(i)]$b$ is independent of $\mu$. 
        \item[(ii)]$f$ can be separated into the form $f(t, x, \mu, q, a) = f_1(t, x, \mu) + f_2(t, \mu, q) + f_3(t, x, a)$ for bounded functions $f_1$, $f_2$ and $f_3$.
        \item[(iii)]The redefined Hamiltonian $h(t, x, z, a) = e^{-\lambda t} f_3(t, x, a) + z^\top \sigma^{-1}b(t, x, a)$ at time $t$ is $e^{-\lambda t}m$-strongly concave in $a$ for some $m > 0$, i.e.\ for any $ x, z$, i.e.\ $a \mapsto h(t, x, z, a) + \frac{e^{-\lambda t} m}{2}\Vert a \Vert_A^2$ is concave.
        \item[(iv)]For all $t \geq 0$ and $\mu, \mu' \in \mathcal{P}(\mathcal{C})$, we have
        $$\int_\mathcal{C} (f_1(t, x, \mu) - f_1(t, x, \mu'))(\mu - \mu')(dx) \leq \delta (\mathcal{H}(\tilde{\pi}^t(\mu), \tilde{\pi}^t(\mu'))+ \mathcal{H}(\tilde{\pi}^t(\mu'), \tilde{\pi}^t(\mu)))$$
        for some $0 \leq \delta < \frac{m\lambda }{2L^2}$.
        \item[(v)]For any $t, x$, the map $\sigma^{-1}b$ is $L$-Lipschitz in $a$.
        \item[(vi)]$\Vert \cdot \Vert_A$ is induced by an inner product.
    \end{itemize}
\end{assumption}
Observe that the assumed monotonicity for each time is stronger than the integrated one that could be formulated in the spirit of Assumption \ref{asmp: unq} in $(iv)$. The benefit of this assumption is that we can integrate the inequality over arbitrary time intervals, not just the whole time horizon. Comparing assumption \ref{asmp: cnv rate} with assumption \ref{asmp: unq}, we can see that for $\delta = 0$, the new assumptions immediately imply uniqueness by Theorem \ref{thm: unq}, and for $\delta > 0$, this condition is weaker than the classical Lasry-Lions monotonicity condition\par
The last assumption is merely to ensure that the notion of strong concavity is compatible with our computations. Note however that we do not need to impose any differentiability on $h$.
\begin{theorem}\label{thm: cnv rate}
    Let the assumptions \ref{asmp: stnd}; \ref{asmp: cons} and \ref{asmp: cnv rate} be satisfied.
    Let $(\mu,q, \alpha)$ be a solution to the discounted infinite horizon mean field game and for any $T > 0$, let $(\mu^T, q^T, \alpha^T)$ be a solution to the mean field game \eqref{fnt MFG} with finite horizon $T$. Then,
    $$\mathcal{H}(\tilde{\pi}^{T, t}(\mu^T), \tilde{\pi}^t(\mu)) + \mathcal{H}( \tilde{\pi}^t(\mu), \tilde{\pi}^{T, t}(\mu^T)) \leq \frac{2Me^{-\lambda (T-t)}}{\frac{m\lambda}{2 L^2} - \delta}$$
    and
    $$(\mathbb{E}^{\alpha} + \mathbb{E}^{\alpha^T})\left[\int_0^T e^{-\lambda s}\Vert \alpha_s - \alpha^T_s\Vert_A^2 ds\right]\leq\left(1 + \frac{\delta}{\frac{m\lambda}{2L^2} - \delta}\right)\frac{8Me^{-\lambda T}}{\lambda m}$$
    as well as
    $$\int_0^T e^{-\lambda s} \mathcal{W}_1(q_s, q^T_s)^2 ds\leq \left(\frac{8}{m} + \frac{1}{\frac{m\lambda}{2L^2} - \delta}\left(C_A^2 + \frac{8\delta}{m}\right)\right)\frac{Me^{-\lambda T}}{\lambda}$$
    with $C_A$ being an upper bound of $\Vert \cdot \Vert_A$ on $A$.\par
    In particular, any solution to the infinite horizon game or the finite horizon game has to be unique.
\end{theorem}
The reason why we have stated the first bound in Theorem \ref{thm: cnv rate} in terms of the relative entropy is that it allows to deduce a variety of other bounds. 
For instance, if the the law of the solution for the finite or the infinite horizon mean field games satisfies a transportation cost inequality (see e.g.\ \cite{TangpiTI} and \cite{FBSDETI}), then the relative entropy serves as an upper bound for the Wasserstein distance.\par
For the last bound, as one would expect from the other proofs, convergence of the law of the control variable only occurs weakly. For our convenience, we provide the bound in terms of the $1$-Wasserstein metric on $\mathcal{P}(A)$ over the metric space $(A, \Vert \cdot \Vert_A)$, but it could again be rewritten in terms of any other metric metrizing the weak topology on $\mathcal{P}(A)$ such as the Bounded Lipschitz metric from \cite{Dudley} or (as $\Vert \cdot \Vert$ is bounded) any Wasserstein metric (see \cite[Theorem 5.5]{CarmonaDelarue}).\par
Further for $\delta > 0$, note that this result incorporates a novel uniqueness result that does not require the classic Lasry-Lions monotonicity, but only a weaker form.\par
Given the previous setting, we demonstrate the usefulness of the entropy formulation by rewriting it as a Theorem concerning the total variation distance.\par

\begin{corollary}\label{cor: cnv rate}
    Under the assumptions and notation of Theorem \ref{thm: cnv rate}, for any $0 \leq t \leq T$, we have
    $$d_{TV}(\tilde{\pi}^{T, t}(\mu^T), \tilde{\pi}^t(\mu)) \leq \sqrt{\frac{M}{2(\frac{m\lambda}{2L^2} - \delta)}}e^{-\frac{\lambda}{2} (T-t)}.$$
\end{corollary}
The proof is an immediate consequence of Pinsker's inequality.\par
\subsubsection{The invariant mean field game}
As an application of the results we have described, let us finish by discussing the "invariant mean field game". In this game, we would like to find a mean field equilibrium in which the state interaction measure stays invariant. Such a mean field game fundamentally differs from the ones we have considered so far as it becomes necessary to vary the initial distribution as well to ensure that the law of the state process stays invariant. In our previous results, the initial distribution was always held fixed which was essential to analyze the densities from our change of measures.\par
For our approach for the invariant mean field game, we need to ensure that our state process is given by a time homogeneous Markov SDE. Thus, we cannot work with path dependent coefficients as before, but instead, we will assume that $b$ and $f$ merely depend on the current state $X_s$. Further, for the invariant mean field game, we are also not interested in mean field interaction through the law of the whole path but only the time marginals. Thus, we assume that our coefficients are given in the form $b:\mathbb{R}^d \times \mathcal{P}(\mathbb{R}^d) \times A \rightarrow \mathbb{R}^d$, $\sigma: \mathbb{R}^d \rightarrow GL_d(\mathbb{R}^d)$ and $f: \mathbb{R}^d \times \mathcal{P}(\mathbb{R}^d) \times A \rightarrow \mathbb{R}^d$. In the following, we will call coefficients like this time homogeneous.\par
The way we have treated the extended mean field game does not allow us to apply our previous results to this new case without adjustments. Thus, for simplicity, we assume that there is no mean-field interaction through the law of the control.\par
As we will need to work with different initial distribution, for any $\upsilon \in \mathcal{P}(\mathbb{R}^d)$, we specify $$\mathbb{P}_\upsilon = \upsilon \times \mathbb{P}_W \in \mathcal{P}(\Omega)$$
Under $\mathbb{P}_\upsilon$, $X$ will thus be a solution of the driftless Markovian SDE
\begin{equation}\label{Mrkv SDE}
    dX_s = \sigma(X_s)dW_s
\end{equation}
with initial distribution $X_0 = \xi \sim \upsilon$. For any $\mu$ and $\alpha$, we define the change of measure $\mathbb{P}^{\mu, \alpha}_\upsilon$ with respect to $\mathbb{P}_\upsilon$ as before. For a given random variable $\xi$, we will also write $\mathbb{P}^{\mu, \alpha}_\xi$ directly. We use a similar notation for taking expectations, i.e.\ $\mathbb{E}^{\mu, \alpha}_\upsilon$ denotes taking expectations under $\mathbb{P}^{\mu, \alpha}_\upsilon$.\par
We can now define the invariant mean field game.
\begin{definition}
    We call $\mu \in \mathcal{P}(\mathbb{R}^d)$ a solution to the invariant mean field game, if there exists a control $\alpha\in \mathbb{A}$  that is a maximizer for
    $$\sup_{\alpha'\in \mathbb{A}} \mathbb{E}^{\mu, \alpha'}_\mu\left[\int_0^\infty e^{-\lambda s}f(X_s, \mu, \alpha'_s)ds\right]$$
    and we have for almost every $t \geq 0$ that $\mathbb{P}^{\mu, \alpha}_\mu \circ (X_t)^{-1} = \mu$.
\end{definition}
To analyze this game, we will base our framework on the constructions from \cite{Khasminskii}. In particular, we will be working with the following two domains. We fix two radii $0 < R' < R$ and write $\mathbb{S}'$ and $\mathbb{S}$ for the two spheres $\partial B_0(R')$ and $\mathbb{S}$ and $\mathcal{B}(\mathbb{S}')$ and $\mathcal{B}(\mathbb{S})$ for the Borel $\sigma$ algebras on them. In this setting, we assume
\begin{assumption}[Invariant mean field game]\label{asmp: erg}
    \begin{itemize}
        \item[(i)]$\sigma$ is continuous.
        \item[(ii)]The coefficients $b$, $\sigma$ and $f$ are time homogeneous as described above.
        \item[(iii)]$b$ is bounded on $B_0(R)$ by some $\Lambda > 0$ uniformly over all $\mu$ and $\alpha$. We choose $\Lambda$ such that it also bounds the spectral norms $\Vert \sigma \Vert$ and $\Vert \sigma^{-1}\Vert$.
        \item[(vi)]There is $k> 0$ such that for any $a \in A$ and $\mu \in \mathcal{P}(\mathcal{C})$, we have for any $x$ with $\Vert x \Vert \geq R'$ that $x^\top b(x, \mu, a) + \frac{1}{2}tr(\Sigma(x)) \leq -k$ with $\Sigma(x):= \sigma(x)\sigma(x)^\top$.
        \item[(v)]$b$ and $f$ are jointly continuous in $(\mu, a)$, for $\mu$ with respect to the topology of setwise convergence.
    \end{itemize}
\end{assumption}
The first four assumptions will ensure that for any $\mu, \alpha$, even if we don't start with an invariant measure, the state process will eventually end up in stationarity, with an uniformly quantifiable rate. Note that our last assumption replaces assumption \ref{asmp: cont}.\par
Further, this choice of coefficients can still be embedded within our previous framework. In fact, since for any $t \geq 0$, the map $\mathcal{C} \rightarrow \mathbb{R}^d, \omega \mapsto \omega_t$ is measurable, these coefficients could also be understood as function on $\mathcal{C}$ instead of $\mathbb{R}^d$. In particular, using this identification, we assume that assumption \ref{asmp: stnd} still holds.\par
Here and in the following, we speak of a Markov process having an invariant distribution, if its marginals are constant for all time. In that case, the common marginal distribution is called the invariant distribution. In contrast, in a more general sense, we speak of a stationary distribution as any distribution that if chosen as the initial distribution, would become an invariant distribution for the Markov process. In most of our discussed scenarios, stationary distributions arise as some limit of the marginal distributions.\par
Considering comparable sets of assumptions that deal with drift conditions ensuring the existence of stationary distributions such as in \cite{Veretennikov88, Veretennikov97, Veretennikov00, Malyshkin01}, we do not need to assume $\Vert \sigma \Vert$ or $\Vert b \Vert$ to be uniformly bounded like in these references. Under our set of weaker assumptions, we are still able to guarantee existence of a unique stationary distribution for the solution. However, we are not able to show convergence of the marginals to the stationary distribution: we can only show convergence in Césaro means. Still, this weaker convergence is sufficient to prove existence of a solution for the invariant game.
Even though we do not require uniform boundedness explicitly,  we still need to ensure that $\Vert \sigma^{-1} b\Vert$ stays uniformly bounded as assumed in assumption \ref{asmp: stnd}. These assumptions imply that $\sigma$ must satisfy
$$
-C\Vert x^\top \sigma(x)\Vert  + \frac{1}{2} tr(\Sigma(X)) \leq -k
$$
from which we see that even though $\sigma$ is allowed to be unbounded, it can at most grow linearly in $\Vert x \Vert$. Further, in this case, $\Sigma(x) b(x)$ always has to point sufficiently inwards. Let us also note that instead of using the euclidean norm as we do, the notion of distance to the origin can be generalized through the use of different Lyapunov functions as in \cite{Khasminskii, HairerMattingly, EthierKurtz}. 
\begin{theorem}\label{thm: inv mfg exst}
    Under assumptions \ref{asmp: stnd}, \ref{asmp: sngl} and \ref{asmp: erg}, there exists a solution to the invariant mean-field game.
\end{theorem}

\section{Proofs of the existence and uniqueness results}\label{sec: prf}





This section covers the proofs of our existence and uniqueness results. 
The proofs will rely on some topological arguments that are deferred to  Section \ref{sect: topo} for the sake of readability.
Furthermore, to make the paper self-contained, in Appendices \ref{apdx: IBSDE} and \ref{apdx: setvm} we revisit some existing results on infinite horizon BSDEs and set-valued maps that are used in the proofs.

\subsection{The discounted infinite horizon control problem}\label{sect: ctrl}
For measures $\mu\in \mathcal{P}(\mathcal{C})$ and $q:[0, \infty)\to \mathcal{P}(A)$, we maximize
$$J^{\mu, q}(\alpha) := \mathbb{E}^{\mu, \alpha}\left[ \int_0^\infty e^{-\lambda s} f(s, X, \mu, q_s, \alpha_s) ds \right]$$
over $\alpha \in \mathbb{A}$. 
We consider the Hamiltonian
$$h:[0,\infty) \times \mathcal{C} \times \mathcal{P}(\mathcal{C}) \times \mathbb{R}^d \times A \rightarrow \mathbb{R}, (t, x, \mu, q, z, a)\mapsto e^{-\lambda t}f(t, x, \mu, q, a) + z^\top \sigma^{-1}b(t, x, \mu, a)$$
as well as $H(t, x, \mu, q, z):= \sup_{a \in A} h(t, x, \mu, q, z, a)$.\par
For fixed $\alpha$, let us define the expected remaining utility
$$Y^{\mu, q, \alpha}_t := \mathbb{E}^{\mu, \alpha}\left[ \int_t^\infty e^{-\lambda s} f(s, X, \mu, q_s, \alpha_s) ds \vert \mathcal{F}_t \right].$$
In particular, we have $J^{\mu, q}(\alpha) = \mathbb{E}^{\mu, \alpha}[Y^{\mu, q, \alpha}_0]$ where, since $Y^{\mu, q, \alpha}_0$ is $\mathcal{F}_0$ measurable, the expectation is taken with respect to $\upsilon$ and thus independent of the change of measure and thus, of $\mu$ and $\alpha$. Now,
$$Y^{\mu, q, \alpha}_t + \int_0^t e^{-\lambda s} f(s, X, \mu, q_s, \alpha_s) ds = \mathbb{E}^{\mu, \alpha}\left[ \int_0^\infty e^{-\lambda s} f(s, X, \mu, q_s, \alpha_s) ds \vert \mathcal{F}_t \right]$$
is by assumption a uniformly bounded $\mathbb{P}^{\mu, \alpha}$ martingale that converges $\mathbb{P}^{\mu, \alpha}$ a.s.\ and in $\mathbb{L}^1$ towards the random variable $\int_0^\infty e^{-\lambda s} f(s, X, \mu, q_s, \alpha_s) ds$ for $t \rightarrow \infty$.\par
By Proposition \ref{prop: prp}, the predictable representation property holds for our locally completed filtration $\mathbb{F}$ with respect to $W^{\mu, \alpha}$ for $(\mathbb{P}^{\mu, \alpha}, \mathbb{F})$ martingales, by the martingale representation Theorem there exists a $\mathbb{F}$ progressively measurable square integrable $Z^{\mu, q, \alpha}$ such that
$$\int_0^\cdot e^{-\lambda s} f(s, X, \mu, q_s, \alpha_s) ds = Y^{\mu, q, \alpha}_0 +  \int_0^\cdot Z^{\mu, q, \alpha}_sdW^{\mu, \alpha}_s$$
showing that $(Y^{\mu, q,\alpha}, Z^{\mu, q, \alpha})$ is a solution to the linear infinite time horizon BSDE
\begin{equation}\label{inft cntrl bsde}
    \begin{split}
        Y^{\mu, q, \alpha}_t =& Y^{\mu, q, \alpha}_T +   \int_t^T e^{-\lambda s} f(s, X, \mu, q_s, \alpha_s) ds - \int_t^T Z^{\mu, q, \alpha}_s dW^{\mu, \alpha}_s \\
        = & Y^{\mu, q, \alpha}_T +   \int_t^T h(s, X, \mu, q_s, Z^{\mu, q, \alpha}_s , \alpha_s) ds - \int_t^T Z^{\mu, q, \alpha}_s dW_s
    \end{split}
\end{equation}
holding for all $0\leq t \leq T$. 
Let us define
$$\tilde{Y}^{\mu, q, \alpha}_t := e^{\lambda t} Y^{\mu, q, \alpha}_t \quad \text{and}\quad \tilde{Z}^{\mu, q, \alpha}_t := e^{\lambda t}Z^{\mu, q, \alpha}_t. $$
Note that still, $\mathbb{E}^{\mu, \alpha}[\tilde{Y}^{\mu, q, \alpha}_0] = J^{\mu, q}(\alpha)$. Then, Itô's Lemma gives
\begin{equation*}
    \begin{split}
        d\tilde{Y}^{\mu, q, \alpha}_s = & e^{\lambda s} (-h(s, X, \mu, q_s, Z^{\mu, q, \alpha}_s, \alpha_s)ds + Z^{\mu, q, \alpha}_s dW_s) + \lambda e^{\lambda s} Y^{\mu, q, \alpha}_s ds\\
        = & -(\tilde{h}(s, X, \mu, q_s, \tilde{Z}^{\mu, q, \alpha}_s,\alpha_s) - \lambda \tilde{Y}^{\mu, q, \alpha}_s)ds + \tilde{Z}^{\mu, q, \alpha}_s dW_s
    \end{split}
\end{equation*}
where we recall that $\tilde{h}$ is given by
$$\tilde{h}(t, x, \mu, q, z, a) = f(t, x, \mu, q, a) + z^\top \sigma(t,x)^{-1}b(t, x, \mu, a).$$
Note that by construction, $\tilde{Y}^{\mu, q, \alpha}$ is uniformly bounded by $\frac{M}{\lambda}$, and we can see that $(\tilde{Y}^{\mu, q, \alpha}, \tilde{Z}^{\mu, q, \alpha})$ satisfies
\begin{equation*}
    \tilde{Y}^{\mu, q, \alpha}_t = \tilde{Y}^{\mu, q, \alpha}_T + \int_t^T( \tilde{h}(s, X, \mu, q_s, \tilde{Z}^{\mu, q, \alpha}_s, \alpha_s) - \lambda \tilde{Y}^{\mu, q, \alpha}_s )ds - \int_t^T \tilde{Z}^{\mu, q, \alpha}_s dW_s
\end{equation*}
for any $t \leq T$ and every $T>0$.\par
We define 
$$\tilde{H}(t, x, \mu, q, z):= \sup_{a \in A} \tilde{h}(t, x, \mu, q, z, a)$$ as the maximized Hamiltonian after this transformation. Since $h$ is maximized pointwise for every $t$, we have
\begin{equation*}\label{transfham}
    e^{\lambda t} H(t, x, \mu, q, z)= \sup_{a\in A} (f(t, x, \mu, q, a) + e^{\lambda t} z^\top \sigma(t,x)^{-1}b(t, x, \mu, a)) =   \tilde{H}(t, x, \mu, q, e^{\lambda t} z)
\end{equation*}
and the maximizers of $h$ and $\tilde h$ are the same. Recall that the set of maximizers of the transformed Hamiltonian is written as 
$$\mathcal{A}(t, x, \mu, q, z) := \arg\max_{a \in A} \tilde{h}(t, x, \mu, q, z, a).$$ 
Since $\tilde{h}$ is continuous in $a$ and $A$ is compact, for any $t, x, \mu, q, z$, the set of maximizers is non empty. As we assumed separability of $f$, the maximizers are independent of $q$ and we will drop the argument $q$ from our notation.\par
By Assumption \ref{asmp: stnd}$(i)$, the function $\tilde{h}$ is Lipschitz in the $z$ variable uniformly in the other variables, and thus $\tilde{H}$ is also uniformly Lipschitz in $z$ with a Lipschitz constant that does not depend on $\mu$ and $q$. We can now characterize the optimal controls with a standard optimal control argument.
\begin{proposition}\label{prop.optim.charac}
    Under Assumption \ref{asmp: stnd}, for any fixed $\mu$ and $q$, there exists a unique solution $(\tilde{Y}^{\mu, q}, \tilde{Z}^{\mu, q})$ to the infinite time horizon BSDE
    \begin{equation}\label{inft BSDE}
        \tilde{Y}^{\mu, q}_t = \tilde{Y}^{\mu, q}_T + \int_t^T( \tilde{H}(s, X, \mu, q_s, \tilde{Z}^{\mu, q}_s) - \lambda \tilde{Y}^{\mu, q}_s )ds - \int_t^T \tilde{Z}^{\mu, q}_s dW_s
    \end{equation}
    such that $\vert \tilde{Y}^{\mu, q}_t \vert \leq \frac{M}{\lambda}$ and $\mathbb{E}\left[\int_0^\infty e^{-2\lambda s} \Vert \tilde{Z}^{\mu, q}_s \Vert^2 ds \right] < \infty$. Further, $V^{\mu, q} = \mathbb{E}[\tilde{Y}^{\mu, q}_0]$, an optimal control exists, and $\alpha \in \mathbb{A}$ is optimal if and only if a.e., $\alpha_t \in \mathcal{A}(t, X, \mu, \tilde Z^{\mu,q}_t)$.
\end{proposition}
\begin{proof}
Under our assumptions, the driver $\tilde{H}(s, X, \mu, q_s, z) - \lambda y$ satisfies the requirements in Theorem \ref{thm: Royer}, showing existence, uniqueness, and bounds of $(\tilde{Y}^{\mu, q}, \tilde{Z}^{\mu, q})$. Now, for any $\alpha \in \mathbb{A}$, we have by an adaption of the comparison principle \cite[Theorem 2.2]{Royer04} that $\tilde{Y}^{\mu, q}_t \geq \tilde{Y}^{\mu, q, \alpha}_t$ a.s. for all $t \geq 0$. By the measurable maximum Theorem \cite[Theorem 18.19]{IDA}, there exists a progressively measurable $\hat{\alpha}^{\mu, q}$ such that a.e., $\hat{\alpha}_t \in \mathcal{A}(t, X, \mu, \tilde Z^{\mu,q}_t)$ and thus $\tilde{h}(s, X, \mu, q_s, \tilde{Z}^{\mu, q}_s, \hat{\alpha}^{\mu, q}_s) = \tilde{H}(s, X, \mu, q_s, \tilde{Z}^{\mu, q}_s)$.\par
For any $\alpha$ with this property, we have $\tilde{Y}^{\mu, q}_0 = \tilde{Y}^{\mu, q, \alpha}_0$, thus $V^{\mu, q} = \mathbb{E}[\tilde{Y}^{\mu, q}_0]$ and such $\alpha$ is optimal. It remains to show the converse. Consider an arbitrary optimal $\alpha$. Such $\alpha$ must satisfy $\mathbb{E}[\tilde{Y}^{\mu, q}_0] = \mathbb{E}[ \tilde{Y}^{\mu, q, \alpha}_0]$. By Itô's formula, for any $T \geq 0$, we have
\begin{equation*}
    \begin{split}
        \tilde{Y}^{\mu, q}_0 - \tilde{Y}^{\mu, q, \alpha}_0 =& e^{-\lambda T}(\tilde{Y}^{\mu, q}_T - \tilde{Y}^{\mu, q, \alpha}_T) + \int_0^T e^{-\lambda s} (\tilde{H}(s, X, \mu, q_s, \tilde{Z}^{\mu, q}_s) - \tilde{h}(s, X, \mu, q_s, \tilde{Z}^{\mu, q}_s, \alpha_s))ds\\
        -&\int_0^T e^{-\lambda s} (\tilde{Z}^{\mu, q}_s - \tilde{Z}^{\mu, q, \alpha}_s)(dW_s - \beta_s ds)
    \end{split}
\end{equation*}
where $\beta_s := \mathbb{1}_{\lbrace \tilde{Z}^{\mu, q}_s \neq \tilde{Z}^{\mu, q, \alpha}_s\rbrace} \frac{\tilde{h}(s, X, \mu, q_s, \tilde{Z}^{\mu, q}_s, \alpha_s) - \tilde{h}(s, X, \mu, q_s, \tilde{Z}^{\mu, q, \alpha}_s, \alpha_s)}{\Vert \tilde{Z}^{\mu, q}_s - \tilde{Z}^{\mu, q, \alpha}_s \Vert^2}(\tilde{Z}^{\mu, q}_s - \tilde{Z}^{\mu, q, \alpha}_s)$ and the last integral with respect to $\beta_s ds$ is to be understood as an inner product. As a.s., $\Vert \beta_s \Vert \leq C$, we can define a to $\mathbb{P}$ equivalent measure $\mathbb{Q}$ via $\frac{d\mathbb{Q}}{d\mathbb{P}} = \mathcal{E}(\int_0^\cdot \beta_s dW_s)_T$. Note that both $\tilde{Z}^{\mu, q, \alpha}$ and $\tilde{Z}^{\mu, q} = \tilde{Z}^{\mu, q, \hat{\alpha}}$ are square integrable on $[0, T]$ by Theorem \ref{thm: Royer}. As $\frac{d\mathbb{Q}}{d\mathbb{P}} = \mathcal{E}(\int_0^\cdot \beta_s dW_s)_T$ is square integrable as well, by the Cauchy-Schwarz inequality, $\mathbb{E}^\mathbb{Q}[(\int_0^T e^{-2\lambda s} \Vert \tilde{Z}^{\mu, q}_s - \tilde{Z}^{\mu, q, \alpha}_s \Vert^2 ds)^{\frac{1}{2}}]<\infty$. By the Burkholder-Davis-Gundy inequality, $\int_0^\cdot e^{-\lambda s} (\tilde{Z}^{\mu, q}_s - \tilde{Z}^{\mu, q, \alpha}_s)(dW_s - \beta_s ds)$ on $[0, T]$ is a $\mathbb{Q}$ local martingale of class DL, thus a true $\mathbb{Q}$ martingale.\par
Hence, taking expectation with respect to $\mathbb{Q}$ shows $0 = \mathbb{E}^\mathbb{Q}[\tilde{Y}^{\mu, q}_0 - \tilde{Y}^{\mu, q, \alpha}_0] \geq e^{-\lambda T}\mathbb{E}^\mathbb{Q}[\tilde{Y}^{\mu, q}_T - \tilde{Y}^{\mu, q, \alpha}_T]$ as $\mathbb{P}$ and $\mathbb{Q}$ coincide on $\mathcal{F}_0$. Since we have already shown $\tilde{Y}^{\mu, q}_T \geq \tilde{Y}^{\mu, q, \alpha}_T$ a.s., this implies $\tilde{Y}^{\mu, q}_T = \tilde{Y}^{\mu, q, \alpha}_T$ a.s. for any $T\geq 0$. One can thus see that $\int_0^\cdot \tilde{Z}^{\mu, q}_s - \tilde{Z}^{\mu, q, \alpha}_s dW_s = \int_0^\cdot \tilde{H}(s, X, \mu, q_s, \tilde{Z}^{\mu, q}_s) - \tilde{h}(s, X, \mu, q_s, \tilde{Z}^{\mu, q, \alpha}_s, \alpha_s)ds$ is a finite variation local martingale and thus zero. Since $\int_0^\cdot \tilde{Z}^{\mu, q}_s - \tilde{Z}^{\mu, q, \alpha}_s dW_s$ is a square integrable martingale, this shows that a.e., $\tilde{Z}^{\mu, q}_t =\tilde{Z}^{\mu, q, \alpha}_t$. In particular, this implies a.e.\ that $\tilde{H}(t, X, \mu, q_t, \tilde{Z}^{\mu, q}_t) = \tilde{h}(t, X, \mu, q_t, \tilde{Z}^{\mu, q, \alpha}_t, \alpha_t)$ and thus $\alpha_t \in \mathcal{A}(t, X, \mu, \tilde{Z}^{\mu, q}_t)$.
\end{proof}

\subsection{Compact containment of the laws}Fixed point theorems are at the heart of our existence proof for which compactness is essential.
In particular, we will identify topological spaces in which the sets of laws of states and of controls are (relatively) compact.

\subsubsection{The laws of the state variable}\label{sect: cmp cnt}

To begin with, observe that
for any $p$ with $\vert p \vert \geq 1$ and any $T > 0$, we always have
$$\mathbb{E}\left[\left(\frac{d\mathbb{P}^{\mu, \alpha}_{\vert \mathcal{F}_T}}{d\mathbb{P}_{\vert \mathcal{F}_T}}\right)^p\right] \leq e^{\frac{p(p-1)TC^2}{2}} =: M^T_p. \quad$$
This directly follows from Assumption \ref{asmp: stnd}$(i)$.\par
Recall the topology $\tau$ on $\mathcal{P}(\mathcal{C})$ defined in Definition \ref{def: loc stw}. It will be easier to work with a comparable topology on $\mathcal{P}(\Omega)$ directly. Thus, on $\mathcal{P}(\Omega)$, we define a local topology of setwise convergence which we also denote as $\tilde{\tau}$ as the coarsest topology on $\mathcal{P}(\Omega)$ such that all maps $\mu \mapsto \int \phi d\mu$ are continuous for all $\phi \in B(\Omega)$, namely all measurable bounded $\phi:\Omega \rightarrow \mathbb{R}$ that only depend on $\zeta = (\xi, \omega)$ via $(\xi, \omega_{\vert [0, T]})$ for some $T \geq 0$.\par
To find a proper compact restriction, let us consider the finite horizon spaces first that we have already introduced in Section \ref{sct: LinkMFG}. We denote with $\tilde{\tau}^T$ the topology of setwise convergence on $\mathcal{P}(\Omega^T)$. Now within $\mathcal{P}(\Omega^T)$, we define the subspace
\begin{equation*}
    \tilde{\mathcal{Q}}^T:= \left\lbrace \mathbb{Q} \in \mathcal{P}(\Omega^T) \vert \int_{\mathcal{C}_0^T} \mathbb{Q}(\cdot, d\omega) = \upsilon, \mathbb{Q} \sim \mathbb{P}^T, \mathbb{E}\left[\left(\frac{ d\mathbb{Q}}{d\mathbb{P}^T}\right)^2 \right] \leq M^T, \mathbb{E}\left[\left(\frac{ d\mathbb{Q}}{d\mathbb{P}^T}\right)^{-1} \right]\leq M^T\right\rbrace
\end{equation*}
with $M^T := M_2^T =M_{-1}^T$. We prove in Proposition \ref{prop: ttauT} that $\tilde{\mathcal{Q}}^T$ is a convex, metrizable and compact space. In the following, we fix one metric $d_{\tilde{\tau}^T}$ that we define in Proposition \ref{prop: tau mtrc}. This metric metrizes $\tilde{\tau}^T$ and is bounded by the total variation distance.\par
Using the notation we have introduced, note that we can also write
$$B(\Omega) = \lbrace f \in \mathbb{L}^\infty(\Omega, \mathbb{P}) \vert \exists T > 0, g \in B(\Omega^T): f = g \circ \pi^T \rbrace.$$
where $B(\Omega^T)$ denotes the set of real valued measurable bounded functions on $\Omega^T$. Note that any $f \in B(\Omega)$ is defined locally, the reference measure can be chosen to be any of the locally equivalent measures, i.e.\ $f\in \mathbb{L}^\infty(\Omega, \mathbb{Q})$ for any $\mathbb{Q}\in \mathfrak{E}$. Comparing with Definition \ref{def: loc stw}, it can thus be seen that a sequence $(\mathbb{Q}^n)_{n \geq 1} \subset \mathcal{P}(\Omega)$ converges towards some $\mathbb{Q} \in \mathcal{P}(\Omega)$ with respect to $\tilde{\tau}$ if and only if for any $T > 0$, the sequence $(\tilde{\pi}^T(\mathbb{Q}^n))_{n \geq 1} \subset \mathcal{P}(\Omega^T)$ converges towards $\tilde{\pi}^T(\mathbb{Q}) \in \mathcal{P}(\Omega^T)$.\par
This leads us to defining
\begin{equation}\label{def: Qtild}
    \tilde{\mathcal{Q}} := \lbrace \mathbb{Q} \in \mathcal{P}(\Omega) \vert \forall T > 0: \tilde{\pi}^T(\mathbb{Q}) \in \tilde{\mathcal{Q}}^T \rbrace.
\end{equation}
By construction, for any $\mu, \alpha$, we have $\mathbb{P}^{\mu, \alpha} \in \tilde{\mathcal{Q}} \subset \mathfrak{E}$. As it is shown later in Section \ref{sect: topo}, equipped with the topology $\tilde{\tau}$, this space is metrizable and compact, and a compatible metric reads
\begin{equation}\label{mtr tau def}
    d_{\tilde{\tau}}(\mathbb{Q}, \mathbb{Q}') := \sum_{n = 1}^\infty 2^{-n} d_{\tilde{\tau}^n} (\tilde{\pi}^n(\mathbb{Q}), \tilde{\pi}^n(\mathbb{Q'})),
\end{equation}
see Proposition \ref{prop: inft Q top}.
Coming back to the laws of the state process, consider the map $\mathfrak{u}:\mathcal{P}(\Omega) \rightarrow \mathcal{P}(\mathcal{C}), \mathbb{Q} \mapsto \mathbb{Q} \circ X^{-1}$. Under slight abuse of notation, we also define the maps $\pi^T: \mathcal{C} \rightarrow \mathcal{C}^T$ and $\tilde{\pi}^T: \mathcal{P}(\mathcal{C}) \rightarrow \mathcal{P}(\mathcal{C}^T)$ straightforwardly. On $\mathcal{P}(\mathcal{C}^T)$, we define $\tau^T$ to be the topology of setwise convergence. Clearly, the $\tilde{\pi}^T$ are continuous with respect to $\tau$ and $\tau^T$.\par
Note that $X$ considered as a function $\Omega \rightarrow \mathcal{C}$ is measurable. It is not only measurable, but since as a process it is adapted as well, $X$ is non anticipative, that is, for any $T > 0$, $X(\xi, \omega)_{\vert [0, T]}$ depends only on $(\xi, \omega_{\vert [0, T]})$. Thus, $\pi^T \circ X$ defines a measurable map $X^T: \Omega^T \mapsto \mathcal{C}^T$ such that $\pi^T \circ X = X^T \circ \pi^T$.\par
Thus, for any $\mu \in \mathcal{P}(\Omega)$, we have $(\tilde{\pi}^T \circ \mathfrak{u})(\mu) = \mu \circ (X^T \circ \pi^T)^{-1}$ which shows that $\tilde{\pi}^T \circ \mathfrak{u}$ is continuous for any $T \geq 0$, and by \cite[Proposition 3.2]{Brezis}, this implies that $\mathfrak{u}$ is continuous with respect to $\tilde{\tau}$ and $\tau$. We then define the space 
$$\mathcal{Q}:=\mathfrak{u}(\tilde{\mathcal{Q}}).$$ 
Again, for any $\mu$ and $\alpha$, we have that $\mathbb{P}^{\mu, \alpha} \circ X^{-1}$ lies in $ \mathcal{Q}$. Further, by the Hanai-Morita-Stone Theorem \cite[Theorem 4.4.17.]{Engelking}, since $\mathcal{Q}$ is (by Proposition \ref{prop: ext}) a subset of a Hausdorff space and thus Hausdorff itself, the set $\mathcal{Q}$ is metrizable again. Since $\tilde{\mathcal{Q}}$ is compact, (see Proposition \ref{prop: inft Q top}) it follows that $\mathcal{Q}$ is compact as well.
\subsubsection{The law of the control and the stable topology}
To find compactness in the space of laws of the control variable, we will work with relaxation. Let $\mathcal{M}$ be the space of $\sigma$-finite measures on $[0, \infty) \times \mathcal{P}(A)$ such that the first marginal is the Lebesgue measure $dt$. $\mathcal{P}(A)$ is considered as a compact Polish space using the weak topology. Usual disintegration results, such as \cite[Corollary 3.9]{JacodMemin}, immediately generalize to $\sigma$-finite measures as well, so that every element of $\nu \in \mathcal{M}$ can be disintegrated into the form $\nu_t dt$. This way, we can consider $\nu$ as a Borel measurable map $[0, \infty) \rightarrow \mathcal{P}(\mathcal{P}(A))$ as well.\par
Now, the present work differs from the usual application in that our base space is $[0, \infty)$ equipped with the non finite measure $dt$. We thus need to consider a slightly modified construction.
\begin{definition}
    We call the local stable topology which we also denote by $\mathcal{S}$ the coarsest topology that makes the maps $\nu \mapsto \int_0^\infty \int_{\mathcal{P}(A)} \phi(t, q) \nu_t(dq) dt$ continuous for every function $\phi:[0, \infty) \times \mathcal{P}(A) \rightarrow \mathbb{R}$ that is bounded, measurable in $t$ and continuous in measures $q\in \mathcal{P}(A)$ that are compactly supported in time, i.e.\ there is $T>0$ such that $\phi(t, q) = 0$ for any $t > T$.
\end{definition}
As we will show in Section \ref{sect: stblc}, under this topology, the space $\mathcal{M}$ is metrizable and compact.\par
Recall that we are still mainly interested in $\mathcal{P}(A)$-valued processes. 
We can embed any such $q: [0, \infty) \rightarrow \mathcal{P}(A)$ into $\mathcal{M}$ as the measure $\delta_{q_t}dt$. This embedding is continuous with respect to convergence in $dt$ measure of $\mathcal{P}(A)$ valued processes. 
We denote the image in $\mathcal{M}$ by $\mathcal{M}^{(0)}$ and will refer to $\mathcal{M}^{(0)}$ as the set of strict flows whereas elements $\nu \in \mathcal{M}$ will also be called relaxed flows. These \emph{relaxed} flows can thus be understood as a compactification of the \emph{strict} ones. In the following, we will consider any map $F: \mathcal{P}(A) \rightarrow \mathbb{R}$ as a map $F: \mathcal{P}(\mathcal{P}(A)) \rightarrow \mathbb{R}$ via $F(\nu) \mapsto \int_{\mathcal{P}(A)} F(q) \nu(dq)$. Note that this identification is consistent with the embedding above. As in our results so far, such as in Proposition \ref{prop.optim.charac}, we only worked with fixed $q$, these results still hold true if we replace the strict flows by relaxed flows $\nu \in \mathcal{M}$.\par

\subsection{Continuity of the solution map}
In Proposition \ref{prop.optim.charac}, we showed that for any fixed measures $\mu$ and $\nu$, a control is optimal iff we have $dt \times \mathbb{P} $ a.e.\ $\alpha_t \in \mathcal{A}(t, X, \mu, \tilde{Z}^{\mu, \nu}_t)$ for $(\tilde{Y}^{\mu, \nu}, \tilde{Z}^{\mu, \nu})$ solving \eqref{inft BSDE}. Let us define the set of optimal controls
$$\mathbb{A}(\mu, \nu) := \lbrace \alpha \in \mathbb{A} \vert \alpha_t \in \mathcal{A}(t, X, \mu, Z^{\mu, \nu}_t)\,dt \times d\mathbb{P}\, \, a.e.\rbrace$$
with $Z^{\mu, \nu}$ defined as in \eqref{inft BSDE}. 
With slight abuse of notation, in the following we will also view $\mathbb{A}$ as a set-value map on $\mathcal{P}(\mathcal{C}) \times \mathcal{M}$ as well.\par

Under this formulation, to solve the mean field game, our goal will be to find $(\mu, q)$ for which there exists $\alpha \in \mathbb{A}(\mu, q)$ so that $\mu = \mathbb{P}^{\mu, \alpha} \circ X^{-1}$ and $q_t = \mathbb{P}^{\mu, \alpha} \circ (\alpha_t)^{-1}$ for almost every $t \geq 0$. This of course automatically requires $q \in \mathcal{M}^{(0)}$. As we know that $\mathbb{P}^{\mu, \alpha} \circ X^{-1} \in \mathcal{Q}$, we can restrict ourselves to considering only the space $\mathcal{Q}$ instead of $\mathcal{P}(\mathcal{C})$.

In the following, we will work with upper hemicontinuous set-valued maps. 
As there are different notions of upper hemicontinuity in the literature, we decided to distinguish these by referring to them as metrically or topologically upper hemicontinuous maps, see Appendix \ref{apdx: setvm} for details.
We first consider continuity of the optimal control map.
\begin{lemma}\label{lem: cntrcont}
    The set-valued map $\mathbb{A}: \mathcal{Q} \times \mathcal{M} \twoheadrightarrow \mathbb{A}$ is metrically upper hemicontinuous.
\end{lemma}
\begin{proof}
It suffices by Proposition \ref{prop: metr uhc} to show that for any sequence $(\mu^n, \nu^n) \subset \mathcal{Q} \times \mathcal{M}$ that converges to some $(\mu, \nu) \in \mathcal{Q} \times \mathcal{M}$, we have
$$\sup_{\alpha^n \in \mathbb{A}(\mu^n, \nu^n)} \inf_{\alpha\in \mathbb{A}(\mu, \nu)} \Vert \alpha^n - \alpha \Vert_{\mathbb{A}} \rightarrow 0.$$
By our boundness and Lipschitz assumptions, we have $\tilde{H}(s, X, \mu^n, \nu^n_s, \tilde{Z}^{\mu, \nu}_s) \leq M + C \vert \tilde{Z}^{\mu, \nu}_s \vert$ and $\tilde{H}(s, X, \mu, \nu_s, \tilde{Z}^{\mu, \nu}_s) \leq M + C \vert \tilde{Z}^{\mu, \nu}_s \vert$. Further, using \cite[Lemma 5.6]{Carmona15}, one can prove that for any $t \geq 0$ and any $n$, one has
$$\vert \tilde{H}(t, X, \mu^n, \nu^n_t, Z^{\mu, \nu}_t) - \tilde{H}(t, X, \mu, \nu^n_t, Z^{\mu, \nu}_t) \vert \rightarrow 0$$
keeping in mind that $\nu^n$ may be relaxed and both terms need to be evaluated as integrals over $\mathcal{P}(A)$.\par

Let us consider arbitrary $0 \leq t \leq T$. By the definition of the local stable topology, we immediately get 
$$\left\vert \int_t^T \tilde{H}(s, X, \mu, \nu^n_s,  \tilde{Z}^{\mu, \nu}_s) ds - \int_t^T \tilde{H}(s, X, \mu, \nu_s,  \tilde{Z}^{\mu, \nu}_s) ds \right\vert \rightarrow 0.$$
As $Z^{\mu, \nu}$ is square integrable, by the dominated convergence Theorem, one has
\begin{equation*}
    \begin{split}
        &\mathbb{E}\left[\left(\int_t^T \tilde{H}(s, X, \mu^n, \nu^n_s,  \tilde{Z}^{\mu, \nu}_s) ds - \int_t^T \tilde{H}(s, X, \mu, \nu_s,  \tilde{Z}^{\mu, \nu}_s) ds\right)^2\right]\\
        \leq & 2 \mathbb{E}\left[\left(\int_t^T \tilde{H}(s, X, \mu^n, \nu^n_s,  \tilde{Z}^{\mu, \nu}_s) ds - \int_t^T \tilde{H}(s, X, \mu, \nu^n_s,  \tilde{Z}^{\mu, \nu}_s) ds\right)^2\right] \\
        \qquad  &+ 2 \mathbb{E}\left[\left(\int_t^T \tilde{H}(s, X, \mu, \nu^n_s,  \tilde{Z}^{\mu, \nu}_s) ds - \int_t^T \tilde{H}(s, X, \mu, \nu_s,  \tilde{Z}^{\mu, \nu}_s) ds\right)^2\right] \rightarrow 0.
    \end{split}
\end{equation*}
By Lemma \ref{lem: inft bsde stab}, this shows
$$\mathbb{E}\left[\int_0^\infty e^{-2\lambda s}\Vert \tilde{Z}^{\mu^n, \nu^n}_s - \tilde{Z}^{\mu, \nu}_s  \Vert^2 ds\right] \rightarrow 0.$$
In particular, this proves that $dt \times \mathbb{P}$ a.s., the sequence $(\tilde{Z}^{\mu^n, \nu^n}_s)_n$ converges to $\tilde{Z}^{\mu, \nu}_s$. Now, by Berge's maximum Theorem \cite[Theorem 17.31]{IDA}, for any fixed $(t, x)$, we know that $(\mu, z) \mapsto \mathcal{A}(t, x, \mu, z)$ is upper hemicontinuous. By Proposition \ref{prop: metr uhc}, $dt \times \mathbb{P}$ a.s., we have
$$\sup_{a^n \in \mathcal{A}(t, X, \mu^n, Z^{\mu^n, \nu^n}_t)} \inf_{a \in \mathcal{A}(t, X, \mu, Z^{\mu, \nu}_t)} \Vert a^n - a \Vert_A \rightarrow 0.$$
To extend this pointwise convergence to convergence of the $\mathbb{A}(\mu^n, \nu^n)$ to $\mathbb{A}(\mu, \nu)$ we apply an idea based on the measurable maximum Theorem \cite[Theorem 18.19]{IDA} as it was used in \cite[Lemma 7.11]{Carmona15}. First, by the measurable maximum Theorem, the set valued map $(t, \omega) \mapsto \mathcal{A}(t, X, \mu, Z^{\mu, \nu}_t)$ on $[0, \infty) \times \Omega$ equipped with the progressive $\sigma$ algebra is weakly measurable in a sense on which we will elaborate on in more detail in Definition \ref{def: wmsbl}.\par
We can apply the measurable maximum Theorem again to find a measurable selector $\hat{\beta}: [0, \infty) \times \Omega \times A \rightarrow A$ such that for any $t, \omega, a$, we have $\hat{\beta} \in \arg\min_{b \in \mathcal{A}(t, X, \mu, Z^{\mu, \nu}_t)} \Vert a - b\Vert_A$. This way, for any $\alpha^n \in \mathbb{A}$,
\begin{equation*}
    \begin{split}
        &\inf_{\alpha \in \mathbb{A}(\mu, \nu)} \mathbb{E}\left[\int_0^\infty e^{-\lambda s} \Vert \alpha^n_s - \alpha_s \Vert_A ds\right] \leq \mathbb{E}\left[\int_0^\infty e^{-\lambda s} \Vert \alpha^n_s - \hat{\beta}(s, \omega, \alpha^n_s)\Vert_A ds\right]\\
        = & \mathbb{E}\left[\int_0^\infty e^{-\lambda s} \inf_{b \in \mathcal{A}(s, X, \mu, Z^{\mu, \nu}_s)} \Vert \alpha^n_s - b \Vert_Ads\right] \leq \inf_{\alpha \in \mathbb{A}(\mu, \nu)} \mathbb{E}\left[\int_0^\infty e^{-\lambda s} \Vert \alpha^n_s - \alpha_s \Vert_A ds\right],
    \end{split}
\end{equation*}
so that all inequalities above are equalities. By the Berge maximum Theorem \cite[Theorem 17.31]{IDA}, for any $(t, \omega)$ fixed, $a \mapsto \Vert a - \hat{\beta}(t, \omega, a) \Vert_A$ is continuous, and we can repeat this argument with the measurable maximum Theorem to find
\begin{equation*}
    \begin{split}
        \sup_{\alpha^n \in \mathbb{A}(\mu^n, \nu^n)} \inf_{\alpha\in \mathbb{A}(\mu, \nu)} \Vert \alpha^n - \alpha \Vert_{\mathbb{A}} = \mathbb{E}\left[\int_0^\infty e^{-s} \sup_{a^n \in \mathcal{A}(t, X, \mu^n, Z^{\mu^n, \nu^n}_t)} \inf_{a \in \mathcal{A}(t, X, \mu, Z^{\mu, \nu}_t)} \Vert a^n - a \Vert_A ds\right].
    \end{split}
\end{equation*}
As the right hand side converges to zero by dominated convergence, this ultimately shows upper hemicontinuity.
\end{proof}
Next, let us introduce the map 
$$\mathfrak{P}: \mathcal{Q} \times \mathbb{A} \rightarrow \tilde{\mathcal{Q}},\,\, (\mu, \alpha) \mapsto \mathbb{P}^{\mu, \alpha}.$$ Consider the composition $\mathfrak{P} \circ \overline{\mathbb{A}}: \mathcal{Q} \times \mathcal{M} \twoheadrightarrow \tilde{\mathcal{Q}}$, where $\overline{\mathbb{A}}(\mu, \nu) := (\mu, \mathbb{A}(\mu, \nu))$ and $\mathfrak{P}(\overline{\mathbb{A}}(\mu, \nu)) := \lbrace \mathbb{P}^{\mu, \alpha} \vert \alpha \in \mathbb{A}(\mu, \nu)\rbrace$. To prove that this map is metrically upper hemicontinuous as well, we can use that $\overline{\mathbb{A}}$ is already metrically upper hemicontinuous by Lemma \ref{lem: cntrcont}. To extend that result we use that $\mathfrak{P}$ is uniformly continuous.

\begin{lemma}\label{lem: P uc}
    The map $\mathfrak{P}: \mathcal{Q} \times \mathbb{A} \rightarrow \tilde{\mathcal{Q}}$ is uniformly continuous. In particular, $\mathfrak{P} \circ \overline{\mathbb{A}}: \mathcal{Q} \times \mathcal{M} \twoheadrightarrow \tilde{\mathcal{Q}}$ is metrically upper hemicontinuous.
\end{lemma}
\begin{proof}
In this result, we are working with the metrics $d_{\tilde{\tau}}$ on $\tilde{\mathcal{Q}}$ that are defined in \eqref{mtr tau def}. For $\mathcal{Q}$, we already know that its topology $\tau$ is metrizable, and for this result, we can work with any metric $d_\tau$ that metrizes $\tau$.\par
Take $\epsilon > 0$ arbitrarily. We can fix some $\tilde{N} \geq 1$ such that $\sum_{n = \tilde{N} + 1}^\infty 2^{-n + 1} < \frac{\epsilon}{2}$. This way, we have for any $\mathbb{Q}, \mathbb{Q}' \in \tilde{\mathcal{Q}}$
$$d_{\tilde{\tau}}(\mathbb{Q}, \mathbb{Q}') = \sum_{n =1}^\infty 2^{-n} d_{\tilde{\tau}^n}(\tilde{\pi}^n(\mathbb{Q}), \tilde{\pi}^n(\mathbb{Q}')) \leq \sum_{n =1}^{\tilde{N}} 2^{-n} d_{\tilde{\tau}^n}(\tilde{\pi}^n(\mathbb{Q}), \tilde{\pi}^n(\mathbb{Q}')) + \frac{\epsilon}{2}.$$
For any $N \leq \tilde{N}$, let us consider the finite time horizon $[0, N]$ for now. Recall that for $\mu \in \mathcal{Q}$ and $\alpha \in \mathbb{A}$, the restriction of the measure $\mathbb{P}^{\mu, \alpha}$ on $\mathcal{F}_N$ does not depend on the whole path of $\alpha$, but only on $\alpha_{\vert [0, N]} \in  \mathbb{A}^N$. For $\alpha \in \mathbb{A}^N$, we also write $\mathbb{P}^{\mu, \alpha, N}$ for the measure equivalent to $\mathbb{P}$ with 
$$\frac{d\mathbb{P}^{\mu, \alpha, N}}{d\mathbb{P}} := \mathcal{E}\left(\int_0^\cdot \sigma^{-1}b(s, X, \mu, \alpha_s)dW_s\right)_N.$$
In the following, we will consider $\mathbb{P}^{\mu, \alpha, N}$ as a measure on $(\Omega, \mathcal{F}_N)$.
By the construction of $\mathcal{Q}^N$, we have for $\alpha \in \mathbb{A}^N$ that $\sup_{\mu \in \mathcal{Q}} \sup_{\alpha \in \mathbb{A}^N} \mathbb{E}[(\frac{d\mathbb{P}^{\mu, \alpha,N}}{d\mathbb{P}})^p] \leq M^N_p < \infty$.\par
In the following, unless specified otherwise, if we need to explicitly metrize the product topology of metric spaces $(\mathcal{X}, d_\mathcal{X})$ and $(\mathcal{Y}, d_\mathcal{Y})$, we choose to work with the maximum metric $d_{\mathcal{X}, \mathcal{Y}}((x_1, y_1), (x_2, y_2)) := \max(d_\mathcal{X}(x_1, x_2), d_\mathcal{Y}(y_1, y_2))$.\par
As $\mathcal{Q} \times A$ is compact, we know that for any $(t, \zeta)$ fixed, $(\mu, a) \mapsto \sigma^{-1}b(t, X, \mu, a)$ is uniformly continuous. To quantify this, we introduce the modulus of continuity
$$\kappa^N: [0, N] \times \Omega \times [0, \infty), (t, \zeta, \delta) \mapsto \sup_{d_{\mathcal{Q}, A}((\mu, a),(\mu', a')) \leq \delta} \vert \sigma^{-1}b(t, X(\zeta), \mu, a) - \sigma^{-1}b(t, X(\zeta), \mu', a') \vert.$$
Since $\mathcal{Q} \times A$ is compact metrizable and thus separable, this supremum can be replaced by a countable supremum, ensuring that $\kappa^N$ is measurable in $(t, \zeta)$. 
As $\sigma^{-1}b$ is uniformly continuous for almost any fixed $(t, \zeta)$, we thus have $\lim_{\delta \searrow 0} \kappa^N(t, \zeta, \delta) = 0$ almost everywhere.

Let us take $\eta > 0$ arbitrarily. Since $\kappa^N(t,\zeta,\delta)$ also converges in $dt \times \mathbb{P}$ measure on $[0, N] \times \Omega$, there must exist some $\delta^N>0$ such that $dt \times \mathbb{P}(\lbrace (t, \zeta) \vert \kappa^N(t, \zeta, \delta^N) \geq \eta^{\frac{1}{2}} \rbrace) < \frac{\eta^2}{N M^N_2}$.
As $A$ is bounded, there exists some $C_A$ such that $\Vert a \Vert_A < C_A$ for all $a\in A$.  Let us consider $\alpha, \alpha' \in \mathbb{A}^N$ arbitrarily such that $\Vert \alpha- \alpha' \Vert_{\mathbb{A}^N} < \frac{\eta^2 (\delta^N)^2 }{2 N C_A M^N_2}$. Then,
$$\mathbb{E}\left[\int_0^N \Vert \alpha_s - \alpha'_s \Vert_A^2 ds\right] \leq 2 C_A \Vert \alpha - \alpha' \Vert_{\mathbb{A}^N} < \frac{\eta^2 (\delta^N)^2}{N M^N_2}$$
and for any $\mu \in \mathcal{Q}$, we have
$$\mathbb{E}^{\mathbb{P}^{\mu, \alpha, N}}\left[\int_0^N \Vert \alpha_s - \alpha'_s \Vert_A ds\right] \leq \mathbb{E}\left[\int_0^N \Vert \alpha_s - \alpha'_s \Vert_A^2 ds\right]^{\frac{1}{2}}\mathbb{E}\left[N \left(\frac{d\mathbb{P}^{\mu, \alpha, N}}{d\mathbb{P}}\right)^2\right]^{\frac{1}{2}} < \eta \delta^N.$$
By Markov inequality, we therefore have
$$
    dt \times \mathbb{P}^{\mu, \alpha, N}(\lbrace (t, \zeta) : \Vert \alpha_t(\zeta) - \alpha'_t(\zeta) \Vert_A \geq \delta^N \rbrace) \leq \frac{1}{\delta^N} \mathbb{E}^{\mathbb{P}^{\mu, \alpha, N}}\left[\int_0^N \Vert \alpha_s(\zeta) - \alpha'_s(\zeta) \Vert_A ds\right] < \eta.
$$
Further, we have
\begin{equation*}
    \begin{split}
        &dt \times \mathbb{P}^{\mu, \alpha, N}(\lbrace (t, \zeta) : \kappa^N(t, \zeta, \delta^N) \geq \eta^{\frac{1}{2}} \rbrace) \leq \mathbb{E}\left[\left(\frac{d \mathbb{P}^{\mu, \alpha, N}}{d\mathbb{P}}\right)^2\right]^{\frac{1}{2}} \mathbb{E}\left[\left(\int_0^N \mathbb{1}_{\lbrace\kappa^N(s, \zeta, \delta^N) \geq \eta^{\frac{1}{2}}\rbrace} ds \right)^2\right]^{\frac{1}{2}} \\
        \leq & \left(N M^N_2 (dt \times \mathbb{P})(\lbrace (t, \zeta) \vert \kappa^N(t, \zeta, \delta^N) \geq \eta^{\frac{1}{2}} \rbrace) \right)^{\frac{1}{2}}  < \eta.
    \end{split}
\end{equation*}
Now, consider any $(\mu, \alpha), (\mu', \alpha') \in \mathcal{Q} \times \mathbb{A}^N$ so that
$$d_{\mathcal{Q}, \mathbb{A}^N}((\mu, \alpha),(\mu', \alpha')) < \tilde{\delta}^N:= \min\left(\delta^N, \frac{\epsilon^2(\delta^N)^2}{2 T C_A M_2}\right).$$
By Proposition \ref{prop: tau mtrc} and Pinkser's inequality, we know that
\begin{equation*}
    \begin{split}
        & d_{\tilde{\tau}^N}(\mathbb{P}^{\mu, \alpha, N}, \mathbb{P}^{\mu', \alpha', N}) \leq d_{TV}(\mathbb{P}^{\mu, \alpha, N}, \mathbb{P}^{\mu', \alpha', N}) \leq \sqrt{\frac{1}{2}\mathcal{H}(\mathbb{P}^{\mu, \alpha, N}, \mathbb{P}^{\mu', \alpha', N})} \\
        = & \frac{1}{2} \sqrt{\mathbb{E}^{\mathbb{P}^{\mu, \alpha, N}}\left[\int_0^N \vert\sigma^{-1}b(s, X, \mu', \alpha'_s) - \sigma^{-1}b(s, X, \mu, \alpha_s)\vert^2 ds\right]}
    \end{split}
\end{equation*}
where the last line uses Girsanov's Theorem to represent $\frac{d\mathbb{P}^{\mu', \alpha', N}}{d\mathbb{P}^{\mu, \alpha, N}}$. Now, for such $(\mu, \alpha)$ and $ (\mu', \alpha')$, we have
\begin{equation*}
    \begin{split}
       &\mathbb{E}^{\mathbb{P}^{\mu, \alpha, N}}\left[\int_0^N \vert \sigma^{-1}b(s, X, \mu, \alpha_s) - \sigma^{-1}b(s, X, \mu', \alpha'_s) \vert^2 ds\right] \\
       \leq & \mathbb{E}^{\mathbb{P}^{\mu, \alpha, N}}\left[\int_0^N \mathbb{1}_{\lbrace \kappa^N(s, \zeta, \delta^N) < \eta^{\frac{1}{2}}, \Vert \alpha_s - \alpha'_s \Vert_A < \delta^N \rbrace} \vert \sigma^{-1}b(s, X, \mu, \alpha_s) - \sigma^{-1}b(s, X, \mu', \alpha'_s) \vert^2 ds\right] + 8 C^2 \eta \\
       \leq & \mathbb{E}^{\mathbb{P}^{\mu, \alpha, N}}\left[\int_0^N \mathbb{1}_{\lbrace \kappa^N(s, \zeta, \delta^N) < \eta^{\frac{1}{2}}\rbrace} \kappa^N(s, \zeta, \delta^N)^2 ds\right] + 8 C^2 \eta < (N + 8C^2)\eta
    \end{split}
\end{equation*}
and therefore, $d_{\tilde{\tau}^N}(\mathbb{P}^{\mu, \alpha, N}, \mathbb{P}^{\mu', \alpha', N}) < \frac{1}{2}\sqrt{(N + 8C^2)\eta}$. Now, for every $N \leq \tilde{N}$, consider $\eta$ so that $\frac{1}{2}\sqrt{(N + 8C^2)\eta} < \frac{\epsilon}{2}$ and fix a corresponding $\tilde{\delta}^N$.\par
Let us go back to the infinite time horizon setting. For any $(\mu, \alpha), (\mu', \alpha') \in \mathcal{Q} \times \mathbb{A}$ with
$$d_{\mathcal{Q},\mathbb{A}}((\mu, \alpha), (\mu', \alpha')) < e^{-\tilde{N}} \min_{N\leq \tilde{N}} \tilde{\delta}^N$$
we have for every $N \leq \tilde{N}$ that $d_{\mathcal{Q}, \mathbb{A}^N}((\mu, \alpha_{\vert[0, N]}),(\mu', \alpha'_{\vert[0, N]})) < \tilde{\delta}^N$. Further, by definition, $\tilde{\pi}^N(\mathbb{P}^{\mu, \alpha})$ is the law of $(\xi, \omega_{\vert [0, N]})$ under $\mathbb{P}^{\mu, \alpha}$. Accordingly on $(\Omega^N, \mathcal{F}^N)$, the measure $\tilde{\pi}^N(\mathbb{P}^{\mu, \alpha})$ coincides with $\mathbb{P}^{\mu, \alpha_{\vert [0, N]}, N}$.
Therefore, by the discussion above,
$$d_{\tilde{\tau}^N}(\tilde{\pi}^N(\mathbb{P}^{\mu, \alpha}), \tilde{\pi}^N(\mathbb{P}^{\mu', \alpha'})) = d_{\tilde{\tau}^N}(\mathbb{P}^{\mu, \alpha_{\vert[0, N]}, N}, \mathbb{P}^{\mu', \alpha'_{\vert[0, N]}, N}) < \frac{\epsilon}{2}.$$
and hence,
$$d_{\tilde{\tau}}(\mathbb{P}^{\mu, \alpha}, \mathbb{P}^{\mu', \alpha'}) = \sum_{N = 1}^{\tilde{N}} 2^{-N} d_{\tilde{\tau}^N}(\tilde{\pi}^N(\mathbb{P}^{\mu, \alpha}), \tilde{\pi}^N(\mathbb{P}^{\mu', \alpha'})) + \frac{\epsilon}{2} < \epsilon.$$
As $\epsilon$ was arbitrary, this shows uniform continuity. For the second statement, note that by Lemma \ref{lem: cntrcont}, the set-valued maps $\mathbb{A}$ and thus, (by Lemma \ref{lem: uhc prod}) $\overline{\mathbb{A}}$ are both metrically upper hemicontinuous. The same follows for the composition by using Lemma \ref{lem: uhc comp}.
\end{proof}

\subsection{Single-valued case: extended mean field games}

We are now ready to prove the first of our two existence results. 
In order to deal with the case where players can (also) interact through their control variables, we will make the stronger assumption that the Hamiltonian admits a unique maximizer. 
In particular, $\mathbb{A}(\mu, \nu)$ will only consist of one element (i.e.\ Assumption \ref{asmp: sngl} is satisfied).
Thus, we view $\mathbb{A}: \mathcal{Q} \times \mathcal{M} \rightarrow \mathbb{A}$ as a function in the usual sense. It is immediate that in this case, Lemma \ref{lem: P uc} implies that $\mathfrak{P} \circ \overline{\mathbb{A}}: \mathcal{Q} \times \mathcal{M} \rightarrow \tilde{\mathcal{Q}}$ is actually a continuous map.
\begin{proof}[Theorem \ref{thm: exst E}]
Recall from section \ref{sect: cmp cnt} the map $\mathfrak{u}: \mathcal{P}(\Omega) \rightarrow \mathcal{P}(\mathcal{C}), \mathbb{Q} \mapsto \mathbb{Q} \circ X^{-1}$ which is continuous with respect to $\tilde{\tau}$ and $\tau$. Further, by construction, the image of $\tilde{\mathcal{Q}}$ under $\mathfrak{u}$ is $\mathcal{Q}$.\par
Let us further define $\mathfrak{v}: \mathcal{Q} \times \mathbb{A} \rightarrow \mathcal{M}, (\mu, \alpha) \mapsto \delta_{\mathfrak{P}(\mu, \alpha) \circ \alpha_t^{-1}} dt$. We also want to show that this map is sequentially continuous. Let us consider any $(\mu^n, \alpha^n)\rightarrow(\mu,\alpha)$ in $\mathcal{Q} \times \mathbb{A}$. By Lemma \ref{lem: P uc} and the properties of the total variation metric, we know for any $t$ that $d_{TV}(\mathfrak{P}(\mu^n, \alpha^n) \circ (\alpha^n_t)^{-1}, \mathfrak{P}(\mu, \alpha)\circ (\alpha^n_t)^{-1}) \leq d_{TV}(\mathfrak{P}(\mu^n, \alpha^n), \mathfrak{P}(\mu, \alpha)) \rightarrow 0$. Further, for almost every $t \geq 0$, we have $\alpha^n_t \rightarrow \alpha_t$ and thus $\mathfrak{P}(\mu, \alpha) \circ (\alpha^n_t)^{-1} \rightarrow \mathfrak{P}(\mu, \alpha) \circ (\alpha_t)^{-1}$ in $\mathcal{P}(A)$ with respect to weak convergence. As weak convergence is metrized by the Lévy-Prokhorov distance which is bounded by the total variation metric, we even have for almost every $t$ that $\mathfrak{P}(\mu^n, \alpha^n) \circ (\alpha^n_t)^{-1} \rightarrow \mathfrak{P}(\mu, \alpha) \circ (\alpha_t)^{-1}$.\par
By \cite[Proposition 3.5]{JacodMemin}, we have for any $T>0$ that
$$(\delta_{\mathfrak{P}(\mu^n, \alpha^n) \circ (\alpha^n_t)^{-1}} dt)_{\vert [0, T]} \rightarrow (\delta_{\mathfrak{P}(\mu, \alpha) \circ (\alpha_t)^{-1}} dt)_{\vert [0, T]}$$
in $\mathcal{M}^T$ with respect to $\mathcal{S}^T$. As this holds for all $T > 0$, by the definition of $\mathcal{S}$, we have
$$\delta_{\mathfrak{P}(\mu^n, \alpha^n) \circ (\alpha^n_t)^{-1}} dt\rightarrow \delta_{\mathfrak{P}(\mu, \alpha) \circ (\alpha_t)^{-1}} dt$$
in $\mathcal{M}$ with respect to the topology $\mathcal{S}$. This shows that $\mathfrak{v}$ is continuous as well.\par
Thus, the composition
$$(\mathfrak{u} \circ \mathfrak{P}, \mathfrak{v}) \circ \overline{\mathbb{A}}:\mathcal{Q} \times \mathcal{M} \rightarrow \mathcal{Q} \times \mathcal{M}, (\mu, \nu) \mapsto (\mathbb{P}^{\mu, \mathbb{A}(\mu, \nu)} \circ X^{-1}, \delta_{\mathbb{P}^{\mu, \mathbb{A}(\mu, \nu)} \circ (\mathbb{A}(\mu, \nu)_t)^{-1}})$$
is continuous.\par
As we explain below in remark \ref{rmk: lchtvs}, $\mathcal{Q}$ and $\mathcal{M}$ can both be embedded as subspaces of Hausdorff locally convex topological vector spaces. As $\mathcal{Q} \times \mathcal{M}$ is convex and by Propositions \ref{prop: inft Q top} and \ref{prop: inft M top} is compact as well, by Schauder-Tychonoff's fixed point Theorem \cite[Theorem §2]{Tychonoff}, $(\mathfrak{u} \circ \mathfrak{P}, \mathfrak{v}) \circ \overline{\mathbb{A}}$ must admit a fixed point $(\hat{\mu}, \hat{\nu})$. Such $(\hat{\mu}, \hat{\nu})$ together with its optimal control $\mathbb{A}(\hat{\mu}, \hat{\nu})$ form a solution of the discounted infinite horizon extended mean-field game.
\end{proof}
\begin{remark}\label{rmk: lchtvs}
    Let $\overline{\mathcal{Q}}$ be the space of all finite signed Borel measures on $\mathcal{C}$ and let $\overline{\mathcal{M}}$ be the space of all $\sigma$-finite signed Borel measures on $[0, \infty) \times \mathcal{P}(A)$. For $\overline{\mathcal{Q}}$, observe that for any $f \in B(\mathcal{C})$, the map $\mu \mapsto \int_{\mathcal{C}} f d\mu$ is linear, and $p_f: \overline{\mathcal{Q}} \rightarrow \mathbb{R}, \mu \mapsto \vert \int_{\mathcal{C}} f d\mu \vert$ thus defines a seminorm. The family $(p_f)_{f}$ for $f\in B(\mathcal{C})$ of seminorms defines a locally convex vector space topology on $\overline{\mathcal{Q}}$. It is immediate that under this topology, the induced topology on $\mathcal{Q} \subset \overline{\mathcal{Q}}$ is exactly $\tau$. Further, by using the Hahn-Jordan decomposition and Proposition \ref{prop: ext}, we can see that if for some $\mu \in \overline{\mathcal{Q}}$, we have $p_f(\mu) = 0$ for any $f \in B(\mathcal{C})$, we must have $\mu = 0$. Thus, the induced topology in $\overline{\mathcal{Q}}$ is Hausdorff.\par
    The same can be done with $\overline{\mathcal{M}}$ and seminorms $p_f(\nu) = \vert \int_0^\infty f(t, q) \nu_t(dq)dt\vert$ over all bounded, compactly supported and measurable $f$ that are continuous in $q$ to define a compatible Hausdorff locally convex Hausdorff topology. 
\end{remark}

\subsection{Set-valued case: non extended mean field games}
We now turn to the existence result when there is no interaction through the control.
In this case, we allow there to be several maximizers of the Hamiltonian.
Before presenting the proof, let us prove a necessary result that follows from the convenient structure of the values of $\mathbb{A}$.

\subsubsection{Full subsets of $\mathbb{A}$}

In the set-valued case, we would like the set-valued map $\mathfrak{P} \circ \overline{\mathbb{A}}$ to admit compact values. 
In general, once $\mathbb{A}$ does not admit almost everywhere single values, it is not compact but only closed! 
As $\mathfrak{P} \circ \overline{\mathbb{A}}$ maps into a compact space, it suffices to show that its values are closed.\par

First, we define measurability of set-valued maps. 
\begin{definition}\label{def: wmsbl}
    We call a set-valued map $\phi: \mathcal{X} \twoheadrightarrow \mathcal{Y}$ from a measure space $(\mathcal{X}, \Sigma)$ into a metric space $(\mathcal{Y}, d)$ weakly measurable if for each open subset $O \subset \mathcal{Y}$, we have
    $$\lbrace x \in \mathcal{X} \vert \phi(x) \cap O \neq \emptyset \rbrace \in \Sigma.$$
\end{definition}
\begin{remark}
    If we replace the open sets by closed set, then $\phi$ is called measurable set-valued function. In case the values of $\phi$ are compact, which is the case in most of our following applications, by \cite[Lemma 18.2]{IDA}, the notions of weak measurability and measurability are equivalent.
 \end{remark}
In the following, we equip the space $[0, \infty) \times \Omega$ with the progressive $\sigma$ algebra.
\begin{definition}
    We call a subset $\mathbb{V} \subset \mathbb{A}$ full if there exists a weakly measurable, nonempty, and compact set-valued function $V: [0, \infty) \times \Omega \twoheadrightarrow A$ with
    $$\mathbb{V} = \lbrace \alpha \in \mathbb{A} \vert \alpha_t(\zeta) \in V(t, \zeta), dt\times d\mathbb{P}\,\, a.e.\ \rbrace.$$
    Such a function $V$ will be called a spanning function.
\end{definition}
By the measurable maximum Theorem, \cite[Theorem 18.19]{IDA}, we know that for fixed $\mu, \nu$ the set-valued map $(t, x) \mapsto \mathcal{A}(t, x, \mu, Z^{\mu, \nu}_t)$ is compact set-valued and measurable, thus weakly measurable as well. In particular, this proves that it is a spanning function and for each $\mu$ and $\nu$, the set $\mathbb{A}(\mu, \nu)$ is thus full.

We are interested in the behaviour of full sets under the map $\mathfrak{P}$. Note that $\mathfrak{P}$ always maps into $\tilde{\mathcal{Q}}$. Let us first consider $\mathfrak{P}_T = \tilde{\pi}^T \circ \mathfrak{P}$ which always maps into $\tilde{\mathcal{Q}}^T \subset \mathcal{P}(\Omega^T)$. In the following, $\mu$ is a fixed element in $\mathcal{Q}$.

\begin{lemma}\label{lem: str clsd}
    For any full subset $\mathbb{V}$ with spanning function $V$, the subset $\mathfrak{P}_T(\mu, \mathbb{V})$ is closed with respect to the total variation metric on $\mathcal{P}(\Omega^T)$.
\end{lemma}
\begin{proof}
Consider an arbitrary sequence $(\alpha^n)_{n \geq 1} \subset \mathbb{V}$ such that $\mathfrak{P}_T(\mu, \alpha^n) \rightarrow \hat{\mathbb{P}}$ in total variation for some $\hat{\mathbb{P}} \in \mathcal{P}(\Omega^T)$.
Note that for any $\alpha$, by assumption, $\mathfrak{P}_T(\mu, \alpha)$ only depends on $\alpha_{\vert[0, T]}$. Hence, $\mathfrak{P}_T$ can also be considered as a map on $\Omega \times \mathbb{A}^T$. Let us thus consider the subset $\mathbb{V}^T = \lbrace \alpha_{\vert[0, T]} \vert \alpha \in \mathbb{V}\rbrace \subset \mathbb{A}^T$ which we can also write as 
$$\mathbb{V}^T = \lbrace \alpha \in \mathbb{A}^T \vert \alpha_t(\zeta) \in V(t, \zeta), dt\times d\mathbb{P} a.e.\ (t, \zeta) \in [0, T] \times \Omega \rbrace.$$
By construction, $\mathfrak{P}_T$ maps into $\tilde{\mathcal{Q}}^T$ which is a closed subset of $\mathcal{P}(\Omega^T)$ in $\tilde{\tau}^T$. As convergence in total variation implies convergence in $\tilde{\tau}^T$ as well, we must have $\hat{\mathbb{P}} \in \tilde{\mathcal{Q}}^T$. In particular, $\hat{\mathbb{P}}$ is still equivalent to $\mathbb{P}^T$. Since $\mathbb{L}^\infty(\Omega^T)$ is the dual space of $\mathbb{L}^1(\Omega^T)$, if we write $B^\infty_0(1)$ the unit ball in $\mathbb{L}^\infty(\Omega^T)$ around $0$, we can write the $\mathbb{L}^1$ norm as
\begin{equation*}
    \begin{split}
        &\mathbb{E}\left[\left\vert \frac{d\mathfrak{P}_T(\mu, \alpha^n)}{d\mathbb{P}^T} - \frac{d\hat{\mathbb{P}}}{d\mathbb{P}^T}\right\vert \right]=  \sup_{\phi \in B^\infty_0(1)} \mathbb{E} \left[\phi \left(\frac{d\mathfrak{P}_T(\mu, \alpha^n)}{d\mathbb{P}^T} - \frac{d\hat{\mathbb{P}}}{d\mathbb{P}^T}\right)\right] \\
        =& \sup_{\phi \in B^\infty_0(1)} \int_{\Omega^T} \phi d(\mathfrak{P}_T(\mu, \alpha^n) - \hat{\mathbb{P}}) = d_{TV}(\mathfrak{P}_T(\mu, \alpha^n), \hat{\mathbb{P}}).
    \end{split}
\end{equation*}
Now since by assumption the total variation converges, it follows that $\frac{d\mathfrak{P}_T(\mu, \alpha^n)}{d\mathbb{P}^T}$ converges in $\mathbb{L}^1$ towards $\frac{d\hat{\mathbb{P}}}{d\mathbb{P}^T}$ for $n \rightarrow \infty$. Recall that $\sup_{n \geq 1} \mathbb{E}[(\frac{d\mathfrak{P}_T(\mu, \alpha^n)}{d\mathbb{P}^T})^2] \leq M^T_2 < \infty$, thus by Vitali's convergence Theorem this convergence also holds in $\mathbb{L}^{\frac{3}{2}}$.\par
Let us write $\beta^n_t := \sigma^{-1}b(t, X, \mu, \alpha^n_t)$. Further, for $\mathfrak{L}$ defined as below in Lemma \ref{lem: stolog}, we use the predictable representation property to find a $\mathbb{R}^d$ valued process $\beta$ such that $\int_0^\cdot \beta dW_s = \mathfrak{L}(\mathbb{E}[\frac{d\hat{\mathbb{P}}}{\mathbb{P}^T} \vert \mathcal{F}^T_\cdot])$. By Doob's inequality and \cite[Remark 12.4.5.]{CE}, the $\mathbb{L}^{\frac{3}{2}}$ convergence above implies that the $\mathcal{E}(\int_0^\cdot \beta^n_s dW_s)$ converge to $\mathcal{E}(\int_0^\cdot \beta_s dW_s)$ in the Émery topology. Thus, by Lemma \ref{lem: stolog}, $\int_0^\cdot \beta^n_sdW_s$ converges in the Émery topology to $\int_0^\cdot \beta_sdW_s$ and hence by \cite[Proposition 2.10]{Kardaras}, $\int_0^\cdot \Vert \beta_s- \beta^n_s\Vert^2 ds$ converges to zero in the Émery topology. This implies that $\int_0^T \Vert \beta_s-\beta^n_s\Vert^2ds$ and by Hölder's inequality also $\int_0^T \Vert \beta_s - \beta^n_s \Vert ds$ convergences in probability to zero. By Jensen's inequality applied to $\frac{1}{T}dt$, it follows that
$$\mathbb{E}\left[\int_0^T \max(1, \Vert \beta_s - \beta^n_s\Vert)ds\right] \leq \mathbb{E}\left[\max\left(T, \int_0^T \Vert \beta_s - \beta^n_s \Vert ds\right)\right] \rightarrow 0$$
so that $\beta^n$ converges in $dt \times\mathbb{P}$ measure to $\beta$, and once we pass to a subsequence, we can thus assume $\beta^n_t \rightarrow \beta_t$ $dt \times d\mathbb{P}^T$ almost everywhere on $[0, T] \times \Omega$.\par
Let us define the set-valued map
$$\mathfrak{B}:[0, T] \times \Omega \rightarrow 2^{A},\,\, (t, \zeta) \mapsto \lbrace a \in V(t, \zeta) \vert \sigma^{-1}b(t, X(\zeta), \mu, a) = \beta_t(\zeta) \rbrace.
$$
As $A$ is compact, for any $(t, \zeta)$, the map $a \mapsto \sigma^{-1}b(t, X(\zeta), \mu, a)$ is closed as it is continuous.\par
Now, there exists a measurable set $E \subset [0, T] \times \Omega$ such that $(dt \times \mathbb{P})(E) = 0$ and on the set $([0, T] \times \Omega) \setminus E$, we have $\alpha^n_t(\zeta) \in V(t, \zeta)$ for any $n$ and $\beta^n_t(\zeta) \rightarrow \beta_t(\zeta)$. Thus, since $V(t, \zeta)$ is closed, on $([0, T] \times \Omega) \setminus E$, there exists $a \in V(t, \zeta)$ such that $\sigma^{-1}b(t, X(\zeta), \mu, a) = \beta_t(\zeta)$ making $\mathfrak{B}(t, \zeta)$ almost everywhere non empty. By Filippov's implicit function Theorem \cite[Theorem 18.17]{IDA}, there exists a measurable selector $\gamma:[0, T] \times \Omega \rightarrow A$ of $\mathfrak{B}$ such that $\sigma^{-1}b(t, X(\zeta), \mu, \gamma_t(\zeta)) = \beta_t(\zeta)$ almost everywhere.\par
Such $\gamma$ needs to lie in $\mathbb{V}^T$. By definition, there exists $\alpha \in \mathbb{A}$ so that $\alpha_{\vert [0, T]} = \gamma$ and therefore, $\mathfrak{P}_T(\mu, \gamma) = \mathfrak{P}_T(\mu, \alpha) = \hat{\mathbb{P}}$ which shows that $\mathfrak{P}_T(\mu, \mathbb{V})$ is closed in total variation.
\end{proof}

Later, we will relate this result to closedness under $\tilde{\tau}^T$ as well. 
Images under $\mathfrak{P}$ of full subsets have the desirable property that closedness under $\tilde{\tau}$ can be characterized locally. This is not true for general subsets of $\tilde{\mathcal{Q}}$.
\begin{lemma}\label{lem: full clsd}
    Let $\mathbb{V}$ be a full subset with spanning function $V$. Then, for any $\mu \in \mathcal{Q}$, we have that $\mathfrak{P}(\mu, \mathbb{V})$ is closed if for any $T > 0$, the sets $\tilde{\pi}^T(\mathfrak{P}(\mu, \mathbb{V}))$ are closed in $\mathcal{P}(\Omega^T)$.
\end{lemma}
\begin{proof}
Let $\mathbb{V}$ and $\mu$ be arbitrary such that for any $T > 0$, the sets $\tilde{\pi}^T(\mathfrak{P}(\mu, \mathbb{V}))$ are closed. Consider an arbitrary sequence $(\alpha^n)_{n \geq 1}\subset \mathbb{V}$ such that $\mathfrak{P}(\mu, \alpha^n ) \rightarrow \hat{\mathbb{P}}$ for some $\hat{\mathbb{P}}\in \tilde{\mathcal{Q}}$. By definition of our topology $\tilde{\tau}$, for any $N \geq 1$, we have $\tilde{\pi}^N(\mathfrak{P}(\mu, \alpha^n)) \rightarrow \tilde{\pi}^N(\hat{\mathbb{P}})$.
Since $\hat{\mathbb{P}} \in \tilde{\mathcal{Q}} \subset\mathfrak{E}$ is locally equivalent to $\mathbb{P}$, by \cite[Theorem 3.4]{JS}, there is a $\mathbb{P}$ martingale $D$ such that $\frac{d \hat{\mathbb{P}}_{\vert[0, T]}}{d\mathbb{P}_{\vert[0, T]}} = D_T$ for all $T\geq 0$. 
Just as in the proof of Lemma \ref{lem: str clsd}, we can find a progressively measurable $\mathbb{R}^d$ valued $\beta$ such that $\mathcal{E}(\int_0^\cdot \beta_s dW_s) = D$.\par
To show that $\hat{\mathbb{P}} \in \mathfrak{P}(\mu, \mathbb{V})$, we need to find $\alpha \in \mathbb{V}$ so that $\sigma^{-1}b(t, X(\zeta), \mu, \alpha_t(\zeta)) = \beta_t(\zeta)$ holds $dt \times \mathbb{P}$ a.e.\ In particular, if this holds, then for any $N \geq 1$, we have $\frac{d\mathfrak{P}(\mu, \alpha)_{\vert \mathcal{F}_N}}{d\mathbb{P}_{\vert \mathcal{F}_N}} = \frac{d\hat{\mathbb{P}}_{\vert \mathcal{F}_N}}{d\mathbb{P}_{\vert \mathcal{F}_N}}$. Since this implies $\tilde{\pi}^N(\mathfrak{P}(\mu, \alpha)) =\tilde{\pi}^N(\hat{\mathbb{P}})$, by Proposition \ref{prop: ext}, we then even get $\mathfrak{P}(\mu, \alpha) = \hat{\mathbb{P}}$.\par
To find such $\alpha$, note that as we know that for any $N\geq 1$, the set $\tilde{\pi}^N(\mathfrak{P}(\mu, \mathbb{V}))$ is closed, there must be $\alpha^N \in \mathbb{V}$ so that $\tilde{\pi}^N(\mathfrak{P}(\mu, \alpha^N)) = \tilde{\pi}^N(\hat{\mathbb{P}})$. In particular, this shows that on $[0, N] \times \Omega$, we have $dt \times \mathbb{P}$ a.e.\ that $\sigma^{-1}b(t, X(\zeta), \mu, \alpha^N_t(\zeta)) = \beta_t(\zeta)$. Now, let us define $\alpha \in \mathbb{A}$ so that for any $N\geq 1$, we have $\alpha_{\vert [N-1, N)} = \alpha^N_{\vert [N-1, N)}$. As $\mathbb{V}$ is full, we must still have $\alpha \in \mathbb{V}$. Further, this $\alpha$ satisfies $dt \times \mathbb{P}$ a.e., $\sigma^{-1}b(t, X(\zeta), \mu, \alpha_t(\zeta)) = \beta_t(\zeta)$.
\end{proof}

\subsubsection{Continuity of the stochastic logarithm}
Let us proof the properties about the stochastic logarithm that we have used above. For a fixed finite time horizon $T>0$, we define it as follows. Let $\mathbb{M}$ be the space of continuous local martingales on $[0, T]$ starting at zero and $\mathbb{M}^+$ the space of the ones that are positive and start at one at time zero. Then, we define the stochastic logarithm as $\mathfrak{L}: \mathbb{M}^+ \rightarrow \mathbb{M}, M \mapsto \int_0^\cdot \frac{1}{M_s}dM_s$. Then, considering the stochastic exponential as a map $\mathcal{E}:\mathbb{M} \rightarrow \mathbb{M}^+$, an application of Itô's formula immediately shows that the stochastic logarithm $\mathfrak{L}$ is the inverse of $\mathcal{E}$.\par
For our continuity result, we equip $\mathbb{M}^+$ and $\mathbb{M}$ with the Émery topology (or also known as the semimartingale topology) as in \cite[Definition 12.4.3.]{CE}. 
\begin{lemma}\label{lem: stolog}
    The stochastic logarithm $\mathfrak{L}$ is continuous with respect to the Émery topology.
\end{lemma}
\begin{proof}
Consider any $(M^n)_{n \geq 1} \subset \mathbb{M}^+$ that converge against some $M \in \mathbb{M}^+$ in the Émery topology. By \cite[Lemma 12.4.7.]{CE}, this implies that the $D^n$ also converge to $D$ in the ucp topology as in \cite[Definition 12.4.1.]{CE}. Note that we can consider $M^n$ and $M$ as $\mathcal{C}^T_+ := \lbrace \omega \in \mathcal{C}^T\vert \omega > 0\rbrace$ valued random variables so that ucp convergence convergence is the same as convergence in probability with respect to the supremum metric. Under this metric, the map $\mathcal{C}^T_+ \rightarrow \mathcal{C}^T, \omega_\cdot \mapsto \frac{1}{\omega_\cdot}$ is continuous. Thus, by the continuous mapping Theorem, $\frac{1}{M^n}$ converges to $\frac{1}{M}$ in ucp. Thus, by \cite[Theorem 5.3]{Protter85}, $\mathfrak{L}(M^n)$ converges in the Émery topology to $\mathfrak{L}(M)$.
\end{proof}

\subsubsection{Finishing the proof of Theorem \ref{thm: exst NE}}
Let us now prove our second existence result. This one only requires the convexity assumption \ref{asmp: NEMFG} and does not need to assume that the Hamiltonian can be uniquely maximized, but we need to further assume that there is no interaction through the law of the control variable. In this framework, reward and value function can be rewritten as
$$J^{\mu}(\alpha) := \mathbb{E}^{\mu, \alpha}\left[\int_0^\infty e^{-\lambda s} f(s, X, \mu, \alpha_s)ds\right], \; \; V^{\mu} := \sup_{\alpha \in \mathbb{A}} J^{\mu}(\alpha).$$
\begin{proof}[Theorem \ref{thm: exst NE}]
By Lemma \ref{lem: P uc}, we already know that $\mathfrak{P} \circ \overline{A}: \mathcal{Q} \rightarrow \tilde{\mathcal{Q}}$ is metrically upper hemicontinuous. By Lemma \ref{lem: str clsd}, we have for any $\mu \in \mathcal{Q}$ and any $T > 0$ that $\mathfrak{P}_T(\mu, \mathbb{A}(\mu))$ is closed in total variation. First, we would like to show that it is closed in $\tilde{\tau}^T$ as well.\par
To do so, we first show that $\mathfrak{P}_T(\mu, \mathbb{A}(\mu))$ is convex for any $\mu$ fixed. Consider any $\alpha^1, \alpha^2 \in \mathbb{A}(\mu)$ and any $\lambda \in [0, 1]$. Then, $\mathbb{P}^\lambda := \lambda \mathfrak{P}_T(\mu, \alpha^1) + (1- \lambda)  \mathfrak{P}_T(\mu, \alpha^2)$ defines a probability measure equivalent to $\mathbb{P}^T$ with density
$$\frac{d\mathbb{P}^\lambda}{d\mathbb{P}^T} = \lambda \mathcal{E}\left(\int_0^\cdot \sigma^{-1}b(s, X, \mu, \alpha^1_s)dW_s\right)_T + (1-\lambda) \mathcal{E}\left(\int_0^\cdot \sigma^{-1}b(s, X, \mu, \alpha^2_s)dW_s\right)_T.$$
Just as in previous proofs, we can find a predictable $\beta$ such that $\frac{d\mathbb{P}^\lambda}{d\mathbb{P}^T} = \mathcal{E}(\int_0^\cdot \beta_s dW_s)$. 
Indeed, by \cite{tarpodual} (see the proof of Theorem 3.10 therein) such $\beta$ is explicitly given as a convex combination by
\begin{equation*}
    \begin{split}
        \beta_t = \tilde{\lambda}_t \sigma^{-1}b(t, X, \mu, \alpha^1_t) + (1-\tilde{\lambda}_t) \sigma^{-1}b(t, X, \mu, \alpha^2_t)
    \end{split}
\end{equation*}
a.e.\ where we defined the $[0,1]$ valued process $$\tilde{\lambda}_t := \frac{\lambda \mathcal{E}(\int_0^\cdot \sigma^{-1}b(s, X, \mu, \alpha^1_s) dW_s)}{\lambda \mathcal{E}(\int_0^\cdot \sigma^{-1}b(s, X, \mu, \alpha^1_s) dW_s)+ (1-\lambda) \mathcal{E}(\int_0^\cdot \sigma^{-1}b(s, X, \mu, \alpha^2_s)dW_s)}.$$
We have $\tilde{\lambda}_t b(t, X, \mu, \alpha^1_t) + (1-\tilde{\lambda}_t) b(t, X, \mu, \alpha^2_t) \in B(t, x, \mu, z)$ a.e.\ by assumption \ref{asmp: NEMFG}. Since $\sigma$ is independent of $a$, a.e., $\lbrace a \in \mathcal{A}(t, X, \mu, z) \vert \sigma^{-1}b(t, X(\zeta), \mu, a) = \beta_t(\zeta) \rbrace$ is almost everywhere non empty. By Filippov's implicit function Theorem \cite[Theorem 18.17]{IDA}, there exists a progressively measurable $\gamma \in \mathbb{A}(\mu)$ such that $\beta_t(\zeta) = \sigma^{-1}b(t, X(\zeta), \mu, \gamma_t(\zeta))$ almost everywhere and thus $\mathbb{P}^\lambda = \mathfrak{P}_T(\mu, \gamma) \in \mathfrak{P}_T(\mu, \mathbb{A}(\mu))$, proving that $\mathfrak{P}_T(\mu, \mathbb{A}(\mu))$ is convex.\par
To relate convexity to closedness in $\tilde{\tau}^T$, we consider the map
$$\mathcal{I}^T: \tilde{\mathcal{Q}}^T \rightarrow \left\lbrace Z \in \mathbb{L}^1(\mathbb{P}^T) \vert Z > 0\,\, \text{ a.s.}, \mathbb{E}[Z \vert \sigma(\xi)] = 1, \int Z^2 d\mathbb{P}^T\leq M^T, \int Z^{-1} d\mathbb{P}^T \leq M^T\right\rbrace$$
given by $\mathcal{I}^T(\mathbb{Q}) := \frac{d\mathbb{Q}}{d\mathbb{P}^T}$ with $\tilde{\mathcal{Q}}^T$. The image space is considered as a subspace of $\mathbb{L}^1(\mathbb{P}^T)$ equipped with the weak topology. By the definition of $\tilde{\tau}^T$, it is immediate that $\mathcal{I}^T$ is an homeomorphism. Further, both the image as well as the domain are affine subspaces of vector spaces, and both $\mathcal{I}$ as well as its inverse are linear.\par
By our previous discussion, we know that $\mathcal{I}^T(\mathfrak{P}_T(\mu, \mathbb{A}(\mu)))$ is strongly $\mathbb{L}^1$ closed and convex. Therefore, by Mazur's Lemma, it is weakly $\mathbb{L}^1$ closed as well which implies that $\mathfrak{P}_T(\mu, \mathbb{A}(\mu))$ is closed in $\tilde{\tau}^T$. By Lemma \ref{lem: full clsd}, this further implies that $\mathfrak{P}(\mu, \mathbb{A}(\mu))$ is closed under $\tilde{\tau}$. As a closed subset of a compact space, it is in fact compact as well.\par
Recall that $\mathfrak{u}: \tilde{\mathcal{Q}} \rightarrow \mathcal{Q}, \mathbb{Q} \mapsto \mathbb{Q} \circ X^{-1}$ is a continuous map on a closed domain, and thus a closed map as well. Further, take any $\mathbb{Q}^1, \mathbb{Q}^2 \in \tilde{\mathcal{Q}}$ and $\lambda \in [0, 1]$. For any $A \in \mathcal{F}$, we have
\begin{equation*}
    \begin{split}
        &\left(\lambda \mathfrak{u}(\mathbb{Q}^1) + (1 - \lambda) \mathfrak{u}(\mathbb{Q}^2)\right)[A] = \lambda \mathfrak{u}(\mathbb{Q}^1)[A] + (1 - \lambda) \mathfrak{u}(\mathbb{Q}^2)[A]\\
        =& \lambda \mathbb{Q}^1[X^{-1}(A)] + (1 - \lambda) \mathbb{Q}^2[X^{-1}(A)] = \left(\lambda \mathbb{Q}^1 + (1 - \lambda) \mathbb{Q}^2\right)[X^{-1}(A)] \\
        =& \mathfrak{u}\left(\lambda \mathbb{Q}^1 + (1 - \lambda) \mathbb{Q}^2\right)[A]
    \end{split}
\end{equation*}
showing that the measures $\lambda \mathfrak{u}(\mathbb{Q}^1) + (1 - \lambda) \mathfrak{u}(\mathbb{Q}^2)$ and $\mathfrak{u}(\lambda \mathbb{Q}^1 + (1 - \lambda) \mathbb{Q}^2)$ coincide. By Lemma \ref{lem: uhc comp}, $\mathfrak{u} \circ \mathfrak{P} \circ \overline{\mathbb{A}}: \mathcal{Q} \rightarrow \mathcal{Q}$ is a upper hemicontinuous map with compact and convex values.\par
As shown in remark \ref{rmk: lchtvs}, $\mathcal{Q}$ is a compact convex subset of a locally convex Hausdorff topological vector space. Thus, by Kakutani-Fan-Glicksbergs's fixed point Theorem \cite[Corollary 17.55]{IDA}, $\mathfrak{u} \circ \mathfrak{P} \circ \overline{\mathbb{A}}$ must admit a fixed point $\hat{\mu}$. By construction, such a fixed point must have an optimal control $\hat{\alpha} \in \mathbb{A}(\hat{\mu})$ so that $\mathbb{P}^{\hat{\mu}, \hat{\alpha}} \circ X^{-1} = \mathfrak{u}(\mathfrak{P}(\hat{\mu}, \hat{\alpha})) = \hat{\mu}$, which proves that it is a solution to the discounted infinite time horizon mean field game.
\end{proof}

\subsection{Proof of the uniqueness result}
Recall from Proposition \ref{prop.optim.charac} that for any given $(\mu, q)$, optimal controls are given by the $\alpha$ that satisfy $dt \times \mathbb{P}$ a.e., $\alpha_t \in \mathcal{A}(t, X, \mu, \tilde Z^{\mu,q}_t)$ with $(Y^{\mu, q}, Z^{\mu, q})$ being a solution of \eqref{inft BSDE}. If the Hamiltonian is everywhere uniformly maximizable, the optimal control is thus unique. For $(Y^{\mu, q}_t, Z^{\mu, q}_t) := (e^{-\lambda t} \tilde{Y}^{\mu, q}_t, e^{-\lambda t}\tilde{Z}^{\mu, q}_t)$, we can see that together with the optimal control found above, it will be a solution to the infinite horizon BSDE \eqref{inft cntrl bsde} and $Y^{\mu, q}$ is the remaining utility process for the optimal control. Note that as we have assumed $b$ is independent of $\mu$, the change of measure $\mathbb{P}^{\mu, \alpha}$ is actually already fully determined by $\alpha$. For notational simplicity, we write $\mathbb{P}^\alpha = \mathbb{P}^{\mu, \alpha}$ and $\mathbb{E}^\alpha$ for the expectation taken with it.
\begin{proof}[Theorem \ref{thm: unq}]
    Having assumed that the cost is separated and satisfies the Lasry-Lions monotonicity condition, uniqueness follows by standard argument, see e.g.\ \cite{Carmona15}, with minor modifications to account for the infinite horizon. Let us consider any two mean field equilibria $(\mu^1, q^1,\alpha^1)$ and $(\mu^2, q^2,\alpha^2)$. 
    We write $(Y^1, Z^1) = (Y^{\mu^1, q^1}, Z^{\mu^1, q^1})$ and $(Y^2, Z^2) = (Y^{\mu^2, q^2}, Z^{\mu^2, q^2})$. As they both solve BSDE \eqref{inft cntrl bsde}, we have for any $T \geq 0$ that
\begin{equation*}
    \begin{split}
        \int Y^1_0 - Y^2_0 d\upsilon &=  \mathbb{E}^{\alpha^1}\left[Y^1_T - Y^2_T + \int_0^T e^{-\lambda s} (f(s, X, \mu^1, q^1_s, \alpha^1_s) - f(s, X, \mu^2, q^2_s, \alpha^2_s)) ds\right]\\
        &\quad +\mathbb{E}^{\alpha^1}\left[\int_0^T(Z^2_s)^\top(\sigma^{-1}b(s, X, \alpha^1_s) - \sigma^{-1}b(s, X, \alpha^2_s))ds\right]\\
        &= \mathbb{E}^{\alpha^2}\left[Y^1_T - Y^2_T + \int_0^T e^{-\lambda s} (f(s, X, \mu^1, q^1_s, \alpha^1_s) - f(s, X, \mu^2, q^2_s, \alpha^2_s)) ds\right]\\
        &\quad +\mathbb{E}^{\alpha^2}\left[\int_0^T(Z^1_s)^\top(\sigma^{-1}b(s, X, \alpha^1_s) - \sigma^{-1}b(s, X, \alpha^2_s))ds\right].
    \end{split}
\end{equation*}
where the left hand size takes on this form since both $\mathbb{P}^{\alpha^1}$ and $\mathbb{P}^{\alpha^2}$ are equal to $\upsilon$ on $\mathcal{F}_0$. Further, we have used that by the same argument as in the proof of Proposition \ref{prop.optim.charac}, $\int_0^\cdot Z^1_s dW^{\alpha^2}_s$ on $[0, T]$ is a true $\mathbb{P}^{\alpha^2}$ martingale and vice versa for $\int_0^\cdot Z^2_s dW^{\alpha^1}_s$.\par
By Proposition \ref{prop.optim.charac}, $\alpha^2$ maximizes the Hamiltonian at $(t, X, \mu^2, Z^2_t)$ a.e., and we have by the separability assumption that a.e.
\begin{equation}\label{unqhml1}
    e^{-\lambda t}f_3(t, X, \alpha^1_t) + (Z^2_t)^\top \sigma^{-1}b(t, X, \alpha^1_t) \leq e^{-\lambda t}f_3(t, X, \alpha^2_t) + (Z^2_t)^\top \sigma^{-1}b(t, X, \alpha^2_t).
\end{equation}
Similarly, since $\alpha^1$ maximizes the Hamiltonian at $(t, X, \mu^1, Z^1_t)$ a.e., we have
\begin{equation}\label{unqhml2}
    e^{-\lambda t}f_3(t, X, \alpha^1_t) + (Z^1_t)^\top \sigma^{-1}b(t, X, \alpha^1_t) \geq e^{-\lambda t}f_3(t, X, \alpha^2_t) + (Z^1_t)^\top \sigma^{-1}b(t, X, \alpha^2_t).
\end{equation}
Furthermore, since $f_2$ is not dependent on $x$ and thus deterministic, we can immediately see that $\mathbb{E}^{\alpha^1}[f_2(s, \mu^1, q^1)] = \mathbb{E}^{\alpha^2}[f_2(s, \mu^1, q^1)]$ and $\mathbb{E}^{\alpha^1}[f_2(s, \mu^2, q^2)] = \mathbb{E}^{\alpha^2}[f_2(s, \mu^2, q^2)]$. Subtracting the expectations above together with these bounds gives
$$
    (\mathbb{E}^{\alpha^1} - \mathbb{E}^{\alpha^2})\left[Y^1_T - Y^2_T + \int_0^T e^{-\lambda s} (f_1(s, X, \mu^1) - f_1(s, X, \mu^2))ds \right] \geq 0.
$$
Note that this holds for $T$ arbitrary. Since $\Vert Y^1_T - Y^2_T\Vert_\infty \leq \frac{2M}{ \lambda}e^{-\lambda T}$, we have for any $T$, that
$$
    (\mathbb{E}^{\alpha^1} - \mathbb{E}^{\alpha^2})\left[\int_0^T e^{-\lambda s} (f_1(s, X, \mu^1) - f_1(s, X, \mu^2))ds \right] \geq -\frac{4M}{\lambda}e^{-\lambda T}.
$$
As $f_1$ is assumed to be bounded, taking the limit $T\rightarrow\infty$ on both sides shows 
$$
    (\mathbb{E}^{\alpha^1} - \mathbb{E}^{\alpha^2})\left[\int_0^\infty e^{-\lambda s} (f_1(s, X, \mu^1) - f_1(s, X, \mu^2))ds \right] \geq 0.
$$
As we are working with equilibria, we have $\mathbb{P}^{\alpha^1} \circ X^{-1} = \mu^1$ and $\mathbb{P}^{\alpha^2} \circ X^{-1} = \mu^2$. We can thus use the monotonicity assumption to find that the reverse inequality should be true as well, which implies equality.
This then implies that \eqref{unqhml1} and \eqref{unqhml2} needed to hold with equality a.e.\ as well. In particular, $\alpha^1$ needs to be an a.e.\ maximizer of the Hamiltonian at $(t, X, \mu^2, Z^2_t)$ and vice versa. Under our assumptions, the maximizer is unique implying a.e., $\alpha^1 = \alpha^2$ and thus $\mu^1 = \mathbb{P}^{\alpha^1} \circ X^{-1} = \mathbb{P}^{\alpha^2} \circ X^{-1} = \mu^2$ and $q^1 = \delta_{\mathbb{P}^{\alpha^1} \circ (\alpha^1_t)^{-1}}dt = \delta_{\mathbb{P}^{\alpha^2} \circ (\alpha^2_t)^{-1}}dt = q^2$.
\end{proof}

\section{Topological prerequisites}\label{sect: topo}
This section's aim is to gather some definitions and give detailed proofs of abstract topological properties of the spaces $\tilde{\mathcal{Q}}$, $\mathcal{Q}$ and $\mathcal{M}$ and the topologies $\tilde{\tau}$, $\tau$ and $\mathcal{S}$ used in the existence proofs.
As already explained in the definitions of $\tau$ and $\mathcal{S}$, our infinite horizon constructions build upon the finite horizon setup of \cite{Carmona15}.
We start by reviewing it in the next subsection.

\subsection{Family of finite horizon probability spaces}
\label{sct: fnt hrz}
We begin by taking a look at the link between the collection of finite horizon probability spaces $(\Omega^T, \mathcal{F}^T, \mathbb{P}^T)_{T > 0}$. First we consider how our projection maps $\tilde{\pi}$.
\begin{proposition}\label{prop: cont proj incl}
For any $0 \leq t \leq T < \infty$, if we equip $\mathcal{P}(\Omega^T)$ and $\mathcal{P}(\Omega^t)$ with $\tilde{\tau}^T$ and $\tilde{\tau}^t$ respectively, then $\tilde{\pi}^{T, t}$ is continuous.\par
\end{proposition}
\begin{proof}
    Take any measurable bounded $f: \Omega^t \rightarrow \mathbb{R}$. Then, we have for any $\mu \in \mathcal{P}(\Omega)^T$ that $ \int_{\Omega^t} f d(\tilde{\pi}^{T, t}(\mu)) = \int_{\Omega^T} f \circ \pi^{T, t} d\mu$. As $f \circ \pi^{T, t} : \Omega^T \rightarrow \mathbb{R}$ is measurable and bounded as well, the map $\mu \mapsto \int_{\Omega^t} f d(\tilde{\pi}^{T, t}(\mu))$ is continuous. 
    By \cite[Proposition 3.2]{Brezis}, the map $\tilde{\pi}^{T, t}$ is thus continuous.\par
\end{proof}
A notable difficulty in our analysis is the fact that the topology $\tilde{\tau}^T$ is, in general, not metrizable. Therefore, we will consider the compact subspace $\tilde{\mathcal{Q}}^T \subset \mathcal{P}(\Omega^T)$
\begin{equation}\label{def: QtildT}
    \tilde{\mathcal{Q}}^T:= \left\lbrace \mathbb{Q} \in \mathcal{P}(\Omega^T) \vert \int_{\mathcal{C}_0^T}\mathbb{Q}(\cdot, d\omega) = \upsilon, \mathbb{Q} \sim \mathbb{P}^T, \mathbb{E}\left[\left(\frac{ d\mathbb{Q}}{d\mathbb{P}^T}\right)^2 \right] \leq M^T, \mathbb{E}\left[\left(\frac{ d\mathbb{Q}}{d\mathbb{P}^T}\right)^{-1} \right]\leq M^T\right\rbrace
\end{equation}
on which we can find an inducing metric.
\begin{proposition}\label{prop: ttauT}
    Under $\tilde{\tau}^T$, the subset $\tilde{\mathcal{Q}}^T$ is convex, compact and metrizable.
\end{proposition}
\begin{proof}
First, note $\tilde{\tau}^T$ coincides with the $s$ topology of \cite[Definition 1.2.]{Balder}. By construction, $\tilde{\mathcal{Q}}^T$ is dominated by $\mathbb{P}^T$ and the densities are uniformly integrable under $\mathbb{P}^T$. Thus, by \cite[Proposition 2.2]{Balder}, $\tilde{\mathcal{Q}}^T$ is relatively compact under $\tilde{\tau}^T$. Further, by \cite[Proposition 2.3]{Balder}, the topology $\tilde{\tau}^T$ is metrizable on $\tilde{\mathcal{Q}}^T$.\par
Next, it is straightforward to see that $\tilde{\mathcal{Q}}^T$ is convex, since $z \mapsto z^2$ and $z \mapsto z^{-1}$ are convex. It remains to show that $\tilde{\mathcal{Q}}^T$ is closed. As it is metrizable, it is sufficient to show that it is sequentially closed. Take any $(\mathbb{Q}^n)_{n\geq 1} \subset \tilde{\mathcal{Q}}^T$ that converge towards some $\mathbb{Q} \in \mathcal{P}(\Omega^T)$. First, for any $A \in \mathcal{B}(\mathbb{R}^d)$, we have $\mathbb{Q}[A \times \mathcal{C}_0^T] = \lim_{n \rightarrow \infty} \mathbb{Q}^n[A \times \mathcal{C}_0^T] = \upsilon[A]$. Further, for any $A \in \mathcal{F}^T$ with $\mathbb{P}^T[A] = 0$, we have $\mathbb{Q}[A] = \lim_{n\rightarrow \infty} \mathbb{Q}^n[A] = 0$. Additionally, for any $A \in \mathcal{F}^T$ with $\mathbb{Q}[A] = 0$, we have 
$$\mathbb{P}^T[A] = \mathbb{E}^{\mathbb{Q}^n}\left[\left(\frac{d\mathbb{Q}^n}{d\mathbb{P}^T}\right)^{-1} \mathbb{1}_A\right] \leq \mathbb{E}^{\mathbb{Q}^n}\left[\left(\frac{d\mathbb{Q}^n}{d\mathbb{P}^T}\right)^{-2}\right]^{\frac{1}{2}} \mathbb{Q}^n[A]^{\frac{1}{2}} = \mathbb{E}\left[\left(\frac{d\mathbb{Q}^n}{d\mathbb{P}^T} 0\right)^{-1}\right] \mathbb{Q}^n[A].$$
Now, by assumption the $\mathbb{E}\left[\left(\frac{d\mathbb{Q}^n}{d\mathbb{P}^T}\right)^{-1}\right]$ are uniformly bounded, and we have $\lim_{n \rightarrow \infty} \mathbb{Q}^n[A] = \mathbb{Q}[A]$, thus $\mathbb{P}^T[A] = 0$. Hence, $\mathbb{P}^T \sim \mathbb{Q}$.\par
In particular, $\frac{d\mathbb{Q}}{d\mathbb{P}^T} \in \mathbb{L}^1(\Omega^T, \mathbb{P}^T)$ exists and is by assumption, the weak $\mathbb{L}^1$ limit of the $\frac{d\mathbb{Q}^n}{d\mathbb{P}^T}$. Since we have already shown that $\tilde{\mathcal{Q}}^T$ is convex, and since the density of absolutely continuous measures behaves linearly, we can find another sequence $(\mathbb{Q}^n)_{n\geq 1} \subset \tilde{\mathcal{Q}}^T$ such that $\frac{d\mathbb{Q}^n}{d\mathbb{P}^T}$ converges to $\frac{d\mathbb{Q}}{d\mathbb{P}^T}$ $\mathbb{L}^1$ strongly and if we pass to another subsequence $\mathbb{P}^T$ a.s. as well. Thus, by Fatou's lemma, we have $\mathbb{E}\left[\left(\frac{ d\mathbb{Q}}{d\mathbb{P}^T}\right)^2\right]\leq M^T$ and $\mathbb{E}\left[\left(\frac{ d\mathbb{Q}}{d\mathbb{P}^T}\right)^{-1} \right]\leq M^T$, and hence $\mathbb{Q} \in \tilde{\mathcal{Q}}^T$.
\end{proof}
For our purposes, we would like to construct a specific metrizing metric to ensure that it can be bounded by the total variation metric. Let $B_T^1 := \lbrace f \in \mathbb{L}^\infty(\Omega^T, \mathbb{P}^T) \vert \Vert f \Vert_\infty \leq 1 \rbrace$. 
We can also consider $B_T^1$ as a subset of $\mathbb{L}^2(\Omega^T, \mathbb{P}^T)$. Since $\Omega^T$ is separable, it follows that $\mathbb{L}^2(\Omega^T, \mathbb{P}^T)$ is separable, and $B_T^1$ as a subset of $\mathbb{L}^2(\Omega^T, \mathbb{P}^T)$ is separable as well, with respect to the induced $\mathbb{L}^2$ metric. We can thus find an $\mathbb{L}^2$-dense countable sequence $(h^n)_{n\geq 1}$ in $B_T^1$. 
Then, we define
$$d_{\tilde{\tau}^T}(\mathbb{Q}, \mathbb{Q}') := \sum_{n\geq 1} 2^{-n} \left\vert \int h^n d\mathbb{Q} - \int h^n d\mathbb{Q}' \right\vert.$$
\begin{proposition}\label{prop: tau mtrc}
    The metric $d_{\tilde{\tau}^T}$ metrizes $\tilde{\tau}^T$ on $\tilde{\mathcal{Q}}$. 
    Moreover, the metric is bounded by the total variation, that is for any $\mathbb{Q}, \mathbb{Q}'\in\tilde{\mathcal{Q}}^T$, we have $d_{\tilde{\tau}^T}(\mathbb{Q}, \mathbb{Q}') \leq d_{TV}(\mathbb{Q}, \mathbb{Q}')$.
\end{proposition}
\begin{proof}
It is clear that $d_{\tilde{\tau}^T}$ defines a metric. Clearly, if a sequence $(\mathbb{Q}^n)_{n\geq 1} \subset \tilde{\mathcal{Q}}^T$ converges to some $\mathbb{Q} \in \tilde{\mathcal{Q}}^T$, we have $\lim_{k \rightarrow \infty} \int h^n d\mathbb{Q}^k = \int h^n d\mathbb{Q}$ for all $n \geq 1$. 

Conversely, assume for all $n \geq 1$, we have $\lim_{k \rightarrow \infty} \int h^n d\mathbb{Q}^k = \int h^n d\mathbb{Q}$. 
Now, take any $f \in B_T^1$. 
For any $\epsilon > 0$, there exists $h^n\in \mathcal{D}$ such that $\mathbb{E}[|f - h^n|^2] < \frac{\epsilon^2}{16M^T}$. For such, there exists $K \geq 1$ such that for any $k \geq K$, we have $\vert \int h^n d\mathbb{Q}^k - \int h^n d\mathbb{Q} \vert < \frac{\epsilon}{2}$. Then,
\begin{equation*}
    \begin{split}
         \left\vert \int f d\mathbb{Q}^k - \int f d\mathbb{Q} \right\vert & \leq  \left\vert \int f d\mathbb{Q}^k - \int h^n d\mathbb{Q}^k \right\vert + \left\vert \int h^n d\mathbb{Q}^k - \int h^n d\mathbb{Q} \right\vert + \left\vert \int h^n d\mathbb{Q} - \int f d\mathbb{Q} \right\vert\\
         & <  \mathbb{E}\left[\vert f-h^n \vert\left(\frac{d\mathbb{Q}^k}{d\mathbb{P}^T} + \frac{d\mathbb{Q}}{d\mathbb{P}^T} \right) \right] + \frac{\epsilon}{2}\\
         & \leq \frac{\epsilon}{4(M^T)^{\frac{1}{2}}}\left(\mathbb{E}\left[\left(\frac{d\mathbb{Q}^k}{d\mathbb{P}^T}\right)^2\right]^{\frac{1}{2}} + \mathbb{E}\left[\left(\frac{d\mathbb{Q}}{d\mathbb{P}^T}\right)^2\right]^{\frac{1}{2}}\right) + \frac{\epsilon}{2} \leq \epsilon.
    \end{split}
\end{equation*}
Thus, $\lim_{k\rightarrow\infty}\int f d\mathbb{Q}^k = \int f d\mathbb{Q}$ and as $f$ was arbitrary, it follows that $\mathbb{Q}^k$ converges to $\mathbb{Q}$ in $\tilde{\tau}^T$.We have thus shown that sequences have the same limits under $\tilde{\tau}^T$  as under the topology induced by $d_{\tilde{\tau}^T}$. 
Since by Proposition \ref{prop: ttauT}, we already know that $\tilde{\tau}^T$ is metrizable on $\tilde{\mathcal{Q}}$ and thus first countable, it thus follows that the metric $d_{\tilde{\tau}^T}$ metrizes $\tilde{\tau}^T$.\par
Lastly, note that for any $\mathbb{Q}, \mathbb{Q}' \in \tilde{\mathcal{Q}}^T$ we have
$$d_{\tilde{\tau}^T}(\mathbb{Q}, \mathbb{Q}') \leq \sum_{n\geq 1} 2^{-n} \sup_{\phi \in B_T^1} \int \phi d(\mathbb{Q} - \mathbb{Q}') = d_{TV}(\mathbb{Q}, \mathbb{Q}')$$
as either for $\phi = h^n$ or $\phi = -h^n$ we have $\vert \int h^n d\mathbb{Q} - \int h^n d\mathbb{Q}' \vert = \int \phi d(\mathbb{Q} - \mathbb{Q}')$.
\end{proof}

\subsection{The controlled measure}\label{def: tau}

Let us first restate a version of Kolmogorov's extension theorem to our setting of topological spaces.
\begin{definition}\label{def: const}
    A sequence of measures $(\mu^n)_{n\geq 1}$ is called \emph{consistent} if there exists a diverging sequence of finite times $0 \leq t_1 < t_2 < \dots$, such that $\mu^n \in \mathcal{P}(\Omega^{t_n})$ for $n \geq 1$ and if for any $m \leq n$, we have $\tilde{\pi}^{t_n, t_m}(\mu^n) = \mu^m$.
\end{definition}
\begin{proposition}\label{prop: ext}
    Given a consistence sequence $(\mu^n)_{n\geq 1}$ with $\mu^n \in \mathcal{P}(\Omega^{t_n})$ for all $n\geq1$, there exists a unique $\mu\in \mathcal{P}(\Omega)$ such that $\tilde{\pi}^{t_n}(\mu) = \mu^n$ for all $n$.
\end{proposition}
\begin{proof}
The existence of a unique $\mu$ has been proven in \cite[Theorem 1.3.5]{StroockVaradhan}. The convergence in $\tilde{\tau}$ is immediate by definition.
\end{proof}
This extension result has an immediate consequence for the topology $\tilde{\tau}$. This is best understood by considering the inverse system$(\mathcal{P}(\Omega^T), \tilde{\pi}^{T, t}, [0, \infty))$.\par
With an inverse system, we mean a family of topological spaces $(\mathcal{X}_i)_{i \in I}$ over an ordered index set $(I, \preceq)$ and continuous mappings $f_{j, i} : \mathcal{X}_i \rightarrow \mathcal{X}_j$ for any $i \preceq j$ such that for any $i$, we have $f_{i, i} = id_{\mathcal{X}_i}$ and for any $i \leq j \leq k$, we have $f_{k, i} = f_{j, i} \circ f_{f, j}$. For such an inverse system $(\mathcal{X}_i, f_{i, j}, I)$, we can consider the inverse limit $\lim\limits_{\longleftarrow} X_i= \lbrace (x_i)_{i \in I} \in \prod_{i \in I} \mathcal{X}_i \vert \forall i \preceq j: f_{j, i}(x_j) = x_i \rbrace $ that is equipped with the by the product topology induced subset topology.\par
We can thus consider the map $\mathcal{P}(\Omega) \rightarrow \lim\limits_{\longleftarrow} \mathcal{P}(\Omega^T), \mathbb{Q} \mapsto (\tilde{\pi}^{T}(\mathbb{Q}))_{T \in [0, \infty)}$. By Proposition \ref{prop: ext}, this map is a bijection between $\mathcal{P}(\Omega^T)$ and by the way we defined $\tilde{\tau}$ it is immediate that this map is a homeomorphism. This gives a much more convenient way to describe $\tilde{\tau}$.
\begin{proposition}\label{prop: inft Q top}
    $\tilde{\mathcal{Q}}$ is Hausdorff compact with respect to $\tilde{\tau}$. Further, the induced subspace topology on $\tilde{\mathcal{Q}}$ is metrizable, e.g.\ by the metric
    $$d_{\tilde{\tau}}(\mathbb{Q}, \mathbb{Q}') := \sum_{n = 1}^\infty 2^{-n} d_{\tilde{\tau}^n} (\tilde{\pi}^n(\mathbb{Q}), \tilde{\pi}^n(\mathbb{Q'})).
    $$
\end{proposition}
\begin{proof}
The Hausdorffness follows from \cite[Theorem 2.5.2.]{Engelking}. As an intersection of closed subsets of $\mathcal{P}(\Omega)$, we know that $\tilde{\mathcal{Q}}$ is closed. By \cite[Corollary 2.5.7.]{Engelking}, $\tilde{\mathcal{Q}}$ is thus homeomorph to inverse limit of $(\overline{\tilde{\pi}^T(\tilde{\mathcal{Q}})}, \tilde{\pi}^{T, t}, [0, \infty))$. As by Proposition \ref{prop: ttauT} all $\overline{\tilde{\pi}^T(\tilde{\mathcal{Q}})} \subset \tilde{\mathcal{Q}}^T$ are compact, by \cite[Theorem II.2.1]{Maurin}, $\tilde{\mathcal{Q}}$ is compact.\par
To show metrizability, first note that by \cite[Corollary 2.5.11.]{Engelking}, $\mathcal{Q}$ is also homeomorph to the inverse limit of the $(\overline{\tilde{\pi}^n(\tilde{\mathcal{Q}})}, \tilde{\pi}^{n, n'}, \mathbb{Z}_{\geq 0})$. The metrizability through $d_{\tilde{\tau}}$ then follows from Lemma \ref{lem: inv mtrz}.
\end{proof}
Note that metrizability of $\tilde{\tau}$ on $\tilde{\mathcal{Q}}$ is also important to show that $\tilde{\mathcal{Q}}$ is sequential compact as well. For metrizability of inverse limits for inverse sequences we prove the following result-.
\begin{lemma}\label{lem: inv mtrz}
    Let $\mathcal{X}$ be the inverse limit of an inverse sequence $(\mathcal{X}_n, f_{n, m}, \mathbb{Z}_{\geq 0})$ such that for any $n \geq 0$, the topology $\tau_n$ on $\mathcal{X}_n$ is metrized by some metric $d_n$ that are uniformly bounded by some constant $K\geq 0$. Then, the topology on the inverse limit $\mathcal{X}$ is metrized through the metric
    $$d(x, y) := \sum_{i \geq 1} 2^{-i} d_i(f_i(x), f_i(y)).$$
\end{lemma}
\begin{proof}
First, as $d_i \leq K$, we also have $d \leq K$, thus $d$ is finite and well defined. It is clear that $d$ does define a metric.\par
Let us first show that the topology induced by $d$ is coarser than $\tau$. For this, consider for an arbitrary $x\in \mathcal{X}$ the open ball $B_x(\epsilon)$ with some radius $\epsilon>0$. It suffices to show that there exists a $\tau$-open set $O$ such that $x \in O \subset B_{x}(\epsilon)$. We will denote the projection maps from $\mathcal{X}$ onto $\mathcal{X}_i$ by $f_i$.\par
With respect to $d_i$, let us for any $i \geq 1$ consider the ball $B_{f_i(x)}(\frac{\epsilon}{2}) \subset \mathcal{X}_i$ which is by assumption open with respect to $\tau_i$. Since $f_i$ is continuous, $f_i^{-1}(B_{f_i(x)}(\frac{\epsilon}{2}))$ is $\tau$-open. Take $I$ such that $\sum_{i = I + 1}^\infty 2^{-i + 1} < \frac{\epsilon}{2K}$. If we define $O := \cap_{i=1}^{I} f_i^{-1}(B_{f_i(x)}(\frac{\epsilon}{2}))$, as a finite intersection, $O$ is open with respect to $\tau$ as well.
Now, for any $y \in O$, we have
$$d(x, y) \leq \sum_{i = 1}^{I} 2^{-i} d_i(f_i(x), f_i(y)) + \sum_{n = I + 1}^\infty 2^{-i} K < \sum_{i = 1}^{I} 2^{-i} \frac{\epsilon}{2} + \frac{\epsilon}{2} < \epsilon$$
and thus $x \in O \subset B_{x}(\epsilon)$.\par
Let us now show that the topology induced by $d$ makes all $f_i$ continuous. Since this topology is metric, it suffices to show that for any $i\geq 1$, for any $(x_n)_{n \geq 1} \subset \mathcal{X}$ and $x\in \mathcal{X}$ with $d(x_n, x) \rightarrow 0$, we have $f_i(x_i) \rightarrow f_i(x)$ in $\tau_i$. But note that $d(x_n, x) \rightarrow 0$ implies $d_i(f(x_n), f(x))\rightarrow 0$ for all $i \geq 1$. As $d_i$ metrizes $\tau_i$, the $f_i(x_n)$ thus converge to $f_i(x)$ in $\tau_i$, making all $f_i$ continuous. It follows that $\tau$ must coincide with the topology that $d$ induces.
\end{proof}
\subsection{The local stable topology}\label{sect: stblc}
Let us now revisit the space $\mathcal{M}$ of relaxed flows of the laws in the control variable and its local stable topology $\mathcal{S}$. 
Here again, we prove metrizability and compactness of the local stable topology restricted on the set $\mathcal{M}$.
\begin{proposition}\label{prop: inft M top}
    The local stable topology $\mathcal{S}$ is Hausdorff and metrizable. Further, $\mathcal{M}$ under $\mathcal{S}$ is compact.
\end{proposition}
\begin{proof}
Just as with $\mathcal{Q}$, we can relate $\mathcal{M}$ to its finite time horizon counterparts. That is, we write $\mathcal{M}^T$ for the space of finite measures on $[0, T] \times \mathcal{P}(A)$ with first marginal $dt$. We equip this space with the stable topology as defined in \cite{JacodMemin} which we denote by $\mathcal{S}^T$.\par
Any $\nu \in \mathcal{M}$ induces a measure $\nu_{\vert [0, T]}$ on $[0, T] \times \mathcal{P}(A)$ via $\nu_{\vert [0, T]}(\cdot) = \nu(\cdot \cap ([0, T] \times \mathcal{P}(A)))$ and we have $\nu_{\vert [0, T]} \in \mathcal{M}^T$. As
$$\int_0^T \int_{\mathcal{P}(A)} f(t, q) (\nu_{\vert [0, T]})_t(dq) dt = \int_0^\infty \int_{\mathcal{P}(A)} \mathbb{1}_{\lbrace t \leq T \rbrace} f(t, q) \nu_t(dq) dt$$
and $\mathbb{1}_{\lbrace \cdot \leq T \rbrace} f$ can be considered as a function on $[0, \infty) \times \mathcal{P}(A)$ that satisfies the above conditions on $f$, it follows that for any $T$, the restriction operator $\nu \mapsto \nu_{\vert [0, T]}$ is continuous. Further, $\mathcal{S}$ is by construction the coarsest topology that makes for all $T\geq 0$ the maps $\mathcal{M} \rightarrow \mathcal{M}^T, \nu \mapsto \nu_{\vert [0, T]}$ continuous. Since if two elements of $\mathcal{M}$ agree on any finite time horizon, they need to agree everywhere, it is immediate to see that $\mathcal{M}$ is homeomorph to the inverse limit of the $\mathcal{M}^T$ (with the corresponding restriction maps defined straightforwardly) and in particular Hausdorff. Since every $\mathcal{M}^T$ is already compact by \cite[Corollary 3.9]{JacodMemin}, we can again use \cite[Theorem II.2.1]{Maurin} to show that $\mathcal{M}$ is compact.\par
Now, for any $T>0$, because the Borel $\sigma$ algebra on $[0, T]$ is countably generated, by \cite[Proposition 3.25.]{Florescu}, the topologies $\mathcal{S}^T$ are metrizable. For each $T$, choose such a compatible metric $d_{\mathcal{S}^T}$ that is bounded by $1$. Then, a metric for $\mathcal{S}$ is by Proposition \ref{lem: inv mtrz} given by
\begin{equation}
\label{eq:def.dtau}
    d_{\mathcal{S}}(\nu, \nu') := \sum_{n = 1}^\infty 2^{-n} d_{\mathcal{S}^n}(\nu_{\vert[0, n]}, \nu'_{\vert[0, n]}).
\end{equation}
\end{proof}
\section{Proofs of the long time asymptotic results}\label{sec: asym}


We begin with some comments on the regularity of functions $F:[0, \infty) \times \mathcal{P}(\mathcal{C}) \rightarrow \mathbb{R}$.
Assume the function $F$ to be consistent in the sense of Assumption \ref{asmp: cons} and sequentially continuous in $\mu$ as in Assumption \ref{asmp: cont}.
To properly ensure continuity in the finite horizon problems, we will need a stronger notion of progressive sequential continuity.
That is, we need that given $t>0$, for any sequence $\mu^n$ and a measure $\mu$ such that $\tilde{\pi}^t(\mu^n) \rightarrow \tilde{\pi}^t(\mu)$ we also have $F(t, \mu^n) \rightarrow F(t, \mu)$. 
In general, this is stronger than mere continuity in $\mu$, but we will see that together with the consistency assumption the two continuity properties are equivalent.\par

To see this, let us take any sequence $\mu^n$, some $\mu$ and some $t \geq 0$ that satisfy $\tilde{\pi}^t(\mu^n) \rightarrow \tilde{\pi}^t(\mu)$. Writing $\omega$ as the canonical element on $\mathcal{C}$, we define $\hat{\mu}^n = \mu^n \circ (\omega_{\cdot \wedge t})^{-1} \in \mathcal{P}(\mathcal{C})$ and $\hat{\mu} = \mu \circ (\omega_{\cdot \wedge t})^{-1} \in \mathcal{P}(\mathcal{C})$ as the laws of $\omega$ stopped at $t$. As the restriction $\omega \mapsto \omega_{\cdot \wedge t}$ is measurable, it is easy to check that $\hat{\mu}^n \rightarrow \hat{\mu}$ in $\tau$. By consistency and continuity in $\tau$, we must have 
$$F(t, \mu^n) = F(t, \hat{\mu}^n) \rightarrow F(t, \hat{\mu}) = F(t, \mu).$$
The above discussion naturally extends to the case of maps taking more arguments.

\subsection{Infinite horizon equilibria as approximate equilibria}

The intuition behind the inverse convergence result is that 
the fixed point property is preserved when restricting to a smaller time horizon. 
That is,
given infinite horizon mean field equilibrium $(\mu, q,\alpha)$, the restriction $(\mu^T, q^T, \alpha^T):= (\tilde{\pi}^T(\mu), q_{\vert[0, T]}, \alpha_{\vert[0, T]}) \in \mathcal{P}(\mathcal{C}^T) \times \mathcal{M}^T \times \mathbb{A}^T$ satisfies the fixed point property.
However, in general, $\alpha^T $ needs not be optimal for the finite horizon control problem. 
The fixed point property for $q^T$ is immediate as $\mathbb{P}^{\mu, \alpha}_{\vert \mathcal{F}_T} = \mathbb{P}^{\mu^T, \alpha^T, T}$. For the state variable, we have
\begin{align*}
    \mu^T &= \tilde{\pi}^T(\mathbb{P}^{\mu, \alpha} \circ X^{-1}) = \mathbb{P}^{\mu, \alpha}_{\vert \mathcal{F}_T} \circ (X_{\vert[0, T]})^{-1}\\
    & = \mathbb{P}^{\mu^T, \alpha^T} \circ (X_{\vert[0, T]})^{-1}.
\end{align*}
\begin{proof}[Theorem \ref{thm: eps opt}]
Let $(\mu, q)$, $(\mu^T, q^T)$, $\alpha$ and $\alpha^T$ be taken as described above. In the finite horizon control problem, for the control $\alpha^T$, the "remaining utility process"
$$
    Y^{\mu^T, q^T, \alpha^T}_t = \mathbb{E}^{\mu^T, \alpha^T}\left[\int_t^T e^{-\lambda s} f(s, X, \mu^T, q^T_s, \alpha^T_s) ds \vert \mathcal{F}_t\right]
$$
is a solution of the finite horizon BSDE
$$Y^{\mu^T, q^T, \alpha^T}_t = \int_t^T h(s, X, \mu^T, q^T_s, Z^{\mu^T, q^T, \alpha^T}_s, \alpha^T_s)ds - \int_t^T Z^{\mu^T, q^T, a^T}dW_s$$
with terminal condition $Y^{\mu^T, q^T, \alpha^T}_T = 0$. Just as we have seen before for the infinite horizon control problem, the transformed process $(\tilde{Y}^T_\cdot, \tilde{Z}^T_\cdot) = (e^{\lambda \cdot} Y^{\mu^T, q^T, \alpha^T}_\cdot, e^{\lambda \cdot} Z^{\mu^T, q^T, \alpha^T}_\cdot)$ is then a solution of
$$\tilde{Y}^T_t = \int_t^T \tilde{h}(s, X, \mu^T, q^T_s, \tilde{Z}^T_s, \alpha^T_s) -\lambda\tilde{Y}^T_s ds -\int_t^T \tilde{Z}^T_s dW_s.
$$
with $\tilde{Y}^T_T = 0$. By construction, $\int Y^{\mu^T, q^T, \alpha^T}_0 d\upsilon= \int\tilde{Y}^T_0 d\upsilon= J^{\mu^T, q^T,T}(\alpha^T)$. Let us also consider the transformed BSDE corresponding to the optimal control on finite horizon given by
$$\hat{Y}^T_t = \int_t^T \tilde{H}(s, X, \mu^T, q^T_s, \hat{Z}^T_s) - \lambda \hat{Y}^T_s ds - \int_t^T \hat{Z}^T_s dW_s$$
with terminal condition $\hat{Y}_T = 0$.
By optimality,  we have $\int \hat{Y}_0d\upsilon = V^{\mu^T, q^T,T}$.\par
To connect these two solutions, let us recall the same BSDE \eqref{inft BSDE} on infinite horizon:
$$\tilde{Y}_t = \tilde{Y}_T +  \int_t^T \tilde{H}(s, X, \mu, q_s, \tilde{Z}_s) - \lambda \tilde{Y}_s ds - \int_t^T \tilde{Z}_s dW_s$$
holding for all $t \leq T < \infty$. As $\alpha$ is the optimal control on infinite horizon, we have
$$
    \int \tilde{Y}_0 d\upsilon = \mathbb{E}^{\mu, \alpha}\left[\int_0^\infty e^{-\lambda s} f(s, X, \mu, q_s, \alpha_s) ds\right] = \int \tilde{Y}^T_0 d\upsilon + \mathbb{E}^{\mu, \alpha}\left[\int_T^\infty e^{-\lambda s} f(s, X, \mu, q_s, \alpha_s) ds\right]
    $$
since by the consistency assumption \ref{asmp: cons}
$$\int \tilde{Y}^T_0 d\upsilon= \mathbb{E}^{\mu^T, \alpha^T}\left[\int_0^T e^{-\lambda s}f(s, X, \mu^T, q^T_s, \alpha^T_s)ds\right] = \mathbb{E}^{\mu, \alpha}\left[\int_0^T e^{-\lambda s}f(s, X, \mu, q_s, \alpha_s)ds\right]$$
and thus by Assumption \eqref{asmp: stnd}.$(ii)$ we have
$$\left\vert \int \tilde{Y}_0 - \tilde{Y}^T_0 d\upsilon \right\vert = \left\vert \mathbb{E}^{\mu, \alpha}\left[\int_T^\infty e^{-\lambda s} f(s, X, \mu, q_s, \alpha_s) ds\right]\right \vert \leq \frac{M}{\lambda}e^{-\lambda T}.$$
Now, using the consistency assumption, $\tilde{H}(s, X, \mu^T, q^T_s, \cdot) = \tilde{H}(s, X, \mu, q_s, \cdot)$ for $s \leq T$, and $\hat{Y}^T$ is actually the finite horizon approximation of $\tilde{Y}$ as discussed in Lemma \ref{lem: YZ aprx}. 
In particular, we a.s.\ have $\vert \tilde{Y}_0 - \hat{Y}^T_0 \vert \leq \frac{M}{\lambda} e^{- \lambda T}$. Together, this shows $J^{\mu^T, q^T,T}(\alpha^T) \geq V^{\mu^T, q^T, T} - \frac{2M}{\lambda}e^{-\lambda T}$.
\end{proof}

\subsection{Convergence to infinite horizon equilibria}\label{sect: cnv inft eq}

\begin{proof}[Theorem \ref{thm: cnv inft}]
For each $T_n$, We consider the infinite horizon mean field game that admits the reward functional
$$\mathbb{E}^{\mu, \alpha}\left[\int_0^{T_n} e^{-\lambda s} f(s, X, \mu, q_s, \alpha_s)ds - \int_{T_n}^\infty e^{-\lambda s} \Vert \alpha_s - a_0 \Vert_A ds\right]$$
for an arbitrary fixed element $a_0\in A$. For this infinite horizon game, we can find a solution $(\mu^{n, 0}, q^{n, 0}, \alpha^{n, 0})$ by Theorem \ref{thm: exst E}. Since after $T_n$, the control problem admits a unique optimal control given by the constant control at $a_0$, one can see that given $(\tilde{\pi}^{T_n}(\mu^{n, 0}), q^{n, 0}_{\vert[0, T_n]})$, the control $\alpha^{n}_{\vert [0, T_n]}$ is optimal for the control problem of the finite horizon game. Further, $\tilde{\pi}^{T_n}(\mu^{n, 0}) = \tilde{\pi}^{T_n}(\mathbb{P}^{\mu^{n, 0}, \alpha^{n, 0}}) = \mathbb{P}^{\tilde{\pi}^{T_n}(\mu^{n, 0}), \alpha^{n}_{\vert [0, T_n]}}$, and the fixed point property is fulfilled. Thus, $(\tilde{\pi}^{T_n}(\mu^{n, 0}), q^{n, 0}_{\vert[0, T_n]}, \alpha^{n}_{\vert [0, T_n]})$ is a solution to the finite horizon mean field game which shows existence.\par
Just as in the statement Theorem \ref{thm: cnv inft}, we now consider $(\mu^n, q^n, \alpha^n)$ as any solution to the game on horizon $[0, T_n]$. It is immediate that $\overline{q}^n \in \mathcal{M}$ and as $\overline{\mu}^n = \mathbb{P}^{\mu^n, \alpha^n} \circ X^{-1}$, we must also have $\overline{\mu}^n \in \mathcal{Q}$. As $\mathcal{Q} \times \mathcal{M}$ is compact, after passing to a subsequence, we will from now on assume $(\mu^n, q^n) \rightarrow (\mu, \nu) \in \mathcal{Q} \times \mathcal{M}$. Note that at this point, it is not clear yet whether $\nu$ is a strict flow of measures.\par
As we have done before, we can consider the solution $(\widetilde{Y}^n, \widetilde{Z}^n)$ of the finite horizon BSDE
$$
    \widetilde{Y}^n_t = \int_t^{T_n} \widetilde{H}(s, X, \mu^n, q^n_s, \widetilde{Z}^n_s) - \lambda \widetilde{Y}^n_s ds - \int_t^{T_n} \widetilde{Z}^n_s dW_s
$$
with terminal condition $\widetilde{Y}^n_{T_n} = 0$. We consider $(\widetilde{Y}^n, \widetilde{Z}^n)$ as processes on $[0, \infty)$ by defining $(\widetilde{Y}^n_t, \widetilde{Z}^n_t) = 0$ for $t > T_n$. For later use, we also consider the unique solution $(\overline{Y}^n, \overline{Z}^n)$ of the infinite horizon BSDE
$$\overline{Y}^n_t = \overline{Y}^n_T + \int_t^T \widetilde{H}(s, X, \overline{\mu}^n,  \overline{q}^n_s, \overline{Z}^n_s) - \lambda \overline{Y}^n_s ds - \int_t^T \overline{Z}^n_s dW_s$$
holding for all $t \leq T<\infty$.\par
Lastly, we can also consider the infinite horizon BSDE for the optimal control of the limit measures $(\mu, \nu)$ given by
$$
    \widetilde{Y}_t = \widetilde{Y}_T + \int_t^T \widetilde{H}(s, X, \mu, \nu_s, \widetilde{Z}_s) - \lambda \widetilde{Y}_s ds - \int_t^T \widetilde{Z}_s dW_s
$$
for all $t \leq T < \infty$.\par
Recall from Lemma \ref{lem: YZ aprx} that the solution to an infinite horizon BSDE can be constructed as a limit of a sequence of solutions for corresponding finite horizon BSDEs with increasing time horizons. In our case, by consistency, the $\widetilde{Z}^n$ are the finite horizon approximation of the $\overline{Z}^n$. Since the Lipschitz constant in $Z$ of $\widetilde{H}$ is independent of the measure arguments, by Lemma \ref{lem: YZ aprx}, they become uniformly close for big $n$, i.e.\ for any $\epsilon > 0$, there exists $\widetilde{n}$ so that for all $n \geq \widetilde{n}$, we have $\mathbb{E}[\int_0^\infty e^{-2\lambda s} \Vert \widetilde{Z}^n_s - \overline{Z}^n_s \Vert^2 ds] < \epsilon$. As in the proof of \ref{lem: cntrcont}, for $(\overline{\mu}^n, \overline{q}^n) \rightarrow (\mu, \nu)$ we also have $\mathbb{E}[\int_0^\infty e^{-2\lambda s} \Vert \widetilde{Z}_s - \overline{Z}^n_s \Vert^2 ds] \rightarrow 0$. Together, we have $\mathbb{E}[\int_0^\infty e^{-2\lambda s} \Vert \widetilde{Z}^n_s - \widetilde{Z}_s \Vert^2 ds] \rightarrow 0$. \par
Denote by $\alpha = \mathbb{A}(\mu, \nu)$ the optimal control for $(\mu, \nu)$. Applying Lemma \ref{lem: cntrcont} here shows $\Vert\overline{\alpha}^n - \alpha\Vert_\mathbb{A} \rightarrow 0$.\par
Recall that $\mathfrak{u} \circ \mathfrak{P}: \mathcal{Q} \times \mathbb{A} \rightarrow\mathcal{Q}, (\mu, \alpha) \mapsto \mathbb{P}^{\mu, \alpha} \circ X^{-1}$ and $\mathfrak{v}: \mathcal{Q} \times \mathbb{A} \rightarrow \mathcal{M}, (\mu, \alpha) \mapsto \delta_{\mathfrak{P}(\mu, \alpha) \circ \alpha_t^{-1}} dt$ as considered in Lemma \ref{lem: P uc} and the proof of Theorem \ref{thm: exst E} are continuous.\par
Under the consistency assumption, the law of $X_{\vert [0, T_n]}$ and $\alpha_{\vert [0, T_n]}$ under the optimal control are fully determined by $\mu^n$ and $q^n$ already. Thus, we can write the fixed point property of our finite horizon solutions also as $\mu^n = \tilde{\pi}^{T_n}(\overline{\mu}^n) = \tilde{\pi}^{T_n}(\mathfrak{u}(\mathfrak{P}(\overline{\mu}^n, \overline{\alpha}^n)))$ and $q^n = \mathfrak{v}(\overline{\mu}^n, \overline{\alpha}^n)_{\vert[0, T_n]}$.\par
Further, the fixed point property is preserved under restrictions, that is for $n \leq n'$ we have $\tilde{\pi}^{T_{n'}, T_n}(\mu^{n'}) = \tilde{\pi}^{T_n}(\mathfrak{u}(\mathfrak{P}(\overline{\mu}^{n'}, \overline{\alpha}^{n'})))$ and $q^{n'}_{\vert [0, T_n]} = \mathfrak{v}(\overline{\mu}^{n'}, \overline{\alpha}^{n'})_{\vert[0, T_n]}$.\par
Therefore, if we fix a horizon $T_n$ for now, we have
$\tilde{\pi}^{T_n}(\mu) = \lim_{n' \rightarrow \infty} \tilde{\pi}^{T_n}(\overline{\mu}^{n'})$ as well as $\nu_{\vert [0, T_n]} = \lim_{n' \rightarrow \infty} q^{n'}_{\vert[0, T_n]}$. By the continuity of $(\mathfrak{u} \circ\mathfrak{P} \circ \overline{A}, \mathfrak{v})$ and of the restrictions, we must have $\tilde{\pi}^{T_n}(\mu) = \tilde{\pi}^{T_n}(\mathfrak{u}(\mathfrak{P}(\mu, \alpha)))$ and $ \nu_{\vert [0, T_n]} = \mathfrak{v}(\mu, \alpha)_{\vert [0, T_n]}$.\par
This holds for all $T_n$. As the $T_n$ diverge, we immediately have $\nu = \mathfrak{v}(\mu, \alpha)$ and in particular, $\nu$ is a strict flow of measures. By uniqueness of Lemma \ref{prop: ext}, $\mu = \mathfrak{u}(\mathfrak{P}(\mu, \alpha))$, showing that $(\mu, \nu)$ is an equilibrium.\par
To see $\lim_{n \rightarrow \infty} V^{\mu^n, q^n, T_n} = V^{\mu, \nu}$, observe that $\widetilde{Y}^n_0$ and $\overline{Y}^n_0$ become arbitrarily close uniformly for big $n$ by Lemma \ref{lem: YZ aprx}. Further, by Lemma \ref{lem: inft bsde stab}, $\overline{Y}^n_0 \rightarrow \widetilde{Y}_0$ a.s.\ as well which implies $\lim_{n \rightarrow \infty} V^{\mu^n, q^n, T_n} = \lim_{n \rightarrow \infty} \mathbb{E}[\widetilde{Y}^n_0] = \mathbb{E}[\widetilde{Y}_0]  = V^{\mu, \nu}$.
\end{proof}
\begin{remark}
    Recall, as observed in the proofs of existence, that the optimal control map $\mathbb{A}$ is in general only metrically upper hemicontinuous which is too weak for us to infer convergence in $\mathcal{M}$. 
    Therefore, we decided to state our result under the additional assumption \ref{asmp: sngl}. 
    One can derive an equivalent result under assumption \ref{asmp: NEMFG} instead.
    However, in this case we do not have convergence of the optimal controls anymore.
\end{remark}
\subsection{Quantitative convergence to infinite horizon equilibria}\label{sect: rates}
We now present the proof of the quantitative convergence to the infinite horizon equilibrium. The proof is based on an idea that exploits the stabilizing effect of Lasry-Lions monotonicity.

\begin{proof}[Theorem \ref{thm: cnv rate}]
We again consider the solution of the infinite horizon BSDE
$$
Y_r = Y_t + \int_r^t e^{-\lambda s} f(s, X, \mu, \alpha_s) + Z_s^\top \sigma^{-1}b(s, X, \alpha_s) ds - \int_r^t Z_s dW_s$$
holding for any $0 \leq r \leq t<\infty$. 
Recall that $Y$ is the remaining utility process. By fixing $t$ at $T$, we can also consider the restriction of $Y$ and $Z$ on $[0, T]$ as a solution of a finite horizon BSDE on $[0, T]$ with terminal value $Y_T$.\par
For solutions of the finite horizon problem, we can consider the finite horizon BSDE
$$
    Y^T_t = \int_t^T e^{-\lambda s} f(s, X, \mu^T, \alpha^T_s) + (Z^T_s)^\top \sigma^{-1}b(s, X, \alpha^T_s) ds - \int_t^T Z^T_s dW_s
$$
with terminal condition $Y^T_T=0$. Just as in the proof of Theorem \ref{thm: unq}, we can find
\begin{equation}\label{wlldiff}
    \begin{split}
        & (\mathbb{E}^{\alpha} - \mathbb{E}^{\alpha^T})\left[Y_T + \int_0^T e^{-\lambda s}(f_1(s, X, \mu)-f_1(s, X, \mu^T))ds\right] \\
        =& \mathbb{E}^{\alpha}\left[\int_0^T h(s, X, Z^T_s, \alpha^T_s) - h(s, X, Z^T_s, \alpha_s)ds \right] + \mathbb{E}^{\alpha^T}\left[\int_0^T h(s, X, Z_s, \alpha_s) - h(s, X, Z_s, \alpha^T_s)ds\right]
    \end{split}
\end{equation}
where the $f_2$ terms cancel between the different expectations as they are deterministic and $f_3$ becomes part of the redefined Hamiltonian $h$ as in assumption \ref{asmp: cnv rate}. As $Y$ is defined as the value function and since each $f_1(s, \cdot, \cdot)$ satisfies the weak version of Lasry-Lions monotonicity in assumption \ref{asmp: cnv rate}, the left hand side is bounded from above by
$$\frac{2Me^{-\lambda T}}{\lambda} + \delta \int_0^T e^{-\lambda s} (\mathcal{H}(\tilde{\pi}^s(\mu), \tilde{\pi}^{T, s}(\mu^T))+\mathcal{H}(\tilde{\pi}^{T, s}(\mu^T), \tilde{\pi}^s(\mu)))ds.$$
Now using strong concavity of $h$ in $\alpha$, using Lemma \ref{lem: strcv} the right hand side of \eqref{wlldiff} can be lower bounded by $\frac{m}{4}(\mathbb{E}^{\alpha} + \mathbb{E}^{\alpha^T})[\int_0^T e^{-\lambda s}\Vert \alpha_s - \alpha^T_s\Vert_A^2 ds]$. Note that for any $t \leq T$, we have since $\sigma^{-1}b$ is $L$ Lipschitz in $\alpha$ and since the relative entropy decreases under pushforwards (see e.g. \cite{Lehec}) that
\begin{equation*}
    \begin{split}
        &(\mathbb{E}^{\alpha} + \mathbb{E}^{\alpha^T})\left[\int_0^t e^{-\lambda s}\Vert \alpha_s - \alpha^T_s\Vert_A^2 ds\right] \\
        \geq& \frac{1}{L^2}(\mathbb{E}^{\alpha} + \mathbb{E}^{\alpha^T})\left[\int_0^t \vert \sigma^{-1}b(s, X, \mu, \alpha_s) - \sigma^{-1}b(s, X, \mu', \alpha'_s)\vert^2ds\right] \\
        =&\frac{2}{L^2}(\mathcal{H}(\mathbb{P}^{\alpha^T}_{\vert \mathcal{F}_t}, \mathbb{P}^\alpha_{\vert \mathcal{F}_t}) + \mathcal{H}(\mathbb{P}^\alpha_{\vert \mathcal{F}_t}, \mathbb{P}^{\alpha^T}_{\vert \mathcal{F}_t})) \geq \frac{2}{L^2}(\mathcal{H}(\tilde{\pi}^{T, t}(\mu^T), \tilde{\pi}^t(\mu)) + \mathcal{H}( \tilde{\pi}^t(\mu), \tilde{\pi}^{T, t}(\mu^T)).
    \end{split}
\end{equation*}
Using integration by parts, we have for any $t\geq 0$,
\begin{equation*}
    \begin{split}
        & e^{-\lambda t} (\mathbb{E}^{\alpha} + \mathbb{E}^{\alpha^T})\left[\int_0^t \Vert \alpha_s - \alpha^T_s\Vert_A^2 ds\right] + \lambda\int_0^t e^{-\lambda s}(\mathbb{E}^{\alpha} + \mathbb{E}^{\alpha^T})\left[ \int_0^s \Vert \alpha_r - \alpha^T_r \Vert_A^2 dr\right] ds\\
        = &  (\mathbb{E}^{\alpha} + \mathbb{E}^{\alpha^T})\left[\int_0^t e^{-\lambda s}\Vert \alpha_s - \alpha^T_s\Vert_A^2 ds\right].
    \end{split}
\end{equation*}
Since $\delta < \frac{\lambda m}{2L^2}$,
\begin{equation*}
    \begin{split}
        &\frac{2Me^{-\lambda T}}{\lambda} + \delta \int_0^T e^{-\lambda s} (\mathcal{H}(\tilde{\pi}^s(\mu), \tilde{\pi}^{T, s}(\mu^T))+\mathcal{H}(\tilde{\pi}^{T, s}(\mu^T), \tilde{\pi}^s(\mu)))ds\\
        \geq &\frac{m}{4}(\mathbb{E}^{\alpha} + \mathbb{E}^{\alpha^T})\left[\int_0^T e^{-\lambda s}\Vert \alpha_s - \alpha^T_s\Vert_A^2 ds\right]\\
        \geq & \frac{L^2\delta}{2} \int_0^T e^{-\lambda s}(\mathbb{E}^{\alpha} + \mathbb{E}^{\alpha^T})\left[ \int_0^s \Vert \alpha_r - \alpha^T_r \Vert_A^2 dr\right] ds + \left(\frac{m}{4} - \frac{L^2\delta}{2\lambda}\right) e^{-\lambda t} (\mathbb{E}^{\alpha} + \mathbb{E}^{\alpha^T})\left[\int_0^t \Vert \alpha_s - \alpha^T_s\Vert_A^2 ds\right]  \\
        \geq &\delta \int_0^T e^{-\lambda s} (\mathcal{H}(\tilde{\pi}^s(\mu), \tilde{\pi}^{T, s}(\mu^T))+\mathcal{H}(\tilde{\pi}^{T, s}(\mu^T), \tilde{\pi}^s(\mu)))ds\\
        + &\left(\frac{m}{2L^2} - \frac{\delta}{\lambda}\right) e^{-\lambda t}(\mathcal{H}(\tilde{\pi}^{T, t}(\mu^T), \tilde{\pi}^t(\mu)) + \mathcal{H}( \tilde{\pi}^t(\mu), \tilde{\pi}^{T, t}(\mu^T))).
    \end{split}
\end{equation*}
which proves our first statement.
Similarly, we can also find
\begin{equation*}
    \begin{split}
        &\frac{2Me^{-\lambda T}}{\lambda} + \delta \int_0^T e^{-\lambda s} (\mathcal{H}(\tilde{\pi}^s(\mu), \tilde{\pi}^{T, s}(\mu^T))+\mathcal{H}(\tilde{\pi}^{T, s}(\mu^T), \tilde{\pi}^s(\mu)))ds\\
        \geq &\frac{m}{4}(\mathbb{E}^{\alpha} + \mathbb{E}^{\alpha^T})\left[\int_0^T e^{-\lambda s}\Vert \alpha_s - \alpha^T_s\Vert_A^2 ds\right]\geq  \frac{m\lambda}{4} \int_0^T e^{-\lambda s}(\mathbb{E}^{\alpha} + \mathbb{E}^{\alpha^T})\left[ \int_0^s \Vert \alpha_r - \alpha^T_r \Vert_A^2 dr\right] ds\\
        \geq & \delta \int_0^T e^{-\lambda s} (\mathcal{H}(\tilde{\pi}^s(\mu), \tilde{\pi}^{T, s}(\mu^T))+\mathcal{H}(\tilde{\pi}^{T, s}(\mu^T), \tilde{\pi}^s(\mu)))ds \\
        +& \left(\frac{m\lambda}{2 L^2} - \delta\right)\int_0^T e^{-\lambda s}(\mathcal{H}(\mathbb{P}^{\alpha^T}_{\vert \mathcal{F}_s}, \mathbb{P}^\alpha_{\vert \mathcal{F}_s}) + \mathcal{H}(\mathbb{P}^\alpha_{\vert \mathcal{F}_s}, \mathbb{P}^{\alpha^T}_{\vert \mathcal{F}_s}))ds
    \end{split}
\end{equation*}
showing
$$\int_0^T e^{-\lambda s}(\mathcal{H}(\mathbb{P}^{\alpha^T}_{\vert \mathcal{F}_s}, \mathbb{P}^\alpha_{\vert \mathcal{F}_s}) + \mathcal{H}(\mathbb{P}^\alpha_{\vert \mathcal{F}_s}, \mathbb{P}^{\alpha^T}_{\vert \mathcal{F}_s}))ds \leq \frac{1}{\frac{m\lambda}{2L^2} - \delta} \frac{2Me^{-\lambda T}}{\lambda}$$
and
$$(\mathbb{E}^{\alpha} + \mathbb{E}^{\alpha^T})\left[\int_0^T e^{-\lambda s}\Vert \alpha_s - \alpha^T_s\Vert_A^2 ds\right]\leq\left(1 + \frac{\delta}{\frac{m\lambda}{2L^2} - \delta}\right)\frac{8Me^{-\lambda T}}{\lambda m}.$$
Let us now take a look at the law of the control variable. First, note that for any $0 \leq t \leq T$, we have
\begin{equation*}
    \begin{split}
        &2\mathcal{W}_1(q_t, q^T_t) \\
        \leq&  \mathcal{W}_1(\mathbb{P}^\alpha \circ (\alpha_t)^{-1}, \mathbb{P}^\alpha \circ (\alpha^T_t)^{-1}) + \mathcal{W}_1(\mathbb{P}^\alpha \circ (\alpha^T_t)^{-1}, \mathbb{P}^{\alpha^T} \circ (\alpha^T_t)^{-1})\\
        + &  \mathcal{W}_1(\mathbb{P}^\alpha \circ (\alpha_t)^{-1}, \mathbb{P}^{\alpha^T} \circ (\alpha_t)^{-1}) + \mathcal{W}_1(\mathbb{P}^{\alpha^T} \circ (\alpha_t)^{-1}, \mathbb{P}^{\alpha^T} \circ (\alpha^T_t)^{-1})\\
        \leq & C_A (d_{TV} (\mathbb{P}^\alpha \circ (\alpha_t)^{-1}, \mathbb{P}^{\alpha^T} \circ (\alpha_t)^{-1}) + d_{TV}(\mathbb{P}^\alpha \circ (\alpha^T_t)^{-1}, \mathbb{P}^{\alpha^T} \circ (\alpha^T_t)^{-1}))+(\mathbb{E}^\alpha + \mathbb{E}^{\alpha^T})[\Vert \alpha_t - \alpha^T_t \Vert_A]\\
        \leq &2C_A d_{TV}(\mathbb{P}^\alpha_{\vert \mathcal{F}_t}, \mathbb{P}^{\alpha^T}_{\vert \mathcal{F}_t}) + \mathbb{E}^{\alpha}[\Vert \alpha_t - \alpha^T_t \Vert_A^2]^{\frac{1}{2}} + \mathbb{E}^{\alpha^T}[\Vert \alpha_t - \alpha^T_t \Vert_A^2]^{\frac{1}{2}}
    \end{split}
\end{equation*}
where $C_A$ is as before an upper bound on $\Vert \cdot \Vert_A$ on $A$, and we have used in the last line that $\alpha_t$ and $\alpha^T_t$ are essentially measurable maps $\Omega^t \rightarrow A$. Thus, we can bound
$$\int_0^T e^{-\lambda s} \mathcal{W}_1(q_s, q^T_s)^2 ds \leq  2 C_A^2 \int_0^T e^{-\lambda s} d_{TV}(\mathbb{P}^\alpha_{\vert \mathcal{F}_s}, \mathbb{P}^{\alpha^T}_{\vert \mathcal{F}_s})^2 ds + (\mathbb{E}^{\alpha} + \mathbb{E}^{\alpha^T})\left[\int_0^T e^{-\lambda s} \Vert \alpha_s - \alpha^T_s \Vert_A^2 ds\right]$$
and we can conclude using Pinsker's inequality and the bounds above.
\end{proof}
Here, we state the (probably well-known) convex analysis result we used above.
\begin{lemma}\label{lem: strcv}
    Let $\phi: E \rightarrow \mathbb{R}$ be strongly $m$ concave for some $m > 0$ defined on some convex subset $E$ of a normed vector space where the norm is induced by an inner product. Assume that $\phi$ admits on $E$ a minimizer $e^\star$. Then, for any $e \in E$, we have $\Vert e - e^\star \Vert^2 \leq \frac{4}{m}(\phi(e^\star) - \phi(e))$.
\end{lemma}
\begin{proof}
    By assumption, $\phi + \frac{m}{2}\Vert \cdot\Vert^2$ is concave. In the case that $\Vert \cdot \Vert$ is induced by an inner product, $\Vert \cdot \Vert^2 - \Vert \cdot - e^\star \Vert^2$ is linear, thus $\phi + \frac{m}{2}\Vert \cdot - e^\star\Vert^2$ is concave as well. Hence,
    $$\frac{1}{2}\left(\phi(e) + \frac{m}{2}\Vert e - e^\star\Vert^2 + \phi(e^\star)\right) \leq \phi\left(\frac{e + e^\star}{2}\right) + \frac{m}{2}\left\Vert  \frac{e + e^\star }{2} - e^\star \right\Vert^2 \leq \phi(e^\star) + \frac{m}{8}\Vert e - e^\star \Vert^2$$
    and thus $\Vert e - e^\star \Vert^2 \leq \frac{4}{m}(\phi(e^\star) - \phi(e))$.
\end{proof}

\section{The invariant mean field game}\label{sect: StMFG}
The goal of this section is to prove Theorem \ref{thm: inv mfg exst}, and on the way, we will discuss Markovian properties of solutions to the mean field games as well as the existence of stationary distribution. Differently than in Theorem \ref{thm: inv mfg exst} we will for now work with a fixed initial distribution $\upsilon$ instead of the case where the initial distribution becomes part of the solution. Indeed, our approach consists of reducing the invariant mean field game to the so called stationary mean field game that we need to define first.\par
\subsection{Time homogeneous mean field games and feedback controls}
We first consider whether we can always find $\mu$ and $\alpha$ that solve the mean field game in the sense of Definition \ref{def.MFE} so that the state process becomes a time homogeneous Markovian SDE under the controlled measure. In general, both $b$ and $f$ depend on $\alpha$ which was allowed to be any progressively measurable process. Such controls are also known as open loop controls, in what follows, we would like to consider so called feedback controls, in which we can write $\alpha_s = \alpha(X_s)$ for a deterministic function $\alpha : \mathbb{R}^d \rightarrow A$. For such controls, under $\mathbb{P}^{\mu, \alpha}$, the forward process $X$ is a weak solution to a time homogeneous Markovian SDE.\par
Note that for the notion of mean-field equilibria, we are working with the open loop formulation of the game. The notion of a closed loop or feedback equilibrium generally differs from the one we are considering, for a more detailed distinction we refer to \cite[2.1.2]{CarmonaDelarue}. We are merely interested in open loop equilibria with an optimal control that happens to be in feedback form.
\begin{theorem}\label{thm: MKV MFG}
    Assume that our coefficients are time homogeneous and that assumption \ref{asmp: stnd} holds. Then given any $\mu$, there exists an optimal control in feedback form. If additionally, assumption \ref{asmp: sngl} holds, the optimal control is unique and hence in feedback form.
\end{theorem}
\begin{proof}
Recall from Proposition \ref{prop.optim.charac} that the optimal controls are given as the pointwise maximizer of the transformed Hamiltonian, i.e. they satisfy a.e. $\alpha_t \in \mathcal{A}(t, X, \mu, \tilde Z_t)$. The process $Z$ comes from the solution $(Y, Z)$ of the infinite horizon BSDE
$$Y_t = Y_T + \int_t^T( \tilde{H}(X_s, \mu, Z_s) - \lambda Y_s )ds - \int_t^T Z_s dW_s.$$
Using Filippov's implicit function Theorem or its variant as measurable maximum Theorem \cite[Theorem 18.19]{IDA}, we can find a measurable map $\alpha' :\mathbb{R}^d \times \mathbb{R}^d \rightarrow A$ such that $\tilde{h}(x, \mu, \alpha'(x, z), z) = \tilde{H}(x, \mu, z)$ for any $x, z \in \mathbb{R}^d \times \mathbb{R}^d$. Now, by Proposition \ref{prop: Mkv BSDE}, there exists a deterministic function $v:\mathbb{R}^d \rightarrow \mathbb{R}^d$ such that the solution $Z$ fulfills a.e.\ $Z_t = v(X_t)$. We have thus constructed an optimal feedback control $\alpha = \alpha'(X_s, v(X_s))$ for the control problem under $\mu$. Uniqueness under assumption \ref{asmp: sngl} follows from Proposition \ref{prop.optim.charac}.
\end{proof}
In the following, we can thus restrict our set of admissible controls to the set $\mathbb{A}_M$ of feedback controls. In particular, under the assumptions above (including \ref{asmp: sngl}), if a solution to the mean field game exists, its optimal control is in feedback form. Under slight abuse of notation, depending on the context, elements $\alpha$ of $\mathbb{A}_M$ can denote processes both as elements maps $\mathbb{R}^d \rightarrow A$ or the process in $\alpha(X_\cdot) \in \mathbb{A}$.\par

\subsection{Existence of stationary distributions}
Here, we will discuss the dynamics of our state variable for fixed $\mu \in \mathcal{P}(\mathbb{R}^d)$ and $\alpha \in \mathbb{A}_M$.
The advantage of working with the feedback controls is that with them, the coefficients become truly time homogeneous. In particular, for any $\alpha \in \mathbb{A}_M$ our process $X$ under $\mathbb{P}^{\mu, \alpha}$ is a weak solution to the time homogenous Markovian SDE
\begin{equation}\label{wk Mkv SDE}
dX_s = b(X_s, \mu, \alpha(X_s))ds + \sigma(X_s)dW^{\mu, \alpha}_s.
\end{equation}
Under Assumption \ref{asmp: erg}, \cite[Theorem 7.2.1]{StroockVaradhan} and \cite[Theorem 5.4.20]{KS} together imply that $X$ admits the strong Markov property under $\mathbb{P}^{\mu, \alpha}$. Recall that we write $\mathbb{S} = \partial B_0(R)$ and $\mathbb{S}' = \partial B_0(R')$.
\begin{theorem}\label{thm: stat MFG}
    Assume the assumptions \ref{asmp: stnd} and \ref{asmp: erg} hold. If $\alpha \in \mathbb{A}_M$, under $\mathbb{P}^{\mu, \alpha}$, the state process admits a unique stationary distribution $\mu_\infty \in \mathcal{P}(\mathbb{R}^d)$.
\end{theorem}
\begin{proof}
    Under assumption \ref{asmp: erg} (iii), we can apply \cite[Theorem 3.9]{Khasminskii} to find that under $\mathbb{P}^{\mu, \alpha}$, $X$ is recurrent with respect to $B_0(R')$. In fact, the required regularity, i.e.\ that the process does not explode, is an immediate consequence of the drift condition, and the Lyapunov function used can simply be chosen as $\Vert \cdot \Vert^2 - (R')^2$. In particular, under the dynamics of \eqref{wk Mkv SDE}, if started at some point $x_0 \in B_0(R')^\complement$, the expected hitting time of $B_0(R')$ is bounded by $\frac{1}{k}(\Vert x_0 \Vert^2 - (R')^2)$ by \cite[Theorem 3.9]{Khasminskii}.\par
    Therefore, under Assumption \ref{asmp: erg}, the Assumption (B) from \cite[4.4]{Khasminskii} is fulfilled. Thus, by \cite[Theorem 4.1, Corollary 4.4]{Khasminskii}, there exists a unique stationary distribution $\mu_\infty$. Note that in the proof of these results, the authors have implicitly assumed Lipschitzness of $b$ and $\mu$. In the following, we will briefly recall where this regularity is needed. In \cite[Lemma 4.6]{Khasminskii}, the authors apply \cite[Chapter V, §5, case b)]{Doob} to the discretization $\tilde{X}_i := X_{\tau_i}$ for the sequence of stopping times defined by $\tau_0 := 0$ and for $n \geq 1$, $\tau_n := \inf\lbrace t \geq \tau'_n \vert \Vert X_t \Vert = R\rbrace$, and $\tau'_n := \inf \lbrace t \geq \tau_{n-1} \vert \Vert X_t \Vert = R'\rbrace$. By the strong Markov property of $X$, the $(\tilde{X}_i)_{i \geq 1}$ form a $\mathbb{S}$ valued Markov chain. To show the required Döblin condition (D') of \cite[Chapter V, §5, case b)]{Doob}, as an intermediate step, it suffices to first find a constant $\Theta>0$ and a measure $\nu$ on $\mathbb{S}$ such that for any $x \in \mathbb{S}'$ and any Borel $A \in \mathcal{B}(\mathbb{S})$, we have $\mathbb{P}^{\mu, \alpha}_x[X_{\tau_1} \in A] \geq \Theta \nu(A)$.\par
    In \cite[Remark 3.10]{Khasminskii}, using \cite[III.21, VI]{Miranda}, existence of such $\Theta$ is justified for Lipschitz coefficients. In our case without regularity, we will use an argument from the proofs of \cite[Lemma 8]{Veretennikov97} or \cite[Theorem 1]{Malyshkin01} relying on Harnack's inequality for elliptic PDEs with only measurable coefficients. In fact, a stronger result can be shown in a twofold approximation scheme. First, one fixes a compact subset $A \subset \mathbb{S}$, for which by a smooth version of Urysohn's lemma (see e.g.\ \cite[Theorem 2.6.1.]{Conlon}) there exists a sequence of $[0, 1]$ valued functions $\phi_n$ that are smooth in a neighbourhood of $B_0(R)$ and converge pointwise to $\mathbb{1}_A$. Let $L = \frac{1}{2}tr(\Sigma(x) \nabla^2) + b(x, \mu, \alpha(x))^\top \nabla$ be the generator of \eqref{wk Mkv SDE}. By \cite[Theorem 9.15.]{Gilbarg}, for every $n$, there exists $u_n \in W^{2, d}(B_0(R))$ such that $Lu_n = 0$ a.e.\ on $B_0(R)$ and $u_n = \phi_\epsilon$ on $\mathbb{S}$. Here, $W^{2, d}(B_0(R))$ denotes the Sobolev space of square integrable functions on $B_0(R)$ with square integrable weak derivatives of order at most $d$, and the latter equality is to be understood for the continuous version of $u$ that exists by \cite[Corollary 7.11., Theorem 7.25.]{Gilbarg}.\par
    By the generalized Itô's formula \cite[2.10, Theorem 1]{Krylov08}, for any $x \in B_0(R)$, we thus have $u_n(x) = \mathbb{E}^{\mu, \alpha}_x[\phi_n(X_\tau)]$. By Harnack's inequality, as stated in \cite[Corollary 9.25.]{Gilbarg}, and using the proof of \cite[Theorem 2.5.]{Gilbarg} to extend the inequality to any closed subset of $B_0(R)$, there exists a constant $\Theta$ only dependent of $d$, $R'$, $R$ and $\Lambda$ such that uniformly over $A$ and all $n$, we have $\mathbb{E}^{\mu, \alpha}_x[\phi_n(X_{\tau_1})] = u_n(x) \geq \Theta u_n(y) = \Theta\mathbb{E}^{\mu, \alpha}_{y}[\phi_n(X_{\tau_1})]$ for any two $x, y \in \overline{B_0(R')}$. We stress here that this $\Theta$ can be chosen to independently of $\mu$ and $\alpha$. Taking $n \rightarrow \infty$ shows $\mathbb{P}_x^{\mu, \alpha}[X_{\tau_1} \in A] \geq \Theta \mathbb{P}_{y}^{\mu, \alpha}[X_{\tau_1} \in A]$. To conclude for general Borel $A$, note that by \cite[Theorem 13.6]{Klenke}, $(\mathbb{P}_x^{\mu, \alpha} + \mathbb{P}_{y}^{\mu, \alpha} )\circ (X_{\tau_1})^{-1}$ is a Radon measure so that it suffices to apply the previous argument to a sequence of compact sets $A_n \subset A$ that fulfill $(\mathbb{P}_x^{\mu, \alpha} + \mathbb{P}_{y}^{\mu, \alpha} )[X_{\tau_1} \in A \setminus A_n] \rightarrow 0$, showing  $\mathbb{P}_x^{\mu, \alpha}[X_{\tau_1} \in A] \geq \Theta \mathbb{P}_{y}^{\mu, \alpha}[X_{\tau_1} \in A]$ for any $x,y \in \overline{B_0(R')}$ and any Borel $A$. In particular, $\nu$ can be chosen as $\mathbb{P}_x^{\mu, \alpha}\circ(X_{\tau_1})^{-1}$ for any arbitrary $x \in \overline{B_0(R')}$.
\end{proof}
Before we study the speed of convergence, let us take a closer look into the discretization $(\tilde{X}_n)_n$ used in \cite[Lemma 4.6]{Khasminskii} that we considered in the previous proof. We denote its one and $n$ step kernel as $q$ and $q^{(n)}$ under $\mathbb{P}^{\mu, \alpha}$. After bypassing the regularity issue, we can continue following the proof of \cite[Lemma 4.6]{Khasminskii} to derive the Döblin condition $q(x, A) \geq \Theta \nu (A)$ uniformly over $x \in \mathbb{S}$ and $A \in \mathcal{B}(\mathbb{S})$ for the same $\Theta$ and $\nu$ as we have found above. By \cite[Chapter V, §5, case b)]{Doob}, $(\tilde{X}_n)_n$ admits a stationary distribution $q^\infty \in \mathcal{P}(\mathbb{S})$ with $d_{TV}(q^{(n)}, q^\infty) \leq (1-C)^{n-1}$. Let us also note at this point that by \cite[Theorem 3.9, Corollary 3.3]{Khasminskii}, there exists a constant $\kappa > 0$ only dependent on $\Lambda$, $R$, $R'$ and $k$ such that for any $x \in \mathbb{S}$, we have $\mathbb{E}_x[\tau_1] < \kappa$.\par
After finding $q^\infty$, we can see using \cite[Theorem 4.1]{Khasminskii} that if we define $\mu_\infty \in \mathcal{P}(\mathbb{R}^d)$ by
$$\int_{\mathbb{R}^d} \phi d\mu_\infty = \frac{1}{\int_\mathbb{S}\mathbb{E}^{\mu, \alpha}_y[\tau_1]q^\infty(dy)} \int_\mathbb{S}  \mathbb{E}^{\mu, \alpha}_y \left[\int_0^{\tau_1}\phi(X_t) dt \right]q^\infty(dy)$$
for all $\phi$ bounded and measurable, $\mu_\infty$ is the unique stationary measure for $X$ under $\mathbb{P}^{\mu, \alpha}$ we found above.
\begin{theorem}\label{thm: ces cnv}
    Under the same setting as the previous result, the marginal distributions of $X$ converge uniformly towards $\mu_\infty$ in the Césaro sense:
    $$\left\vert \frac{1}{T} \int_0^T \mathbb{E}^{\mu, \alpha} [\phi(X_t) ] dt - \int_{\mathbb{R}^d} \phi(x) \mu_\infty(dx) \right\vert \leq \frac{\Gamma}{T}$$
    for any $T > 0$ and measurable $\phi$ with $\vert \phi \vert \leq 1$ for a constant $\Gamma$ that only depends on $\Lambda$, $R$, $R'$, $k$ and $\bar{\upsilon}$.
\end{theorem}
\begin{remark}
    Under stronger conditions, it is possible to further derive convergence of the marginals to the stationary distribution. Indeed, in case $\sigma$ is uniformly bounded and $k$ is big enough, the marginals converge uniformly with a polynomial rate as it can be seen in \cite{Veretennikov97, Veretennikov00}. In case of bounded coefficients and if $k$ is replaced by $k\Vert x \Vert$, this convergence can be made exponential as seen in \cite{Veretennikov88}.
\end{remark}
\begin{proof}
    Let us fix some $T\geq 0$. We define the random variable $N_T := \min\lbrace n \geq 1 \vert \tau_n \geq T \rbrace$. Then,
    $$\mathbb{E}^{\mu, \alpha}_\upsilon\left[\int_0^T \phi(X_t)dt\right] = \mathbb{E}^{\mu, \alpha}_\upsilon\left[\int_0^{\tau_1} \phi(X_t)dt\right] + \mathbb{E}^{\mu, \alpha}_\upsilon\left[\sum_{i = 1}^{N_T}\int_{\tau_i}^{\tau_{i+1}}\phi(X_t)dt\right] - \mathbb{E}^{\mu, \alpha}_\upsilon\left[\int_T^{\tau_{N_T+1}} \phi(X_t)dt\right]$$
    First, we want to show that $N_T$ is integrable. For this, let us find a uniform lower bound on $\mathbb{E}^{\mu, \alpha}_x[\tau_1]$ for $ x \in \mathbb{S}$. Note that $\mathbb{E}^{\mu, \alpha}_x[\tau_1] \geq \mathbb{E}^{\mu, \alpha}_x[\tau_1 - \tau'_1] = \mathbb{E}^{\mu, \alpha}_x[\mathbb{E}^{\mu, \alpha}_{X_{\tau'_1}}[\tau_1]]$ since for any  initial value on $\mathbb{S}'$, we have $\tau'_1 = 0$. Let us now fix $y \in \mathbb{S}'$ arbitrarily. By Itô's Lemma and assumption \ref{asmp: erg}, we have for any $t \geq 0$
    \begin{equation*}
        \begin{split}
            &\mathbb{E}^{\mu, \alpha}_y[\Vert X_{t \wedge \tau_1}\Vert^2] = (R')^2 + \mathbb{E}^{\mu, \alpha}_y\left[\int_0^{t \wedge \tau_1} (2 X_s)^\top b(X_s, \mu, \alpha(X_s)) + tr(\Sigma(X_s))ds\right]\\
            \leq & (R')^2 + (2R\Lambda + d \Lambda^2)\mathbb{E}^{\mu, \alpha}_y[t \wedge \tau_1].
        \end{split}
    \end{equation*}
    As $\tau_1$ is integrable, we get for $x \in \mathbb{S}$ that $\mathbb{E}^{\mu, \alpha}_x[\tau_1] \geq \mathbb{E}^{\mu, \alpha}_x[\mathbb{E}^{\mu, \alpha}_{X_{\tau'_1}}[\tau_1]] \geq \frac{R^2 - (R')^2}{2R\Lambda + d \Lambda^2} =: \xi >0$. Hence, $T \geq \mathbb{E}_\upsilon^{\mu, \alpha}[\sum_{i = 1}^{N_T - 1} \mathbb{E}^{\mu, \alpha}_{X_{\tau_i}}[\tau_1]] \geq \xi (\mathbb{E}^{\mu, \alpha}_\upsilon[N_T] - 1)$ and $N_T$ is integrable.\par
    By the strong Markov property and Fubini's Theorem, we have
    \begin{equation*}
        \begin{split}
            &\mathbb{E}^{\mu, \alpha}_\upsilon\left[\sum_{i = 1}^{N_T}\int_{\tau_i}^{\tau_{i+1}}\phi(X_t)dt\right] = \mathbb{E}^{\mu, \alpha}_\upsilon\left[\sum_{i = 1}^\infty \mathbb{1}_{\lbrace \tau_{i-1} < T\rbrace } \int_{\tau_i}^{\tau_{i+1}} \phi(X_t) dt\right] \\
            =& \mathbb{E}^{\mu, \alpha}_\upsilon \left[\sum_{i = 1}^\infty \mathbb{1}_{\lbrace \tau_{i-1} < T\rbrace} \mathbb{E}^{\mu, \alpha}_\upsilon\left[\int_{\tau_i}^{\tau_{i+1}} \phi(X_t) dt \vert \mathcal{F}_{\tau_i}\right]\right]=  \mathbb{E}^{\mu, \alpha}_\upsilon \left[\sum_{i = 1}^\infty \mathbb{1}_{\lbrace \tau_{i-1} < T\rbrace} \mathbb{E}^{\mu, \alpha}_\upsilon\left[\int_{\tau_i}^{\tau_{i+1}} \phi(X_t) dt \vert X_{\tau_i}\right]\right] \\
            =&  \mathbb{E}^{\mu, \alpha}_\upsilon \left[\sum_{i = 1}^{N_T}\mathbb{E}^{\mu, \alpha}_{X_{\tau_i}}\left[\int_0^{\tau_1} \phi(X_t) dt\right]\right].
        \end{split}
    \end{equation*}
    As $\mathbb{S} \rightarrow \mathbb{R}, x\mapsto \mathbb{E}^{\mu, \alpha}_x[\int_0^{\tau_1} \phi(X_t)dr]$ is a measurable map bounded by $\kappa$, computing the right hand side reduces to analyzing the discrete chain $(\tilde{X}_i)_i$ together with the discrete filtration $\tilde{\mathcal{F}}_i := \mathcal{F}_{\tau_i}$. Indeed, $(\tilde{X}_i)_i$ is a Markov chain under $\tilde{\mathcal{F}}$, and further, $N_T$ is a $\tilde{\mathcal{F}}$ stopping time. By Lemma \ref{lem: Moustakides},
    $$\left\vert \mathbb{E}^{\mu, \alpha}_\upsilon \left[\sum_{i = 1}^{N_T}\mathbb{E}^{\mu, \alpha}_{X_{\tau_i}}\left[\int_0^{\tau_1} \phi(X_t) dt\right]\right] - \mathbb{E}_{\upsilon}^{\mu, \alpha}[N_T]\int_\mathbb{S} \mathbb{E}^{\mu, \alpha}_y \left[\int_0^{\tau_1}\phi(X_t)dt\right]q^\infty(dy)\right\vert  \leq \frac{2 \kappa}{\Theta}.$$
    Also, note that $\mathbb{E}^{\mu, \alpha}_\upsilon[\int_0^{\tau_1} \phi(X_t) dt - \int_T^{\tau_{N_T}+1} \phi(X_t)dt]$ is bounded by a constant depending only on $\bar{\upsilon}$ and $\kappa$. Hence, there exists a constant $\Gamma$ only depending on $\Lambda$, $R$, $R'$, $k$ and $\bar{\upsilon}$ such that
    $$\left\vert \mathbb{E}^{\mu, \alpha}_\upsilon\left[\int_0^T \phi(X_t)dt\right] - \int_\mathbb{S} \mathbb{E}^{\mu, \alpha}_y\left[\int_0^{\tau_1} \phi(X_t) dt\right] q^\infty(dy)\mathbb{E}^{\mu, \alpha}_\upsilon[N_T] \right\vert \leq \frac{\Gamma}{2}.$$
    For $\phi = 1$, we can thus deduce $\left\vert \frac{1}{\int_\mathbb{S}\mathbb{E}^{\mu, \alpha}_y[\tau_1]q^\infty(dy)} - \frac{\mathbb{E}^{\mu, \alpha}_\upsilon[N_T]}{T}\right\vert \leq \frac{\Gamma}{2T\int_\mathbb{S}\mathbb{E}^{\mu, \alpha}_y[\tau_1]q^\infty(dy)}$. For arbitrary $\phi$, we thus get
    \begin{equation*}
        \begin{split}
            &\left\vert \frac{1}{T} \mathbb{E}^{\mu, \alpha}_\upsilon\left[\int_0^T \phi(X_t)dt\right] -\int \phi(y) \mu_\infty(dy) \right\vert\\
            \leq &\frac{\Gamma}{2T} + \int_\mathbb{S} \mathbb{E}^{\mu, \alpha}_y\left[\int_0^{\tau_1} \phi(X_t) dt\right] q^\infty(dy)\left\vert\frac{\mathbb{E}^{\mu, \alpha}_\upsilon[N_T]}{T} - \frac{1}{\int_\mathbb{S}\mathbb{E}^{\mu, \alpha}_y[\tau_1]q^\infty(dy)}\right\vert            \leq \frac{\Gamma}{T}
        \end{split}
    \end{equation*}
    which concludes the proof.
\end{proof}
Let us here restate the results from \cite[Lemma 2.3, Theorem 2.1]{Moustakides} in the way we use them.
\begin{lemma}\label{lem: Moustakides}
    Let $(\theta_n)_{n\geq 1}$ be a discrete time Markov chain with respect to a filtration $\mathcal{F}$ valued in a Polish space $(E, \mathcal{B}(E))$ with transition kernel $p$ that satisfies a Döblin condition in the sense that there is a constant $\Theta > 0$ and Borel measure $\nu$ on $E$ such that for all $x \in E$ and Borel $A \in \mathcal{B}(E)$, we have $p(x, A) \geq \Theta \nu(A)$. Then, $\theta$ admits a stationary distribution $p^\infty$, and for any integrable $\mathcal{F}$ stopping time and any measurable bounded $g$ on $E$, we have $\vert \mathbb{E}[\sum_{i = 1}^N g(\theta_i)] - \mathbb{E}[N]\int_E g dp^\infty \vert \leq \frac{2 \Vert g \Vert_\infty}{\Theta}$.
\end{lemma}
\begin{proof}
    By \cite[Chapter V, §5, case b)]{Doob}, $p^\infty$ exists, and the $n$ step transition kernels $p^{(n)}$ are geometrically ergodic in the sense that $d_{TV}(p^{(n)}(x, \cdot), p^\infty) \leq (1-\Theta)^{n - 1}$. In particular, we can define $h(x) := \sum_{i = 1}^\infty \int_E h(y) d(p^{(i)}(x, dy) - p^\infty(dy))$ with $\Vert h \Vert_\infty \leq \frac{\Vert g \Vert_\infty}{\Theta}$. One can thus check that in our setting, $\theta$ satisfied the assumptions of \cite[Theorem 2.1]{Moustakides}. Now, although we are working with a larger filtration, one can still be see that $(\sum_{i=1}^n g(\theta_i) - n\int_E gdp^\infty + h(\theta_n))_{n \geq 1}$ is still a martingale. As the increments are uniformly bounded, the result follows from the optional stopping Theorem.
\end{proof}

\subsection{The stationary and the invariant mean field games}
To prove Theorem \ref{thm: inv mfg exst}, as an intermediate step, we are going to first introduce a new kind of game that we call the stationary mean field game. Importantly, for such games, our state process starts again at a fixed initial distribution $\upsilon$. We have chosen this notion since any solution to the stationary mean field game will admit a stationary distribution, but the state process will not necessarily stay invariant. For the definition of the stationary mean field game, recall that for any $\mu \in \mathcal{P}(\mathbb{R}^d)$ and $\alpha \in \mathbb{A}_M$, under $\mathbb{P}^{\mu, \alpha}_\upsilon$, the state process $X$ is a strong Markov process for which a stationary distribution is defined as a distribution that will stay invariant if chosen as the initial distribution of $X$, i.e.\ for any $t \geq 0$, we have $\mathbb{P}^{\mu, \alpha}_\mu \circ (X_t)^{-1} = \mu$.
\begin{definition}
    We call $\mu \in \mathcal{P}(\mathbb{R}^d)$ a solution to the stationary mean-field game, if there is $\alpha = \alpha(X_\cdot) \in \mathbb{A}_M$ so that $V^{\mu} = J^{\mu}(\alpha)$ and $\mu$ is an stationary measure of $X$ under $\mathbb{P}^{\mu, \alpha}_\upsilon$.
\end{definition}
As we will see, the reason we have introduced this kind of game is that is falls into the framework of our previously analyzed games, and at the same time, it is closely related to the invariant mean field game.
\begin{theorem}\label{thm: stMFG}
    Assume assumptions \ref{asmp: stnd}; \ref{asmp: sngl} and \ref{asmp: erg}. Then, there exists a solution to the stationary mean-field game.
\end{theorem}
\begin{proof}
The goal is to embed the stationary mean field game into our general framework. To do so, we need to rewrite the game with coefficients taking arguments from $\mathcal{P}(\mathcal{C})$ instead of $\mathcal{P}(\mathbb{R}^d)$. As we know that the stationary distributions arise as a Césaro limits of the marginal distributions, we will first define a proper subset of $\mathcal{P}(\mathcal{C})$ with elements that admit such a limit.\par
Under our assumptions, using Theorem \ref{thm: ces cnv}, under any $\mathbb{P}^{\mu, \alpha}_\upsilon$, the state process $X$ admits a stationary distribution $\hat{p}^{\mu, \alpha} \in \mathcal{P} (\mathbb{R}^d)$, and $\mathbb{P}^{\mu, \alpha}_\upsilon\circ(X_t)^{-1}$ converges in Cesàro sense towards it. More concretely, we can find a constant $\Gamma$ that only depends on $\Lambda$, $R$, $R'$, $k$ and $\bar{\upsilon}$ such that uniformly for all $\mu, \alpha$,
$$\left\vert \frac{1}{T} \mathbb{E}^{\mathbb{P}^{\mu, \alpha}}_\upsilon\left[\int_0^T \phi(X_t) dt\right] - \int \phi(y) \hat{p}^{\mu, \alpha}(dy)\right \vert \leq \frac{\Gamma}{T}$$
for any measurable $\phi$ with $\Vert \phi \Vert_\infty \leq 1$. Let us for $\mu \in \mathcal{P}(\mathcal{C})$ write $\mu_t = \mu \circ (\omega \mapsto \omega_t)^{-1}$ for the marginal distribution at time $t$. For any $T>0$, let us define $\mathfrak{D}^T: \mathcal{P}(\mathcal{C}) \rightarrow \mathcal{P}(\mathbb{R}^d), \mu \mapsto \frac{1}{T} \int_0^T \mu_t dt$, i.e.\ for any bounded measurable $\phi: \mathbb{R}^d \rightarrow \mathbb{R}$, we would have $\int_{\mathbb{R}^d} \phi(x) \mathfrak{D}^T(\mu)(dx) = \frac{1}{T} \int_0^T \int_{\mathbb{R}} \phi(x) \mu_t(dx) dt = \int_\mathcal{C} \frac{1}{T} \int_0^T \phi(\omega_t) dt \mu(d\omega)$. Then, the previous Cesàro convergence can also be rewritten in the form $d_{TV}(\mathfrak{D}^T(\mathbb{P}^{\mu, \alpha} \circ X^{-1}), \hat{p}^{\mu, \alpha}) \leq \frac{\Gamma}{T}$.\par
Since for any measurable bounded $\phi :\mathbb{R}^d\rightarrow\mathbb{R}$ and any $T$, the map $\mathcal{C} \rightarrow \mathbb{R}^d, \omega \mapsto \frac{1}{T} \int_0^T \phi( \omega_t) dt$ is bounded measurable, we know that the map $\mathfrak{D}^T$ is sequential continuous with respect to the topology of setwise convergence on $\mathcal{P}(\mathbb{R}^d)$.\par
Now, let us define
$$\mathcal{Q}^0 := \lbrace \mu \in \mathcal{Q} \vert \exists \mu_\infty \in \mathcal{P}(\mathbb{R}^d), \forall T: d_{BL}(\mathfrak{D}^T(\mu), \mu_\infty) \leq \frac{\Gamma}{T}\rbrace.$$
By our previous discussion, we already know that for any $\mu, \alpha$, we have $\mathbb{P}^{\mu, \alpha} \in \mathcal{Q}^0$. Also, note that the $\mu_\infty$ above is unique as it is the limit of the Césaro means. We can thus define a map $\mathfrak{S}: \mathcal{Q}^0 \rightarrow \mathcal{P}^0, \mu \mapsto \mu_\infty$.\par
We would like to show that $\mathcal{Q}^0$ is closed. Consider any sequence $(\mu^n)_{n\geq 1} \subset \mathcal{Q}^0$ so that $\mu^n \rightarrow \mu$ for some $\mu \in \mathcal{Q}$. We write $\mu^n_\infty = \mathfrak{S}(\mu^n)$. Let us fix any $A \in \mathcal{B}(\mathbb{R}^d)$. For an arbitrary $\epsilon > 0$, we can find $T>0$ such that $\frac{\Gamma}{T} < \frac{\epsilon}{3}$. For such $T$, we have already shown that $\mathfrak{D}^T(\mu^n) \rightarrow \mathfrak{D}^T(\mu)$ showing that $(\mathfrak{D}^T(\mu^n)(A))_n$ is a Cauchy sequence. There is thus $N>0$ such that for any $n,m \geq N$, we have $\vert \mathfrak{D}^T(\mu^n)(A) - \mathfrak{D}^T(\mu^m)(A)\vert<\frac{\epsilon}{3}$. Thus, for any $n, m \geq N$, we have $\vert\mu^n_\infty(A)-\mu^m_\infty(A)\vert< \epsilon$, proving that $(\mu^n_\infty(A))$ is also a Cauchy sequence. We can thus define the set function $\mu_\infty: \mathcal{B}(\mathbb{R}^d)\rightarrow [0, 1]$ via $\mu_\infty(A) := \lim_{n\rightarrow\infty}\mu^n_\infty(A)$. By the Vitali-Hahn-Saks Theorem \cite[Theorem IX.10]{Doob}, $\mu_\infty$ defines a Borel probability measure. As $\mu_\infty$ is the limit of the $\mu^n_\infty$ with respect to the topology of setwise convergence, we know that for any $A \in \mathcal{P}(\mathbb{R}^d)$ and any $T \geq 0$, we have $\vert \mathfrak{D}^T(\mu)(A) - \mu_\infty(A)\vert = \lim_{n \rightarrow \infty} \vert \mathfrak{D}^T(\mu^n)(A) - \mu^n_\infty(A)\vert \leq \frac{\Gamma}{T}$, thus $d_{TV}(\mathfrak{D}^T(\mu), \mu_\infty) < \frac{\Gamma}{T}$ and $\mu \in \mathcal{Q}^0$. This shows that $\mathcal{Q}^0$ is closed. As it is a subset of $\mathcal{Q}$, this implies $\mathcal{Q}^0$ is compact. Further, as $\mu_\infty = \mathfrak{S}(\mu)$, this shows that $\mathfrak{S}$ is continuous with respect to the topology of setwise convergence on $\mathcal{P}(\mathbb{R}^d)$.\par
Note that by Theorem \ref{thm: ces cnv}, we have for any $\mu$ and $\alpha$ that $\mathfrak{S}(\mathbb{P}^{\mu, \alpha}\circ X^{-1})$ is the stationary distribution for $X$ under $\mathbb{P}^{\mu, \alpha}$.\par
Let us now define the coefficients $\overline{b}: [0, \infty) \times \mathcal{C} \times \mathcal{Q}^0 \times A \rightarrow \mathbb{R}, (t, X, \mu, a) \mapsto b(X_t, \mathfrak{S}(\mu), a)$ and $\overline{f}: [0, \infty) \times \mathcal{C} \times \mathcal{Q}^0 \times A \rightarrow \mathbb{R}, (t, X, \mu, a) \mapsto f(X_t, \mathfrak{S}(\mu), a)$. Using these coefficients, we can define a mean field game in the sense of Theorem \ref{thm: exst NE} as it is no problem to restrict our search for a solution to the compact subset $\mathcal{Q}^0 \subset \mathcal{P}(\mathcal{C})$. Further, setwise continuity of $f$ and $b$ implies that $\overline{f}$ and $\overline{b}$ fulfill assumption \ref{asmp: cont}.\par
We can apply the fixed point argument of Theorem \ref{thm: exst NE}, to find a solution $(\mu, \alpha(X_\cdot)) \in \mathcal{Q}^0 \times \mathbb{A}_M$ to the new mean field game. By construction, $(\mathfrak{S}(\mu), \alpha(X_\cdot))$ is a solution to the stationary mean field game.
\end{proof}
\begin{proof}[Theorem \ref{thm: inv mfg exst}]
    We will show that the solution $(\mu, \alpha(X_\cdot))$ for the stationary game that was obtained in Theorem \ref{thm: stMFG} also solves the invariant mean field game.\par
    Besides the initial distribution, we can see that under $\mathbb{P}^{\mu, \alpha}_\upsilon$ and $\mathbb{P}^{\mu, \alpha}_\mu$, the state process $X$ solves the same SDE. Hence, by our notion of stationary measures, it is immediate that if $\mu$ is a stationary measure for $X$ under $\mathbb{P}^{\mu, \alpha}_\upsilon$, it is invariant under $\mathbb{P}^{\mu, \alpha}_\mu$, and vice versa.\par
    It remains to check that $\alpha$ is optimal for the control problem given $\mu$. We know that an optimal control is given by the maximizer of the Hamiltonian. For the control problem we have solved
    $$Y_t = Y_T + \int_t^T(H(X_s,\mu, Z_s) - \lambda Y_s)ds - \int_t^T Z_s dW_s$$
    and that its unique solution satisfies $Z_s = v(X_s)$ or some deterministic $v: \mathbb{R}^d \rightarrow \mathbb{R}^d$.\par
    Now, since $\alpha$ was optimal for the stationary mean-field game, we know that the map $\alpha: \mathbb{R}^d \rightarrow A$ was found such that for any $x$, we a.e.\ have $h(x, \mu, \alpha(x), v(x)) = H(x, \mu, v(x))$. In particular, this is satisfied pointwise and does not depend on the initial distribution. This implies that $\alpha$ is an optimal control for the stationary mean-field problem which in conclusion proves that $(\mu, \alpha)$ is a solution to the stationary mean-field game.
\end{proof}

\appendix

\section{The locally completed filtration}\label{apdx: filtr}
Recall that we have defined our locally completed filtration as $\mathbb{F} = (\mathcal{F}_t)_{t\ge 0}:= (\cap_{\mathbb{Q} \in \mathfrak{E}} \mathcal{F}^\mathbb{Q}_t)_{t\ge 0}$ where $\mathcal{F}^\mathbb{Q}$ is the $\mathbb{Q}$-completed natural filtration over the family $\mathbb{Q} \in \mathfrak{E}$ of locally equivalent probability measures. The underlying $\sigma$ algebra $\mathcal{F}$ of our probability space is completed the same way. 
In this first appendix we give some elementary properties of this filtration and show why we think it is the right choice of filtration for our control problem.\par
One major problem with the standard completion of filtrations is that two given distinct measures $\mathbb{Q}$ and $\mathbb{Q}'$ in the set $\mathfrak{E}$ can, in general, not be well defined on each others completions $\mathcal{F}^{\mathbb{Q}}$ and $\mathcal{F}^{\mathbb{ Q}'}$. Our construction fixes this issue.
\begin{proposition}
\label{Prop.A1}
    Any $\mathbb{Q} \in \mathfrak{E}$ admits a unique extension on $\mathcal{F}$. 
    Moreover, elements of $\mathfrak{E}$ remain locally equivalent with respect to $\mathbb{F}$ in the sense that for any $\mathbb{Q}, \mathbb{Q}' \in \mathfrak{E}$ and any $T > 0$, we have $\mathbb{Q}_{\vert \mathcal{F}_T} \sim \mathbb{Q}'_{\vert \mathcal{F}_T}$.
\end{proposition}
\begin{proof}
Any $\mathbb{Q} \in \mathfrak{E}$ admits a unique extension on $\mathcal{F}^\mathbb{Q}$ and thus on $\mathcal{F} \subset \mathcal{F}^\mathbb{Q}$ as well. Recall that any $A \in \mathcal{F}^\mathbb{Q}$ can be written as $A = B \cup N$ with $B \in \mathcal{F}^0$ and $N \in \mathcal{N}^\mathbb{Q}$ and the extension is defined as $\mathbb{Q}[A] := \mathbb{Q}[B]$. For any $T > 0$ and any $A \in \mathcal{F}^\mathbb{Q}_T$, by \cite[Problem 2.7.3]{KS}, we can always find $B \in \mathcal{F}^0_T$ so that $A \triangle B \in \mathcal{N}^\mathbb{Q}$ and for any such $B$, we have $\mathbb{Q}[A] = \mathbb{Q}[B]$ where $\triangle$ denotes the symmetric difference $A \triangle B = (A \cup B) \setminus (A \cap B)$.\par
For local equivalence, consider any two measures $\mathbb{Q}, \mathbb{Q}' \in \mathfrak{E}$ and $T > 0$. First, note that $\mathfrak{E}$ is a convex subspace of the vector space of signed finite measures on $\Omega$. In particular, $\bar{\mathbb{Q}} := \frac{\mathbb{Q} + \mathbb{Q}'}{2} \in \mathfrak{E}$. For such, even on infinite horizon, we have $\mathbb{Q}, \mathbb{Q}' \ll \bar{\mathbb{Q}}$. Now, consider any $A \in \mathcal{F}_T$ so that $\mathbb{Q}'[A] = 0$. Since $A \in \mathcal{F}_T \subset \mathcal{F}^{\bar{\mathbb{Q}}}_T$, there exists $B \in \mathcal{F}^0_T$ so that $A \triangle B \in \mathcal{N}^{\bar{\mathbb{Q}}}$. Thus, we have both $A \triangle B \in \mathcal{N}^{\mathbb{Q}'}$ and $A \triangle B \in \mathcal{N}^\mathbb{Q}$ so that $\mathbb{Q}[A] = \mathbb{Q}[B]$ and $\mathbb{Q}'[A] = \mathbb{Q}'[B] = 0$. Since $\mathbb{Q}_{\vert \mathcal{F}^0_T} \sim \mathbb{Q}'_{\vert \mathcal{F}^0_T}$ by assumption, we have $\mathbb{Q}[B] = \mathbb{Q}'[B] = 0$ and hence $\mathbb{Q}[A] = 0$ as well. This shows $\mathbb{Q}_{\vert \mathcal{F}_T} \sim \mathbb{Q}'_{\vert \mathcal{F}_T}$.
\end{proof}
An important consequence is that for $\mathbb{F}$ adapted processes, the notion of indistinguishability coincides for all $\mathbb{Q} \in \mathfrak{E}$.\par

Next, notice that $\mathbb{F}$ can indeed be understood as some sort of locally completed filtration. Let us define $\mathcal{N}^T := \lbrace A \subset \Omega \vert \exists B \in \mathcal{F}^0_T: \mathbb{P}[B] = 0, A\subset B \rbrace$. Note that for any locally equivalent measure $\mathbb{Q} \in \mathfrak{E}$ the sets in $\mathcal{N}^T$ are $\mathbb{Q}$-negligible as well, and we could have also defined $\mathcal{N}^T$ using $\mathbb{Q}$.
We denote by $\mathcal{N}_0 = \cup_{T \geq 0} \mathcal{N}^T$ the set of all negligible sets within a finite horizon. 
By definition, we have $\mathcal{N}_0 \subset \mathcal{F}_0$.
\begin{proposition}
\label{Prop.FF.right.cont}
    $\mathbb{F}$ is a right continuous filtration.
\end{proposition}
\begin{proof}
As the intersection of $\sigma$ algebras is a $\sigma$ algebra, $\mathbb{F}$ is again an increasing family of $\sigma$ algebras and hence a filtration. To prove that it is right continuous, let us, for any $t \geq 0$ consider any $A \in \cap_{\epsilon > 0} \mathcal{F}_{t + \epsilon}$. Now, for any $\mathbb{Q} \in \mathfrak{E}$ and $\epsilon > 0$, there exists $B^\epsilon \in \mathcal{F}^0_{t + \epsilon}$ with $A \triangle B^\epsilon \in \mathcal{N}^\mathbb{Q}$. In particular, for any $N \geq 1$, we have $A \triangle (\cup_{n \geq N} B^{\frac{1}{n}}) \subset \cup_{n \geq N} (A \triangle B^{\frac{1}{n}}) \subset \mathcal{N}^{\mathbb{Q}}$. Thus, if we define $B := \cap_{N \geq 1} \cup_{n \geq N} B^{1/n} \in \cap_{\epsilon > 0} \mathcal{F}^0_{t + \epsilon}$, we have
$$A \triangle B = A^\complement \triangle (\cup_{N\geq 1} (\cup_{n \geq N} B^{\frac{1}{n}})^\complement) \subset \cup_{N \geq 1} (A^\complement \triangle(\cup_{n \geq N} B^{\frac{1}{n}})^\complement ) = \cup_{N \geq 1} (A \triangle(\cup_{n \geq N} B^{\frac{1}{n}})) \subset \mathcal{N}^\mathbb{Q}.$$
Now, by \cite[Theorem 7.2.2.]{Durrett}, if we define $B' = \lbrace \omega \vert \mathbb{E}[\mathbb{1}_{B}\vert \mathcal{F}^0_t](\omega) = 1 \rbrace \in \mathcal{F}^0_t$, we have $\mathbb{P}[B \triangle B']
=0$. As $B \triangle B' \in \cap_{\epsilon > 0} \mathcal{F}^0_{t + \epsilon}$, by local equivalence, we also have $\mathbb{Q}[B \triangle B'] = 0$ showing $B \triangle B' \in \mathcal{N}^\mathbb{Q}$. Conclusively, we have $A \triangle B' = (A \triangle B) \triangle (B \triangle B') \subset (A \triangle B) \cup (B \triangle B')$ and thus $A \triangle B' \in \mathcal{N}^\mathbb{Q}$ which shows $A \in \mathcal{F}^\mathbb{Q}_t$.
Since $\mathbb{Q} \in \mathfrak{E}$ was arbitrary, this shows $A \in \mathcal{F}_t$ and thus right continuity of $\mathcal{F}$.
\end{proof}
The finite horizon filtration $(\mathcal{F}_t)_{t\in [0, T]}$ thus satisfies the usual conditions for every $T > 0$. It is thus possible to apply any of the standard stochastic calculus operations that are contained within a finite time horizon such as 
stochastic integration or martingale inequalities.
The main reason we are considering the set $\mathfrak{E}$ is to consider Wiener processes with drift. 
Now that we have proven right continuity, by \cite[Theorem 3.3.4]{JS}, for any $\mathbb{Q} \in \mathfrak{E}$, there exists a $(\mathbb{P}, \mathbb{F})$ martingale $D^\mathbb{Q}$ so that for any $t \geq 0$, we have $\frac{d\mathbb{Q}_{\vert \mathcal{F}_t}}{d\mathbb{P}_{\vert \mathcal{F}_t}} = D^\mathbb{Q}_t$. 
By \cite[Theorem 3.3.24]{JS}, there exists a predictable process $\beta^\mathbb{Q}$ so that $W^\mathbb{Q}:= W - \beta^\mathbb{Q}$ is a $\mathbb{Q}$ Wiener process. 
Another interesting property of $\mathbb{F}$ is that these Wiener processes are compatible with $\mathbb{F}$. Clearly, $W^\mathbb{Q}$ is adapted with respect to $\mathbb{F}$, and since $\mathbb{F} \subset \mathbb{F}^\mathbb{Q}$, any $\mathcal{F}_t$ stays independent of $(W^\mathbb{Q}_{t+s} - W^\mathbb{Q}_t)_{s \geq 0}$ under $\mathbb{Q}$.\par
An important property that we require is the predictable representation property, or whether the martingale representation Theorem holds. We can see that this immediately follows from the finite horizon version of this property.
\begin{proposition}\label{prop: prp}
    For any $\mathbb{Q} \in \mathfrak{E}$, let $M$ be a $(\mathbb{Q}, \mathbb{F})$ local martingale. Then, there exists a $\mathbb{F}$ progressively measurable process $H$ such that for any $t \geq 0$,
    $$M_t = M_0 + \int_0^t H_s dW^{\mathbb{Q}}_s.$$
\end{proposition}
\begin{proof}
Let $t>0$ be fixed and note that $\mathbb{F}_{\vert [0, t]}$ satisfies the usual conditions for the probability space $(\Omega, \mathbb{Q}, \mathcal{F}_t)$, on which $M_{\vert [0, t]}$ is a $(\mathbb{Q}, \mathbb{F}_{\vert [0, t]})$ local martingale. Thus, it admits a c\`adl\`ag modification. We can also consider $\mathbb{F}^t = (\sigma(\mathcal{F}^0_s \cup \mathcal{N}^s))_{\vert s \in [0, t]}$ as the finite time completion of the natural filtration. As $\mathbb{F}^t$ is right continuous, by \cite[Exercise II.2.17]{RY}, one can consider a c\`adl\`ag modification of $M$ that is adapted to $\mathbb{F}^t$. Now, $W^{\mathbb{Q}}$ is also indistinguishable from a $\mathbb{F}^t$ adapted process. In such a setting, it is well known that $W^\mathbb{Q}$ has the predictable representation property with respect to $\mathbb{F}^t$, see e.g.\ \cite[Theorem 14.5.1]{CE}. Hence, there is an $\mathbb{F}^t$ and thus also $\mathbb{F}$-progressively measurable process $H$, such that for $s \in [0, t]$,
$$ M_s = M_0 + \int_0^s H_r dW^\mathbb{Q}_r.$$
By considering arbitrary big $t$, such a process $H$ can be defined on the whole time interval $[0, \infty)$.
\end{proof}

\section{Infinite time horizon BSDEs with strictly monotonous drifts}\label{apdx: IBSDE}
\subsection{Wellposedness and stability}
Let us rephrase some existing results for infinite time horizon BSDEs such as those of \cite{Royer04} within our context. For a given $\mathbb{F}$ progressively measurable driver $F: [0, \infty) \times \Omega \times \mathbb{R} \times \mathbb{R}^d \rightarrow \mathbb{R}$, we say that progressively measurable $\mathbb{R}$ and $\mathbb{R}^d$ valued processes $Y$ and $Z$ are a solution to the infinite horizon BSDE if for any $t \leq T$, we have
\begin{equation}\label{gnr inft BSDE}
    Y_t = Y_T + \int_t^T F(s, \omega, Y_s, Z_s) ds - \int_t^T Z_s dW_s.
\end{equation}
There are different sets of assumptions that can guarantee existence and uniqueness for solutions of \eqref{gnr inft BSDE}. 
We will work with the version as stated in \cite[Theorem 2.1]{Royer04}:
\begin{theorem}\label{thm: Royer}
Assume that $F$ satisfies
\begin{itemize}
    \item $F$ is uniformly Lipschitz in $z$ over all $t, \omega, y$ with Lipschitz constant $C$.
    \item $F$ is continuous in $y$ and there is a continuous increasing $\phi: [0, \infty) \rightarrow [0, \infty)$ such that for any $t, \omega, y, z$ we have $\vert F(t, \omega, y, z) \vert \leq \vert F(t, \omega, 0, z) \vert + \phi(\vert y\vert)$
    \item There is $\lambda>0$ such that $(y-y')(F(t, \omega, y, z) - F(t, \omega, y', z)) \leq - \lambda(y - y')^2$ for any $t, \omega, y, y', z$. 
    \item There is $M >0$ such that for any $t, \omega$ we have $\vert F(t, \omega, 0, 0) \vert \leq M$
\end{itemize}
Then, there exists a unique solution $(Y, Z)$ to \eqref{gnr inft BSDE} within the class of $Y$ that are uniformly bounded and $Z$ that are $dt \times \mathbb{P}$ square integrable on $[0, T] \times \Omega$ for any $T>0$. Further, this solution satisfies 
$$|Y_t|\le \frac{M}{\lambda}\quad \text{and}\quad \mathbb{E}\left[\int_0^\infty e^{-2\lambda s} \Vert Z_s \Vert^2 ds \right] < \infty.
$$
\end{theorem}
Additionally, both $Y$ and $Z$ can be approximated by solutions of finite horizon BSDEs. As proven in \cite[Theorem 1.2.]{Pardoux98}, for any $T\geq 0$, the BSDEs on the finite horizon $[0, T]$
\begin{equation}\label{gnr fnt BSDE}
    Y_t = \int_t^T F(s, \omega, Y_s, Z_s)ds - \int_t^T Z_s dW_s
\end{equation}
admits a unique solution $(Y^T, Z^T)$. 
In the following, we will adapt the usual convention that we extend solutions to finite horizon BSDEs to the whole time interval $[0, \infty)$ by defining $Y^T_t = 0$ and $Z^T_t = 0$ for $t \geq T$.
The proof of \cite[Theorem 2.1]{Royer04} additionally shows
\begin{lemma}\label{lem: YZ aprx}
    For $(Y^T, Z^T)$ and $(Y, Z)$ as above, we can approximate
    $$\vert Y_t - Y_t^T \vert \leq \frac{M}{\lambda} e^{\lambda(t -T)}\quad \mathbb{P}\text{-a.s.}$$
    and
    $$\mathbb{E}\left[\int_0^\infty e^{-2\lambda s} \Vert Z_s - Z^T_s  \Vert^2 ds \right] \leq \overline{M} (1 + T)  e^{-2\lambda T}$$
    for some constant $\overline{M}>0$ not depending on $T$, but fully determined by $M$, $\lambda$, $\phi$ and the Lipschitz constant of $F$ in $z$.
\end{lemma}
Note that in \cite{Royer04}, they are working with a different filtration. Namely, their $\mathcal{F}_0$ is trivial, and further, the whole filtration is completed in the usual manner, to locally like we do. The results are still applicable in our framework. Considering the proof of \cite[Theorem 2.1]{Royer04}, we can see that in our filtration, the finite time approximations $(Y^T, Z^T)$ still exist and Lemma \ref{lem: YZ aprx} can still be shown to hold which is sufficient to conclude existence. In particular, all the $(Y^T, Z^T)$ and their limit are adapted to our filtration $\mathbb{F}$. The proof of uniqueness does not rely on these properties of the filtration as well as well.\par
In particular, not only the existence and uniqueness result transfer to our setup, but also the comparison result \cite[Theorem 2.2]{Royer04}.\par
Lastly, let us prove an infinite horizon version of the well known stability property of BSDEs. Here, we have deliberately chosen a form that suits our setup.
\begin{lemma}\label{lem: inft bsde stab}
    Let us be given drivers $(F^n)_{n \geq 1}$ and $F$ that satisfy the assumptions above with common constants. Additionally, assume $F^n$ and $F$ are uniformly Lipschitz in $y$ as well. Let $(Y^n, Z^n)$ and $(Y, Z)$ be the unique solutions of the respective infinite horizon BSDEs. Further, we assume that for any $0 < t < T$, we have $\mathbb{E}[(\int_t^T  F^n(s, \omega, Y_s, Z_s) - F(s, \omega, Y_s, Z_s)ds)^2] \rightarrow 0$. Then, for any $t > 0$, we have $Y^n_t \rightarrow Y_t$ a.s. and
    $$\mathbb{E}\left[\int_0^\infty e^{-2\lambda s} \Vert Z^n_s - Z_s \Vert^2 ds\right] \rightarrow 0.$$
\end{lemma}
\begin{proof}
Let $\overline{M}$ be as in Lemma \ref{lem: YZ aprx}. By assumption, the drivers of the BSDEs that $(Y^n, Z^n)$ and $(Y, Z)$ solve share a common Lipschitz constant $C$ in $z$, a common monotonicity constant $\lambda$ in $y$ and are all uniformly bounded by $M$. Particularly, Lemma \ref{lem: YZ aprx} holds uniformly over all $n$. Consider an arbitrary $\epsilon > 0$. Let us pick a $T \geq 0$ such that $\overline{M} (1 + T) e^{-2\lambda T} < \frac{\epsilon}{9}$. We have
\begin{equation*}
    \begin{split}
        &\mathbb{E}\left[\int_0^\infty e^{-2\lambda s} \Vert Z^n_s - Z_s \Vert ^2 ds\right] \leq  3 \mathbb{E}\left[\int_0^\infty e^{-2\lambda s} \Vert Z^n_s - Z^{n, T}_s \Vert ^2 ds\right] \\
        + & 3 \mathbb{E}\left[\int_0^\infty e^{-2\lambda s} \Vert Z^{n, T}_s - Z^T_s \Vert ^2 ds\right] + 3 \mathbb{E}\left[\int_0^\infty e^{-2\lambda s} \Vert Z^T_s - Z_s \Vert ^2 ds\right]
    \end{split}
\end{equation*}
where $Z^{n, T}$ and $Z^T$ are the respective finite time horizon approximations as in Lemma \ref{lem: YZ aprx}. By our choice of $T$, we have using Lemma \ref{lem: YZ aprx} that
\begin{equation*}
    \begin{split}
        & 3\mathbb{E}\left[\int_0^\infty e^{-2\lambda s} \Vert Z^n_s - Z^{n, T} \Vert^2 ds\right] + 3\mathbb{E}\left[\int_0^\infty e^{-2\lambda s} \Vert Z_s - Z^T \Vert^2 ds\right]\\
        \leq& 6\overline{M}(1 + T) e^{-2\lambda T} <\frac{2 \epsilon}{3}
    \end{split}
\end{equation*}
for any $n$. Now also on finite time horizon, the drivers $F^n$ and $F$ share a common Lipschitz constant $C$ in $z$, a common monotonicity constant $\lambda$ in $y$ and are all uniformly bounded by $M$. For these finite time horizon BSDE with monotonous drift that we have additionally assumed to be Lipschitz, we can apply \cite[Theorem 2.1]{HuPeng} to find $\mathbb{E}\left[\int_0^T \Vert Z^{n, T}_s - Z^T_s \Vert^2 ds\right] \rightarrow 0$. For big $n$, we thus have
$$\mathbb{E}\left[\int_0^\infty e^{-2\lambda s} \Vert Z^{n, T}_s - Z^T_s \Vert^2 ds\right] < \frac{\epsilon}{9}.$$
Together, this shows
$$\mathbb{E}\left[\int_0^\infty e^{-2\lambda s} \Vert Z^n_s - Z_s \Vert^2 ds\right] < \epsilon$$
and thus convergence to zero.\par
For the $Y^n$, let us also consider $Y^{n, T}$ and $Y^T$ as finite time horizon approximations as in Lemma \ref{lem: YZ aprx}. It is again possible to choose $T$ big enough such that $\vert Y^n_t - Y^{n, T}_t \vert < \frac{\epsilon}{3}$ and $\vert Y_t - Y^T_t \vert < \frac{\epsilon}{3}$. Further for such $T > 0$, \cite[Theorem 2.1]{HuPeng} again implies $\vert Y^{n, T}_t - Y^T_t \vert < \frac{\epsilon}{3}$ for sufficiently big $n$. Together, this also shows $Y^n_t \rightarrow Y_t$.
\end{proof}
\begin{remark}
    If we instead assume for any $y, z$ we have $dt \times \mathbb{P}$ a.e.\ that $F^n(t, \omega, y, z) \rightarrow F(t, \omega, y, z)$ we can prove the convergence result without additionally assuming Lipschitz properties in $y$ of the drivers using \cite[Theorem 1.3]{Pardoux98}.
\end{remark}
\subsection{Markovian properties}
Here, we will discuss the case where the randomness of $F$ only stems from a forward time homogeneous Markov SDE of the form
$$dX_s = b(X_s)ds + \sigma(X_s)dW_S$$
with a random intitial value $X_0 = \xi$. We assume that $b$ and $\sigma$ are chosen such that $X$ fulfills the strong Markov property (it e.g.\ suffices if the martingale problem is well posed). In this case, for the backward equation, we assume that the driver is given as $F(s, \omega, y, z) = G(X_s(\omega), y, z)$. Any randomness or time dependence thus can only enter through the Markov process $X$. We assume that even defined this way, $G$ is taken such that $F$ satisfies the previously found conditions for existence and uniqueness of the infinite horizon BSDE.\par
\begin{proposition}\label{prop: Mkv BSDE}
    Let $(Y, Z)$ be the unique solution for the BSDE described above. Then, there exist deterministic measurable functions $u:\mathbb{R}^d\rightarrow \mathbb{R}$, $v:\mathbb{R}^d \rightarrow \mathbb{R}^d$ independent of the initial distribution, such that for any $t \geq 0$, a.s. $Y_t = u(X_t)$ and $dt \times \mathbb{P}$ a.e.\ $Z_t = v(X_t)$.
\end{proposition}
\begin{proof}
First, we are proving the property for $Y$. Consider any $t \geq 0$ and a $\mathcal{F}_t$ measurable $\mathbb{R}^d$ valued $\chi$. We write $X^{t, \chi}$ for the SDE that starts at $t$ with $X_t = \chi$. We can consider the filtration $\mathbb{F}^t = (\mathcal{F}^t_{s})_{s\geq 0}$ as the local completion of $(\sigma(\chi, W^t_{[0, s]}))_{s\geq 0}$ with $W^t_{\cdot} = W_{t + \cdot} - W_t$. Note that $W^t$ is a $\mathbb{F}^t$ Wiener process. Then, the infinite horizon BSDE
$$Y^{t, \chi}_s = Y^{t, \chi}_T + \int_s^T G(X^{t, \chi}_r, Y^{t, \chi}_r, Z^{t, \chi}_r)dr - \int_s^T Z^{t, \chi}_r dW^t_r$$
for any $t \leq s \leq T < \infty$ is still uniquely solvable. As $(Y^{t, \chi}, Z^{t, \chi})$ needs to be $\mathbb{F}^t$ adapted, by the Doob-Dynkin Lemma, there exists a deterministic measurable function $u^t:\mathbb{R}^d\rightarrow\mathbb{R}$ so that $Y^{t, \chi}_t = u^t(\chi)$.\par
Let us take $\chi = x$ for a fixed $x \in \mathbb{R}^d$ for now. Then, $Y^{t, x}_{t + \cdot}$ and $Y^{0, x}$ solve the same infinite horizon BSDE. In fact, as $(X^{t, x}_{t + \cdot}, W^t)$ has the same law as $(X^{0, x}, W)$, we can construct these two solutions on a common space. As we assumed pathwise uniqueness for the BSDE, this implies $u^t(x) = Y^{t, x}_t = Y^{0, x}_0 = u^0(x)$. Thus, the $u^t$ are independent of $t$, and we simply write $u$ instead.\par
Now, take $\chi$ as $X_t$ where $X$ is the SDE starting at $0$. As $X$ and $X^{t, X_t}$ are indistinguishable on $[t, \infty)$, both $(Y, Z)$ and $(Y^{t, X_t}, Z^{t, X_t})$ are solution to the same infinite horizon BSDE on $[t, \infty)$. By uniqueness of the solution to the infinite horizon BSDE, we must a.s. have $Y = Y^{t, X_t}$ on $[t, \infty)$, in particular $Y_t = Y^{t, X_t}_t = u(X_t)$.\par
For $Z$, let us first note that by above, $Y_t - Y_0 = u(X_t) - u(X_0)$ is a strongly additive continuous additive semimartingale functional of $X$. Further, by assumption, $W_t - W_0$ is a strongly additive continuous additive semimartingale as well, where we note that $W_0$ does not need to stay zero after a time shift. Thus, by \cite[Proposition 3.56, Theorem 6.27]{Cinlar80} using (iii) in convention \cite[Convention 3.6]{Cinlar80} on $Y_t - Y_0 \pm (W^i_t - W^i_0)$ for each $1 \leq i \leq d$, we can find deterministic measurable functions $v^i:\mathbb{R}^d \rightarrow \mathbb{R}$ such that the quadratic covariation can a.s. be written as $\langle Y_t - Y_0, W^i_t - W^i_0 \rangle = \int_0^\cdot Z^i_s ds = \int_0^\cdot v^i(X_s)ds$. Now, by Lebesgue's differentiation theorem, if we write $v = (v^i)_{1 \leq i \leq d}$, must have a.e.\ that $Z_t = v(X_t)$.
\end{proof}

\section{Set-valued maps}\label{apdx: setvm}
Here, we want to clarify the different continuity notions that we are working with. Specifically, we focus on upper hemicontinuity, which is a common notion in game theory and control problems. Throughout the  literature, there are different definitions which we will specify now. We will write set-valued maps $\phi$ from some space $\mathcal{X}$ to $\mathcal{Y}$ with the double arrow $\phi: \mathcal{X} \twoheadrightarrow \mathcal{Y}$.\par
Let us begin with the version as introduced in \cite{IDA} that we will distinguish from our other notion by calling it topologically upper hemicontinuous.
\begin{definition}
    A map $\phi: \mathcal{X} \twoheadrightarrow \mathcal{Y}$ between topological spaces is called topologically upper hemicontinuous if for any $x \in \mathcal{X}$ and any open set $O \subset \mathcal{Y}$ with $\phi(x) \subset O$ there exists an open set $U \subset \mathcal{X}$ containing $x$ such that for any $x' \in U$, we have $\phi(x') \subset O$.
\end{definition}
This notion is particularly useful in regard of selection or fixed point theorems.\par
In case that $\mathcal{X}$ and $\mathcal{Y}$ are metric, \cite{Carmona15} introduced a slightly different notion.
\begin{definition}
    We call $\phi: (\mathcal{X}, d_\mathcal{X}) \twoheadrightarrow (\mathcal{Y}, d_\mathcal{Y})$ metrically upper hemicontinuous if for all $x\in \mathcal{X}$ and $\epsilon > 0$, there exists $\delta > 0$ such that $\phi(B_x(\delta)) \subset B_{\phi(x)}(\epsilon)$ with the open balls taken with the respective metrics. 
\end{definition}
One can immediately see that $\phi$ being topologically upper hemicontinuous implies that it is metrically upper hemicontinuous as well, since $B_{\phi(x)}(\epsilon) = \cup_{y \in \phi(x)} B_y(\epsilon)$ is an open set containing $\phi(x)$. It is often more convenient to work with metrically hemicontinuous upper maps.\par 
Fortunately, in the case $\phi(x)$ is compact set-valued, the two notions are in fact equivalent. To prove the converse, let us consider some $\phi:\mathcal{X} \twoheadrightarrow \mathcal{Y}$ that is metrically upper hemicontinuous and compact set-valued. Fix $x\in \mathcal{X}$ and an arbitrary open set $\phi(x) \subset O \subset \mathcal{Y}$. Then, the function $\phi(x) \rightarrow \mathbb{R}, y \mapsto d(y, O^c) := \inf_{u \in O^c} d(y, u)$ is continuous and by construction strictly positive. As $\phi(x)$ is compact it has a strictly positive minimum $\epsilon$ so that $\phi(x) \subset B_{\phi(x)}(\epsilon) \subset O$. For upper hemicontinuous $\phi$ there must thus exist a $\delta$ such that $\phi(B_x(\delta)) \subset B_{\phi(x)}(\epsilon) \subset O$ proving that $\phi$ is topologically upper hemicontinuous as well. In particular, in case of compact set-valued maps, we do not distinguish between these two notions and only call them upper hemicontinuous.\par
We can also consider normal, singleton valued functions as set-valued functions. As a function between general topological spaces, any continuous singleton valued function is immediately topological upper hemicontinuous. Particularly in case it maps between metric spaces, it is directly metrically upper hemicontinuous as well. That these two notions coincide in this case can also be justified by the compactness of points.\par
We can also rewrite the definition of metrically upper hemicontinuous in an equivalent more convenient form.
\begin{proposition}\label{prop: metr uhc}
    $\phi:(\mathcal{X}, d_\mathcal{X}) \twoheadrightarrow (\mathcal{Y}, d_\mathcal{Y})$ is metrically upper hemicontinuous iff for any sequence $(x^n)_{n \geq 1} \subset \mathcal{X}$ that converges towards some $x \in \mathcal{X}$
    $$\sup_{y^n \in \phi(x^n)}\inf_{y \in \phi(x)} d_\mathcal{Y}(y^n, y) \rightarrow 0.$$
\end{proposition}
\begin{proof}
    We consider any $(x^n)_{n \geq 1}$ that converge to some $x$. First, assume that $\phi$ is metrically upper hemicontinuous as in our first definition. Let $\epsilon > 0$ be arbitrary, then there is $\delta > 0$ such that $\phi(B_x(\delta)) \subset B_{\phi(x)}(\epsilon)$. As $(x^n)_{n \geq 1}$ converge to $x$, there is $N$ so that for $n \geq N$, we have $x^n \in B_x(\delta)$ and thus for any $y^n \in \phi(x^n)$ also $y^n \in B_{\phi(x)}(\epsilon)$ which by definition gives $\inf_{y\in \phi(x)} d_\mathcal{Y}(y^n, y) < \epsilon$. As this holds for any $y^n \in \phi(x^n)$ and $\epsilon > 0$ this shows $\sup_{y^n \in \phi(x^n)}\inf_{y \in \phi(x)} d_\mathcal{Y}(y^n, y) \rightarrow 0$.\par
    For the reverse implication we show the contrapositive. Assume that $\phi$ is not metrically upper hemicontinuous. Then, there exists $\epsilon > 0$ so that for any $\delta > 0$ the set $\phi(B_x(\delta)) \setminus B_{\phi(x)}(\epsilon)$ is non empty. Thus, for any $n \geq 1$, we can choose $y^n \in \phi(B_x(\frac{1}{n})) \setminus B_{\phi(x)}(\epsilon)$. By construction, for any $n$ there is $x^n \in B_x(\frac{1}{n})$ with $y^n \in \phi(x^n)$ and in particular, the $x^n$ converge towards $x$. Further, we have $\inf_{y \in \phi(x)} d_\mathcal{Y}(y^n, y) \geq \epsilon$ and hence $\sup_{y^n \in \phi(x^n)} \inf_{y\in \phi(x)} d_\mathcal{Y}(y^n, y) \geq \epsilon > 0$.
\end{proof}
Let us prove two additional properties that will be useful later.
\begin{lemma}\label{lem: uhc comp}
    Let $\phi: (\mathcal{X}, d_\mathcal{X}) \twoheadrightarrow (\mathcal{Y}, d_\mathcal{Y})$ metrically upper hemicontinuous and $\psi:(\mathcal{Y}, d_\mathcal{Y}) \rightarrow (Z, d_Z)$ be uniformly continuous. Then, the composition $\psi\circ \phi: \mathcal{X} \twoheadrightarrow Z, x \mapsto \lbrace \psi(y) \vert y \in \phi(x) \rbrace$ is a metrically upper hemicontinuous set-valued map.
\end{lemma}
\begin{proof}
Consider any $x^n \rightarrow x$ in $\mathcal{X}$. Let $\epsilon >0$ be arbitrary. Since $\psi$ is uniformly continuous, there exits $\delta > 0$ such that for any $y, y' \in \mathcal{Y}$ with $d_\mathcal{Y}(y, y') < \delta$, we have $d_Z(\psi(y), \psi(y')) < \frac{\epsilon}{2}$. Now by metric upper hemicontinuity, we can find $N \geq 1$ such that for any $n \geq N$, we have 
$$\sup_{y^n \in \phi(x^n)} \inf_{y \in \phi(x)} d_\mathcal{Y}(y^n, y) < \delta.$$
Let us consider an $z^n \in \psi(\phi(x^n))$. For such, there exists $y^n \in  \phi(x^n)$ with $z^n = \psi(y^n)$. By our choice of $N$, for $n \geq N$, there must exists $\hat{y}^n \in \phi(x)$ with $d_\mathcal{Y}(y^n, y) < \delta$. By uniform continuity, this implies $d_Z(z^n, \psi(\hat{y}^n)) = d_Z(\psi(y^n), \psi(\hat{y}^n)) < \frac{\epsilon}{2}$. In particular, since $\psi(\hat{y}^n) \in \psi(\phi(x))$, this implies
$$\sup_{z^n \in \psi(\phi(x^n))} \inf_{z \in \psi(\phi(x))} d_Z(z^n, z) \leq \frac{\epsilon}{2} < \epsilon.$$
As $\epsilon$ was arbitrary, this proves metric upper hemicontinuity.
\end{proof}
Recall that unless specified otherwise, we metrize the product of metric spaces by the maximum metric.
\begin{lemma}\label{lem: uhc prod}
    Let $\phi: (\mathcal{X}, d_\mathcal{X}) \twoheadrightarrow (\mathcal{Y}, d_\mathcal{Y})$ and $\phi': (\mathcal{X}, d_\mathcal{X}) \twoheadrightarrow (\mathcal{Y}', d_{\mathcal{Y}'})$ be two metrically upper hemicontinuous maps. Then, $(\phi, \phi'): (\mathcal{X}, d_\mathcal{X}) \mapsto (\mathcal{Y} \times \mathcal{Y}', d_{\mathcal{Y}, \mathcal{Y}'}), x \mapsto \phi(x) \times \phi'(x)$ is again metrically upper hemicontinuous.
\end{lemma}
\begin{proof}
    Consider any $x^n$ converging towards some $x$ in $\mathcal{X}$. For any $\epsilon > 0$, there must exist a common $N$ such that $\sup_{y^n \in \phi(x^n)} \inf_{y \in \phi(x)} d_\mathcal{Y}(y^n, y) < \epsilon$ and $\sup_{y^n \in \phi'(x^n)} \inf_{y \in \phi'(x)} d_\mathcal{Y}(y^n, y) < \epsilon$. As we are working with the maximum metric on $\mathcal{Y} \times \mathcal{Y}'$, this shows metric upper hemicontinuity.
\end{proof}

\bibliographystyle{plain}
\bibliography{refs.bib}
\end{document}